\documentclass[11pt]{amsart}
\usepackage[T1]{fontenc}
\usepackage{bbm}
\usepackage{amsmath, amssymb, amsfonts, amsthm, verbatim, dsfont, graphicx}
\usepackage[hidelinks]{hyperref}
\usepackage{color}
\usepackage{mathtools}
\mathtoolsset{showonlyrefs}

\makeatletter
\renewcommand\tableofcontents{%
    \section*{\huge{Table of Contents}
        \@mkboth{%
           \MakeUppercase\contentsname}{\MakeUppercase\contentsname}}% DELETED
    \@starttoc{toc}%
    } 
\makeatother

\newtheorem{thm}{Theorem}
\numberwithin{thm}{section}
\newtheorem{lem}[thm]{Lemma}
\newtheorem{prop}[thm]{Proposition}

\newtheorem{cor}[thm]{Corollary}
\newtheorem{defn}[thm]{Definition}

\newtheorem{rmk}{Remark}

\numberwithin{rmk}{section}
\numberwithin{thm}{section}
\numberwithin{equation}{section}
\numberwithin{figure}{section}

\newtheorem{innercustomthm}{Theorem}
\newenvironment{customthm}[1]
  {\renewcommand\theinnercustomthm{#1}\innercustomthm}
  {\endinnercustomthm}

\newcommand{\cA}{\mathcal{A}}
\newcommand{\cB}{\mathcal{B}}

\newcommand{\cE}{\mathcal{E}}

\newcommand{\cH}{\mathcal{H}}
\newcommand{\cI}{\mathcal{I}}

\newcommand{\cM}{\mathcal{M}}

\newcommand{\cO}{\mathcal{O}}

\newcommand{\cS}{\mathcal{S}}

\newcommand{\cX}{\mathcal{X}}
\newcommand{\cY}{\mathcal{Y}}

\newcommand{\bC}{\mathbb{C}}

\newcommand{\bN}{\mathbb{N}}

\newcommand{\bP}{\mathbb{P}}

\newcommand{\bR}{\mathbb{R}}

\newcommand{\bT}{\mathbb{T}}

\newcommand{\bZ}{\mathbb{Z}}

\author[D. Mendelson]{Dana Mendelson}
\address{Department of Mathematics \\ Massachusetts Institute of Technology}
\email{dana@math.mit.edu}

\thanks{The author was supported in part by the U.S. National Science Foundation grants DMS-1068815 and DMS-1362509 and by the NSERC Postgraduate Scholarships Program}

\begin{document}

\title[Symplectic non-squeezing for the NLKG]{Symplectic non-squeezing for the cubic nonlinear Klein-Gordon equation on $\bT^3$}

\begin{abstract}
We consider the periodic defocusing cubic nonlinear Klein-Gordon equation in three dimensions in the symplectic phase space $H^{\frac{1}{2}}(\bT^3) \times H^{-\frac{1}{2}}(\bT^3)$. This space is at the critical regularity for this equation, and in this setting there is no global well-posedness nor any uniform control on the local time of existence for arbitrary initial data. We prove a local-in-time non-squeezing result and a conditional global-in-time result which states that uniform bounds on the Strichartz norms of solutions imply global-in-time non-squeezing. As an application of the conditional result, we conclude non-squeezing for certain subsets of the phase space. The proofs rely on several approximation results for the flow, which we obtain using a combination of probabilistic and deterministic techniques.
\end{abstract}

\maketitle

\section{Introduction}

We consider the behavior of solutions to the Cauchy problem for the cubic defocusing nonlinear Klein-Gordon equation
\begin{equation}
\label{equ:cubic_nlkg}
\left\{\begin{split}
&u_{tt} - \Delta u + u + u^3 = 0, \quad u: \bR \times \bT^3 \to \bR  \\
&(u, \partial_tu)\big|_{t=0} = (u_0, u_1) \in H^{\frac{1}{2}}(\bT^3) \times H^{-\frac{1}{2}}(\bT^3) =: \cH^{1/2}(\bT^3),
\end{split} \right.
\end{equation}
where $H^{\frac{1}{2}}(\bT^3)$ is the usual inhomogeneous Sobolev space. Although there is no scaling symmetry for power-type nonlinear Klein-Gordon equations, the nonlinear wave equation
\begin{equation}
u_{tt} - \Delta u + u^p = 0
\end{equation}
has a scaling symmetry
\[
u(t,x) \mapsto u^\lambda(t,x) := \lambda^{2/(p-1)} u(\lambda t, \lambda x)
\]
and the scale invariant critical space for the nonlinear wave equation, which corresponds to the homogeneous Sobolev space at regularity $s_c := \frac{d}{2} - \frac{2}{p-1}$, can still be regarded as the critical space for \eqref{equ:cubic_nlkg}. This is because ill-posedness and blowup are normally associated with high frequencies or short time scales, and so the nonlinear term dominates the mass term. Note that when $d= 3$, and $p=3$, we have $s_c = \frac{1}{2}$, hence we are interested in the behavior of solutions of \eqref{equ:cubic_nlkg} in the critical space. 

\medskip
Local strong solutions to \eqref{equ:cubic_nlkg} can be constructed by adapting the arguments from \cite{Lindblad_Sogge} to the compact setting. Due to the critical nature of this problem, however, the local time of existence for solutions depends not only on the norm of the initial data but also on its profile. Moreover, as the critical regularity for this equation does not correspond to a conserved quantity, global well-posedness for this equation remains open. The best known results on Euclidean space are for subcritical regularities $s \geq 3/4$, which was proved by Miao, Zhang and Fang \cite{MZF} by working in Besov spaces and adapting Bourgain's high-low argument \cite{B98} and arguments from \cite{KPV}. Once again, these results can be adapted to the periodic setting. We recall that the $L_{t,x}^4$ Strichartz norm controls the global well-posedness for \eqref{equ:cubic_nlkg} and we have the standard finite time blow-up criterion, namely if $T_*$ denotes the maximal time of existence for a solution $u$ to \eqref{equ:cubic_nlkg} then
\begin{align}
\label{equ:blow_up}
T_* < \infty \quad \Longrightarrow \quad \|u\|_{L_{t,x}^4([0, T_*) \times \bT^3)} = + \infty.
\end{align}

\medskip
We will study the long-time qualitative behavior of solutions to \eqref{equ:cubic_nlkg} by investigating symplectic non-squeezing for the flow of this equation. The study of infinite dimensional symplectic capacities and non-squeezing for nonlinear Hamiltonian PDEs was initiated by Kuksin in \cite{Kuk}. There, he extended the definition of the Hofer-Zehnder capacity to infinite dimensional phase spaces and proved the invariance of this capacity under the flow of certain Hamiltonian equations with flow maps of the form
\begin{align}
\label{equ:compact}
\Phi(t) = \textup{linear operator} + \textup{compact smooth operator}.
\end{align}
This infinite dimensional symplectic capacity inherits the finite dimensional normalization
\begin{align}
\label{equ:normalize}
\textup{cap}(\textbf{B}_r(u_*)) = \textup{cap}(\textup{\textbf{C}}_{r}(z;k_0)) = \pi r^2,
\end{align}
where
\begin{align*}
\textbf{B}_r(u_*) := \left\{ u \in \cH^{1/2}(\bT^3) \,: \, \|u - u_*\|_{\cH^{1/2}} \leq r \right\},
\end{align*}
and for $z = (z_0, z_1) \in \bC$ and $k \in \bZ^3$,
\begin{align}
\label{equ:cylinder}
\textup{\textbf{C}}_{r}(z; k_0) := \left\{ (u_0, u_1) \in \cH^{1/2}(\bT^3) \,: \, \langle k \rangle | \widehat{u}_0(k) - z_0|^2  + \langle k \rangle^{-1} | \widehat{u}_1(k) - z_1|^2  \leq r^2 \right\}
\end{align}
is an infinite dimensional cylinder in $\cH^{1/2}(\bT^3)$. We will always work with real-valued functions and the $\widehat{u_i}(k)$ are taken to be real-valued Fourier coefficients. The proof of the normalization \eqref{equ:normalize} in infinite dimensions is an adaptation of the original proof by Hofer and Zehnder which can be found in \cite{HZ}, see \cite{Kuk} for details of the infinite dimensional argument. Consequently, if a flow map $\Phi$ preserves capacities, one can conclude that  squeezing is impossible, namely  
\[
\Phi(t) (\textbf{B}_R(u_*)) \not \subseteq \textup{\textbf{C}}_{r}(z; k)  \qquad \textup{if } R < r.
\]
Several examples of nonlinear Klein-Gordon equations with weak nonlinearities can readily be shown to be of the form \eqref{equ:compact}, see \cite{Kuk}. Symplectic non-squeezing was later proved for certain subcritical nonlinear Klein-Gordon equations in \cite{B95} using Kuksin's framework, see also \cite{Roumegoux}. Bourgain later extended these results to the cubic NLS in dimension one in \cite{B94C}, where the flow is not a compact perturbation of the linear flow. There, the argument follows from approximating the full equation by a finite dimensional flow and applying Gromov's finite dimensional non-squeezing result to this approximate flow. Symplectic non-squeezing was also proven for the KdV \cite{CKSTT05}. In this situation, there is a lack of smoothing estimates in the symplectic space which would allow the infinite dimensional KdV flow to be easily approximated by a finite-dimensional Hamiltonian flow. To resolve this issue, the authors of \cite{CKSTT05} invert the Miura transform to work on the level of the modified KdV equation, for which stronger estimates can be established.

\subsection{Statement of main results}
As we will see in Section \ref{sec:sym_hilb}, the symplectic phase space for any nonlinear Klein-Gordon equation is $\cH^{1/2}(\bT^d)$ for any dimension $d \geq 1$. For the cubic nonlinear Klein-Gordon equation in dimension three, the symplectic phase space is at the critical regularity, which presents some serious obstructions to using simple modifications of the existing arguments. Kuksin's approach requires some additional regularity in the compactness estimates. In light of ill-posedness results below the critical space, for instance \cite{CCT}, \cite{L} or \cite{IMM} adapted to \eqref{equ:cubic_nlkg}, there is no way to gain the additional regularity needed. Bourgain's argument in \cite{B94C} uses an iteration scheme in which one needs uniform control over time-steps of the iteration. Once again, this seems to be a genuine obstruction to applying this argument at the critical regularity, especially since this regularity is not controled by a conserved quantity. Finally, the arguments of \cite{CKSTT05} depend heavily on the structure of the KdV equation. Additionally, the global well-posedness of \eqref{equ:cubic_nlkg} is not know and there is no uniform control on the local time of existence.

\medskip
Ultimately, however, we are able to circumvent these difficulties, using a combination of probabilistic and deterministic techniques, which we combine to obtain several \textit{deterministic} non-squeezing results. Our first result is a local-in-time non-squeezing theorem.

\begin{thm}
\label{thm:non-squeezing}
Let $\Phi$ denote the flow of the cubic nonlinear Klein-Gordon equation \eqref{equ:cubic_nlkg}. Fix $R > 0$, $k_0 \in \bZ^3$, $z \in \bC$, and $u_* \in \cH^{1/2}(\bT^3)$. For all $0 < \eta < R$, there exists $N \equiv N(\eta, u_*, R, k_0)$ and $\sigma \equiv \sigma(\eta, N, u_*) > 0$ such that for all $0 \leq t \leq \sigma$,
\begin{equation}
\label{equ:nonsqueeze}
\Phi(t)\bigl(\Pi_{N}\textbf{B}_R(u_*)  \bigr) \not \subseteq \textup{\textbf{C}}_{r}(z; k_0) \qquad \textup{for } r < R - \eta.
\end{equation}
\end{thm}

\noindent  In the statement of this theorem, $\Pi_N$ is a sharp projection onto frequencies $|k| \leq N$.

\begin{rmk}
\label{rmk:nonsqueeze2}
The parameter $\eta$ which appears in Theorem \ref{thm:non-squeezing} corresponds to the control we can obtain over the radius of the cylinder. If we demand better control over the radius, this theorem only holds for shorter time scales. See also Remark \ref{rmk:ustar_dep} for a discussion of the dependence of the various constants on $u_*$.
\end{rmk}

\begin{rmk}
\label{rmk:nonsqueeze1}
To prove this theorem, we combine a probabilistic approximation argument with deterministic stability theory. As we have no control on the local time of existence in the critical space, \textit{a priori} we cannot ensure the flow map $\Phi(t)$ is well-defined for any positive time. We fix the projection in \eqref{equ:nonsqueeze} at frequency $N \in \bN$ in order to gain enough control to define the flow map locally in time.
\end{rmk}

\begin{rmk}
We pursue a probabilistic approach to Theorem \ref{thm:non-squeezing} for several reasons. First, we obtain intermediate probabilistic results which demonstrate the that flow for the cubic nonlinear Klein-Gordon equation satisfies Kuksin's compactness criterion in a probabilistic sense, see \eqref{equ:nonlin_smoothing}. Second, we suspect that probabilistic arguments may be combined with an alternative definition of an infinite dimensional symplectic capacity to obtain a more direct proof of infinite dimensional sympletic non-squeezing in this setting, see Remark \ref{rmk:capacity1} and Remark \ref{rmk:whynotsmoother}.
\end{rmk}

\medskip
In order to state our global-in-time results, we need to introduce the following nonlinear Klein-Gordon equation with truncated nonlinearity
\begin{equation}
\label{equ:cubic_nlkg_trun}
\left\{\begin{split}
&(u_N)_{tt} - \Delta u_N + u_N + P_N(P_N u_N)^3 = 0, \quad u: \bR \times \bT^3 \to \bR  \\
&(u_N, \partial_tu_N)\big|_{t=0} = (u_0, u_1) \in \cH^{1/2}(\bT^3),
\end{split} \right.
\end{equation}
where $P_N = P_{\leq N}$ denotes the smooth projection operator defined in \eqref{equ:trunc_op}.
We obtain the following global-in-time non-squeezing result.

\begin{thm}
\label{thm:non-squeezing_cond}
Fix $R, T > 0$, $k_0 \in \bZ^3$, $z \in \bC$, and $u_* \in \cH^{1/2}(\bT^3)$. Suppose there exists some $K > 0$ such that for all $(u_0, u_1) \in \textbf{B}_R(u_*)$, the corresponding solutions $u$  to \eqref{equ:cubic_nlkg} and $u_N$ to \eqref{equ:cubic_nlkg_trun} exist on $[0,T]$ and satisfy
\begin{align}
\label{equ:unif_l4}
\|u \|_{L_{t,x}^{4}([0,T] \times \bT^3)} \,\, + \,\,  \sup_N \|P_N u_N\|_{L_{t,x}^{4}([0,T] \times \bT^3)} \leq K.
\end{align}
Let $\Phi$ denote the flow of the cubic nonlinear Klein-Gordon equation \eqref{equ:cubic_nlkg}. Then 
\begin{equation}
\Phi(T)\bigl(\textbf{B}_R(u_*)  \bigr) \not \subseteq \textup{\textbf{C}}_{r}(z; k_0) \qquad \textup{for } r < R.
\end{equation}
In particular, for a fixed choice of $T$, if $\textbf{B}_R(u_*) \subset \textbf{B}_{\rho_0}$ for some sufficiently small $\rho_0(T) > 0$, then non-squeezing holds without any additional assumptions on the initial data.
\end{thm}

\begin{rmk}
\label{rmk:cond}
Rephrased, Theorem \ref{thm:non-squeezing_cond} says that if one can prove existence and uniform Strichartz bounds for solutions to both \eqref{equ:cubic_nlkg} and \eqref{equ:cubic_nlkg_trun} for initial data in bounded subsets, then one obtains the full, deterministic statement of non-squeezing for this equation. We believe that assumption \eqref{equ:unif_l4} is a natural criterion for a long-time non-squeezing result in the critical setting and although the global well-posedness of this equation is currently out of reach, Theorem \ref{thm:non-squeezing_cond} reduces an a priori stronger dynamical statement about the flow to one about long-time existence and uniform Strichartz bounds. 
\end{rmk}

\begin{rmk}
\label{rmk:small_data}
For a fixed $T > 0$, we are able to prove that assumption \eqref{equ:unif_l4} holds automatically for initial data which is sufficiently small, depending on this choice of time, by standard arguments, see also Lemma \ref{lem:trunc_bds_small}. This fact allows us to conclude the small-data component of Theorem \ref{thm:non-squeezing_cond}. Unlike the Euclidean setting, we do not obtain an automatic small data global theory for the critical problem in the periodic setting since one must localize the Strichartz estimates in time.
\end{rmk}

\begin{rmk}
We may also combine Theorem \ref{thm:non-squeezing_cond} with the almost sure global well-posedness and the deterministic critical stability theory for the cubic nonlinear Klein-Gordon equation \eqref{equ:cubic_nlkg} to obtain a probabilistic non-squeezing result on certain open subsets of the phase space with large measure that is, measure which is greater than $1-Ce^{-c/\delta}$ for some small $\delta > 0$. See Theorem \ref{thm:weaknon-squeezing} for a precise statement of this result, as well as Remark \eqref{rmk:capacity} for some discussion.
\end{rmk}

\subsection{Overview of Proof}
\label{sub:overview}
\subsubsection{Almost sure global well-posedness}
To prove Theorem \ref{thm:non-squeezing}, we rely on an adaptation of the almost sure global well-posedness result from \cite{BT4}. This enables us to work on a set of full measure, $\Sigma$, with respect to a suitable randomization of the initial data, on which the nonlinear Klein-Gordon equation is globally well-posed. We will show that for a certain nested sequence of subsets $\Sigma_\lambda \subset \Sigma$, the flow of this equation can be seen as a compact perturbation of a linear flow in the sense used by Kuksin \eqref{equ:compact}.

\medskip
We will now describe the randomization procedure for the initial data. Let $\{ (h_k, l_k) \}_{k \in \bZ^3}$  be a sequence of zero-mean, complex-valued Gaussian random variables on a probability space $(\Omega, {\mathcal A}, \bP)$ with the symmetry condition $h_{-k} = \overline{h_k}$ for all $k \in \bZ^3$ and similarly for the $l_k$. We assume $\{h_0, \textup{Re}(h_k), \textup{Im}(h_k)\}_{k \in \cI}$ are independent, zero-mean, real-valued Gaussian random variables, where $\cI$ is such that we have a \textit{disjoint} union $\bZ^3 = \cI \cup (-\cI) \cup \{0\}$, and similarly for the $l_k$. This set-up ensures that the randomization of real-valued functions is real-valued. 

\medskip
Fix $(f_0, f_1) \in \cH^s(\bT^3)$, and define a randomization map $\Omega \times \cH^s \to \cH^s$ by
 \begin{equation} \label{equ:bighsrandomization_chap5} 
\bigl(\omega, (f_0, f_1)\bigr) \longmapsto (f_0^{\omega}, f_1^{\omega}) := \biggl( \sum_{k \in \bZ^3} h_k(\omega) \widehat{\phi}_0(k)e^{i k \cdot x}, \sum_{k \in \bZ^3} l_k(\omega) \widehat{\phi}_1(k)e^{i k \cdot x} \biggr).
 \end{equation}
We could similarly take non-Gaussian random variables which satisfy suitable boundedness conditions on their distributions. For any $(f_0, f_1) \in \cH^s$, the map \eqref{equ:bighsrandomization_chap5} induces a probability measure on $\cH^s$, given by
\[
\mu_{(f_0, f_1)} (A) = \bP\bigl( \omega \in \Omega : (f_0^{\omega}, f_1^{\omega}) \in A\bigr).
\]
We denote by $\cM^s$ the set of such measures:
\[
\cM^s := \bigl\{ \mu_{(f_0, f_1)} \,:\, (f_0, f_1) \in \cH^s \bigr\}.
\]
\begin{rmk}
\label{rmk:random_props}
The support of any $\mu \in \cM^s$ is contained in $\cH^s$ for all $s \in \bR$. Furthermore, if for some $s_1 >  s$ we have that $(f_0, f_1) \not \in \cH^{s_1}$ then the induced measure satisfies $\mu_{(f_0, f_1)}(\cH^{s_1}) = 0$. In other words, this randomization procedure does not regularize at the level of Sobolev spaces in the sense that almost surely, the randomization of a given function is no more regular than it was orginally. Moreover, if all the Fourier coefficients of $(f_0, f_1)$ are nonzero, then the support of the corresponding measure $\mu_{(f_0, f_1)}$ is all of $\cH^s$, that is, $\mu_{(f_0, f_1)}$ charges every open set in $\cH^s$ with positive measure. As a consequence, for such a measure, sets of full $\mu$ measure are dense. See \cite{BT1} for details.
\end{rmk}

\medskip
The arguments used to prove the almost sure global well-posedness of the defocusing cubic nonlinear wave equation by Burq and Tzvetkov \cite[Theorem 2]{BT4},  apply to the defocusing cubic nonlinear Klein-Gordon equation, with the slight modification that one must consider the inhomogeneous energy functional
\begin{align}
\label{equ:inhomog_energy}
\cE(w) = \frac{1}{2} \int |\nabla w|^2 + w^2 + (w_t)^2 + \frac{1}{2} w^4.
\end{align}
We note that since the Hamiltonian for the nonlinear Klein-Gordon controls the $L_x^2$ norm of solutions, we no longer need the projection away from constants which appears in \cite{BT4}. We denote by $S(t)$ the free evolution for \eqref{equ:cubic_nlkg}, given by
\begin{align}
\label{equ:free}
S(t)(u_0, u_1) = \cos(t \langle \nabla \rangle) u_0 + \frac{\sin(t \langle \nabla \rangle)}{\langle \nabla \rangle} u_1,
\end{align}
and we state \cite[Theorem 2]{BT4} adapted to our situation.
\begin{thm}
\label{thm:bt_global}
Let $M = \bT^3$ with the flat metric and fix $\mu \in \cM^s$, $0 < s < 1$. Then there exists a full $\mu$ measure set $\Sigma \subset \cH^s(\bT^3)$ such that for every $(u_0, u_1) \in \Sigma$, there exists a unique global solution $u$ of the nonlinear Klein-Gordon equation
\begin{equation}
\label{equ:cubic_nlkg_thm}
\left\{\begin{split}
&u_{tt} - \Delta u + u + u^3 = 0, \quad u: \bR \times \bT^3 \to \bR  \\
&(u, \partial_tu) \Big|_{t=0} = (u_0, u_1)
\end{split} \right.
\end{equation}
satisfying
\[
(u(t),u_t(t)) \in (S(t)(u_0, u_1), \partial_t S(t) (u_0, u_1)) + C(\bR_t ; \cH^1_x(\bT^3)).
\]
Furthermore, if we denote by
\[
\Phi(t)(u_0,u_1) \equiv (u(t), \partial_t u(t))
\]
the flow thus defined, the set $\Sigma$ is invariant under the flow $\Phi(t)$:
\[
\Phi(t)(\Sigma) = \Sigma \quad \forall t \in \bR.
\]
Finally, for any $\varepsilon > 0$, there exist $C, c, \theta > 0$ such that for every $(u_0, u_1) \in \Sigma$ there exists $M = M(u_0, u_1) > 0$ such that the global solution of \eqref{equ:cubic_nlkg_thm} given by
\[
u(t) = S(t) (u_0, u_1) + w(t)
\]
satisfies
\begin{equation}
\label{equ:sob_bds}
\|(w(t), \partial_t w(t))\|_{\cH^1} \leq  C(M + |t|)^{\frac{1-s}{s} + \varepsilon},
\end{equation}
and furthermore, for $M$ as in \eqref{equ:sob_bds} and each $\lambda > 0$,
\[
\mu\bigl((u_0, u_1) \in \Sigma \,:\, M > \lambda \bigr) \leq C e^{- c\lambda^\theta}.
\]
\end{thm}

\medskip
\noindent Let us introduce precisely the subset $\Sigma$ of full measure we will work with. Let $0 < \gamma < \frac{1}{2}$ to be fixed later and define
\begin{equation}
\label{equ:sigma_def}
\begin{split}
\Theta_1 &:= \bigl\{(u_0, u_1) \in \cH^{1/2} \,:\, \|S(t)(1 - \Delta)^{\gamma / 2} (u_0, u_1) \|^6_{L_x^6(\bT^3)}  \in L_{\text{loc}}^1(\bR_t)  \bigr\}, \\
\Theta_2 &:= \bigl\{(u_0, u_1) \in \cH^{1/2} \,:\ \|S(t)(u_0, u_1) \|_{L_x^\infty(\bT^3)} \in L_{\text{loc}}^1(\bR_t)  \bigr\}.
\end{split}
\end{equation}
Set $\Theta := \Theta_1 \cap \Theta_2$ and let $\Sigma = \Theta + \cH^1$. Although the set specified in \cite[Theorem 2]{BT4} imposes an $L^3_{t, loc} L^6_x$ condition on the free evolution, it will be more convenient to work with the above definition and it is clear that the conditions in \eqref{equ:sigma_def} are stronger. This choice will not change any of the arguments from \cite{BT4}, and as we will see, $\Sigma$ also has full measure with respect to any $\mu \in \cM^s$. By Remark \ref{rmk:random_props}, $\Sigma$ is not comprised of initial data which are smoother at the level of Sobolev spaces. We will work on the nested subsets $\Sigma_\lambda \subset \Sigma$, which we define in a following section. These subsets have the property that their union is all of $\Sigma$, and we prove in Proposition \ref{prop:lambda_props} that there exists $C, c, \theta > 0$ so that for any $\lambda > 0$, they satisfy
\[
\mu(\Sigma_\lambda) \geq 1 - Ce^{-c\lambda^\theta}.
\]
Let $\widetilde{\Phi}$ denote the nonlinear component of the flow map for the cubic nonlinear Klein-Gordon equation, given by
 \begin{align}
\label{equ:nonlin_com}
\widetilde{\Phi}(t) (u_0,u_1) := \Phi(t) (u_0,u_1) - S(t) (u_0,u_1).
\end{align}
On these subsets, we are able to prove a probabilistic version of the criteria needed for Kuksin's argument from \cite{Kuk}, namely we prove bounds of the form
\begin{align}
\label{equ:nonlin_smoothing}
\|\widetilde{\Phi}(t) (u_0, u_1)\|_{L_t^\infty\cH_x^{s_2}([0,T] \times \bT^3)} \lesssim \|(u_0, u_1)\|_{\cH_x^{s_1}}, \qquad (u_0, u_1) \in \Sigma_\lambda
\end{align}
for some $s_1 < \frac{1}{2} < s_2$.

\begin{rmk}
In \cite{B95}, Bourgain proves the analogue of \eqref{equ:nonlin_smoothing} for subcritical nonlinear Klein-Gordon equations via estimates in local-in-time $X^{s,b}$ spaces, see \eqref{equ:xsb_def_chap3}. The reason Bourgain's estimates fail at the critical regularity is because Strichartz estimates are not available at regularities $s_1 < \frac{1}{2}$, which one would need in order to obtain the smoothing bound \eqref{equ:nonlin_smoothing}. Generally speaking, $X^{s,b}$ spaces are ill-suited for critical problems, resulting in logarithmic divergences in the nonlinear estimates, and problems due to failure of the endpoint Sobolev embedding. Nonetheless, we choose to prove the probabilistic convergence argument in these spaces, since they are simpler to work with than the $U^p$ and $V^p$ spaces, and we believe they exhibit the key aspects of the probabilistic argument more clearly. Ultimately the $U^p$ and $V^p$ spaces are indeed more suitable for the critical problem, however, and we require them to prove the conditional approximation result in Section \ref{sec:approx}.
\end{rmk}

\subsubsection{Probabilistic approximation of the flow map}
Once we have established the probabilistic bounds on the nonlinear component of the flow map, our argument has several key components. The first is an approximation estimate for the flow map on the subsets $\Sigma_\lambda \subset \Sigma$. We consider \eqref{equ:cubic_nlkg_trun} and we observe that this equation is defined on the whole space and it decouples into a nonlinear evolution on low-frequencies and a linear flow on high frequencies. This equation was used in \cite{BT5} and we will show in Section \ref{sec:well-posedness} that this flow is a symplectomorphism when restricted to the finite dimensional subspaces $\Pi_N \cH^{1/2}$. We will show that this flow approximates the flow of the full nonlinear Klein-Gordon equation well and consequently, we are able to show that the full equation preserves infinite dimensional capacities.

\begin{prop}
\label{prop:conv}
Let $\Phi$ and $\Phi_N$ denote the flows of the cubic nonlinear Klein-Gordon equation with full \eqref{equ:cubic_nlkg} and truncated nonlinearities \eqref{equ:cubic_nlkg_trun}, respectively. Then for any $T > 0$ and for every $(u_0, u_1) \in \Sigma_{\lambda} \cap \textbf{B}_R $,
\[
\sup_{t \in [0, T]} \| \Phi(t)(u_0, u_1) -  \Phi_N(t) (u_0, u_1)\|_{\cH_x^{1/2}(\bT^3)} \leq C(\lambda, T, R) \, \varepsilon_1(N)
\]
with $\varepsilon_1(N) \to 0 $ as $N \to \infty$. 
\end{prop}
\noindent We remark that this is a global in time approximation result in the critical space, with no restriction on the size of the initial data.

\begin{rmk}
The dependence of the constant on $R$ in Proposition \ref{prop:conv} is somewhat artificial, and can be removed by proving the statements from Section \ref{sec:compactness} with bounds in terms of the nonlinear components of the solutions. Since both $\Phi$ and $\Phi_N$ have the same free evolution, these bounds would suffice to prove Proposition \ref{prop:conv}.
\end{rmk}

An important ingredient in the proof of Theorem \ref{thm:non-squeezing} is the following theorem which states that locally in time, if one restricts to initial data supported on finitely many frequencies, this approximation still holds uniformly.

\begin{thm}
\label{thm:trunc_approx}
Let $\Phi$ denote the flow of the cubic nonlinear Klein-Gordon equation \eqref{equ:cubic_nlkg} and $\Phi_N$ the flow of \eqref{equ:cubic_nlkg_trun}. Fix $R > 0$, $u_* \in \cH^{1/2}$ and $N', N \in \bN$ with $N'$ sufficiently large, depending on $u_*$. Then there exists a constant $\varepsilon_0(u_*, R) > 0$ such that for all $0 < \varepsilon \leq \varepsilon_0$, there exists $\sigma \equiv \sigma( R, \varepsilon, N')$ such that for any $(u_0, u_1) \in \textbf{B}_R(u_*)$,
\begin{align}
\label{equ:open_conv}
 \sup_{ t \in [0,\sigma]} \| \Phi(t) \Pi_{N'} (u_0, u_1) -  \Phi_N(t) \Pi_{N'} (u_0, u_1) \|_{\cH_x^{1/2}} \leq C(R, u_*) \left[ \, \varepsilon_1(N) + \varepsilon \,\right] \qquad
\end{align}
with $\lim_{N \to \infty} \varepsilon_1(N) = 0 $.
\end{thm}

\begin{rmk}
\label{rmk:ustar_dep}
By using probabilistic techniques, we obtain a uniform statement at the critical regularity for all initial data in $\textbf{B}_R(u_*)$. The dependence on $u_*$ in this theorem arises both due to the minimal value $\lambda > 0$ such that the intersection $\Sigma_\lambda \cap \textbf{B}_R(u_*)$ is non-empty, and also from of the frequency $N$ required so that
\[
\|\Pi_{\geq N} u_*\|_{\cH^{1/2}} < R.
\]
We suspect that the dependence of the constants on the various parameters is non-optimal and some improvement might be possible.
\end{rmk}

\begin{rmk}
The projection operators in \eqref{equ:open_conv} allow us to gain the necessary control via the stability theory in order to obtain convergence on all of $\Pi_{N'} \textbf{B}_R$. It is the combination of Proposition \ref{prop:conv} with the stability theory which enables us to get a deterministic non-squeezing theorem, regardless of the fact that Proposition \ref{prop:conv} is a probabilistic result.
\end{rmk}

\begin{rmk}
\label{rmk:capacity1}
Although Proposition \ref{prop:conv} is a global in time convergence for the flow maps for initial data in $\Sigma_\lambda$, these sets are not open and have a rather complicated structure. This poses serious problems for the proof of the non-squeezing theorem for several reasons, among which is the fact that it is not at all clear how to compute the capacities $\textup{cap}(\Sigma_\lambda \cap \textbf{B}_R)$, or even ensure that they are positive since the subsets $\Sigma_\lambda$ have empty interior. In order to upgrade the approximation to open subsets, one must localize in time or obtain a conditional result. We have no prediction at the moment for an estimate of this capacity using Kuksin's definition. Ultimately, it seems likely that an alternative definition of an infinite dimensional symplectic capacity might be a more suitable approach.
\end{rmk}

\begin{rmk}
\label{rmk:whynotsmoother}
A posteriori, one may observe that smoother initial data satisfies the assumptions needed for our probabilistic smoothing results, however one still needs to perform a similar analysis to that of Section \ref{sec:compactness} in order to obtain \eqref{equ:nonlin_smoothing}, see Remark \ref{rmk:no_sobolev}. This observation, however, does not seem to be promising from the point of view of future work on non-squeezing for \eqref{equ:cubic_nlkg} since there are results which indicate that one does in fact have squeezing for smoother initial data, see \cite{Kuk2}. Moreover, in any reasonable probabilistic set-up for studying \eqref{equ:cubic_nlkg} smoother initial data will have zero measure. Consequently if we wish to find a notion of symplectic capacity which is compatible with our probabilistic perspective, it is not sufficient to consider only smooth initial data.
\end{rmk}

\begin{rmk}
As an easy consequence of the type of arguments used to prove Theorem \ref{thm:trunc_approx}, we obtain well-posedness for long times on open subsets, as well as topological genericity for the set of initial data for which this equation is globally well-posed. Namely, we can write the set of initial data for which we can prove global well-posedness as the intersection of dense, open sets. 
\end{rmk}

\subsubsection{Non-squeezing via Gromov's Theorem}

As in \cite{B94C} and \cite{CKSTT05}, we will prove non-squeezing for the full equation by using an approximation theorem combined with Gromov's finite dimensional non-squeezing theorem. We quote Gromov's theorem here for the symplectic space $(\bR^{2N}, \omega_0)$, where $\omega_0$ is the standard symplectic form on $\bR^{2N}$, see \cite{HZ} for details and generalizations.

\begin{thm}[Gromov's non-squeezing theorem, \cite{Gromov}]
\label{thm:gromov}
Let $R$ and $r > 0$, $z \in \bC$, $0 \leq k_0 \leq N$, and $x_* \in \bR^{2N}$. Let $\phi$ be a symplectomorphism defined on $\textbf{B}_R(x_*)\subset \bR^{2N}$, then
\[
\phi(\textbf{B}_R(x_*)) \not \subset \textbf{C}_r(z; k_0)
\]
for any $r < R$.
\end{thm}

We would like to take $N \to \infty$ in Theorem \ref{thm:gromov}. To do so, we use Theorem \ref{thm:trunc_approx} to obtain the necessary uniform local control. 
 
\medskip
The proof of Theorem \ref{thm:non-squeezing} follows easily once we have proven Theorem \ref{thm:trunc_approx}. Indeed, fix parameters $R > 0$,  $k_0 \in \bZ^3$, $z = (z_0, z_0) \in \bC$, $u_* \in \cH^{1/2}$ and $0 < \eta < R$. Let $\varepsilon_0(R, u_*)$ be as in the statement of Theorem \ref{thm:trunc_approx}, and we fix $N > |k_0|$ sufficiently large and $0< \varepsilon < \varepsilon_0$ sufficiently small so that for $\varepsilon_1(N)$ and $C = C(R, u_*)$ as in \eqref{equ:open_conv} we have
\[
C \varepsilon_1(N) < \frac{\eta}{4} \qquad \textup{and} \qquad C  \varepsilon < \frac{\eta}{4}.
\]
Let $\sigma > 0$ be such that the conclusions of Theorem \ref{thm:trunc_approx} hold with $N' = 2N$ and $\varepsilon_0$ as above. Note that this choice ensures that $\Phi_N$ is a true symplectomorphism, which does not hold for $N' \ll N$, see Proposition \ref{prop:preserve_capacity_trunc}. Then for any $(u_0, u_1) \in \textbf{B}_{R}(u_*)$,
\begin{align}
\label{equ:approx}
\sup_{ t \in [0,\sigma]} \|\Phi(t) \Pi_{2N} (u_0, u_1) -  \Phi_N(t) \Pi_{2N} (u_0, u_1) \|_{\cH_x^{1/2}} < \frac{\eta}{2}.
\end{align}
Let $r < R - \eta$ and define
\[
\varepsilon_2 := \frac{R - r}{2} > \frac{\eta}{2}.
\]
The proof of Proposition \ref{prop:preserve_capacity_trunc} demonstrates that $\Phi_N$ is a symplectomorphism on $\Pi_{2N} \cH^{1/2}$ with the symplectic structure which is compatible with that of the full flow on $\cH^{1/2}$. Then we can find initial data $(v_0, v_1) \in \Pi_{2N} \textbf{B}_{R}(u_*) \subset \textbf{B}_{R_1}$ such that
\begin{align}
\label{equ:nonsqueeze_fin}
\left( \langle k_0 \rangle | \widehat{\Phi_N(\sigma) (v_0, v_1) }(k_0) - z_0 |^2 + \langle k_0 \rangle^{-1}| \widehat{\partial_t \Phi_N(\sigma) (v_0, v_1) }(k_0) - z_1 |^2 \right)^{1/2} >  r + \varepsilon_2. \qquad
\end{align}
By triangle inequality, \eqref{equ:approx} and \eqref{equ:nonsqueeze_fin} we obtain
\[
\left(\langle k_0 \rangle | \widehat{\Phi (\sigma) (v_0, v_1) }(k_0) - z_0 |^2 + \langle k_0 \rangle^{-1}| \widehat{\partial_t \Phi(\sigma) (v_0, v_1) }(k_0) - z_1 |^2 \right)^{1/2} > r + \varepsilon_2 - \frac{\eta}{2} > r , 
\]
which concludes the proof.

\begin{rmk}
As can be seen in this proof, the parameter $\eta$ provides a lower bound for the accuracy needed in the approximation. This allows us to ensure that we can obtain some sort of uniform lower bound on the time in order to prove non-squeezing.
\end{rmk}

\begin{rmk}
If we wish to use Theorem \ref{thm:non-squeezing} to conclude non-squeezing for the flow in the case that a global flow is defined, there are several ways to go about this. For instance, we can perform the same argument taking $N$ sufficiently large \textit{depending on the profile of $u_*$} so as to guarantee that for some choice of $\varepsilon > 0$
\[
\| u_* - \Pi_{2N}u_*\|_{\cH^{1/2}} < \varepsilon,
\]
then we can find initial data $(u_0, u_1) \in \Pi_{2N} \textbf{B}_{R-\varepsilon}(u_*)$, and hence
\[
\| u_* -  (u_0, u_1) \|_{\cH^{1/2}} \leq R.
\]
We can also proceed as follows: there exists some $R_1 >0$ such that $\Pi_{2N} \textbf{B}_{R}(u_*) \subset \textbf{B}_{R_1}(u_*)$, so we conclude that the flow does not squeeze some larger ball of initial data into the cylinder.
\end{rmk}

\subsection{Conditional non-squeezing}
\label{subsec:cond}
In order to prove our conditional result, we prove an approximation result which demonstrates that solutions of \eqref{equ:cubic_nlkg} are stable at low frequencies under high-frequency perturbations to the initial data. While this type of argument is often seen at subcritical regularities, there are additional difficulties at the critical regularity since the decay one usually needs for such arguments is not available in this setting. Nonetheless, one can manufacture some decay using refined bilinear Strichartz estimates which, roughly speaking, show that for $M \lesssim N$, dyadic frequencies,
\[
\|e^{i t \langle \nabla \rangle} \phi_N \, e^{i t \langle \nabla \rangle} \psi_M \|_{L_{t,x}^2} \lesssim M \|\phi_N\|_{L_{x}^2} \|\psi_M \|_{L_{x}^2} \qquad \textup{for }P_N \phi_N = \phi_N, \,\, P_M \psi_M = \psi_M.
\] 
For $M \sim N$ this estunate corresponds to one half derivative loss on each function, but if one can control the frequency separation between functions in certain key multilinear estimates, it is possible to regain some decay even in the critical setting. Using this philosphy, we obtain the following quantitative approximation result.

\begin{thm}
\label{thm:approx_conditional}
Fix $T, R > 0$ and $u_* \in \cH^{1/2}(\bT^3)$. Suppose there exists some $K > 0$ such that for all $(u_0, u_1) \in \textbf{B}_R(u_*)$, the corresponding solutions $u$  to \eqref{equ:cubic_nlkg} and $u_N$ to \eqref{equ:cubic_nlkg} exist on $[0,T]$ and satisfy
\[
\|u \|_{L_{t,x}^{4}([0,T] \times \bT^3)} \,\, + \,\,  \sup_N \|P_N u_N\|_{L_{t,x}^{4}([0,T] \times \bT^3)} \leq K.
\]
Let $\Phi$ and $\Phi_N$ denote the flows of the cubic nonlinear Klein-Gordon equation with full \eqref{equ:cubic_nlkg} and truncated \eqref{equ:cubic_nlkg_trun} nonlinearities, respectively. Then for all $\varepsilon > 0$ and any $N' \in \bN$,
\[
\sup_{ t \in [0,T]} \|P_{N'} \left(\Phi(t)(u_0, u_1) -  \Phi_N(t) (u_0, u_1) \right)\|_{\cH_x^{1/2}} < \varepsilon
\]
for $N = N ( N', \varepsilon,  R, T, K)$ sufficiently large. 
\end{thm}

\begin{rmk}
The same result holds in the Euclidean setting with almost no modifications to the arguments. We restrict the statements to the periodic case, however, for simplicity of exposition. Moreover, in the Euclidean setting we can eliminate the dependence on time in the implicit constants.
\end{rmk}

This result follows in fact from the proof of a stability result for the nonlinear Klein-Gordon equation which quantifies the principle that perturbing initial data in high-frequencies does not affect the low frequencies of the corresponding solutions by too much. We obtain the following statement.

\begin{thm}
\label{thm:local_compare}
Let $\Phi$ denote the flow of the cubic nonlinear Klein-Gordon equation \eqref{equ:cubic_nlkg}. Let $T > 0$ and $1 \leq N' \ll N_*$.  Let $(u_0, u_1), (\widetilde{u}_0, \widetilde{u}_1) \in \textbf{B}_R \subset \cH^{1/2}(\bT^3)$ be such that $P_{\leq N_*} (u_0, u_1) = P_{\leq N_*} (\widetilde{u}_0, \widetilde{u}_1)$, and
suppose there exists some $K > 0$ such that corresponding solutions $u$ and $\widetilde{u}$ to \eqref{equ:cubic_nlkg} exist on $[0,T]$ and satisfy
\[
\| u \|_{L_{t,x}^4([0,T] \times \bT^3)} + \| \widetilde{u} \|_{L_{t,x}^4([0,T] \times \bT^3)} \leq K.
\]
Then for sufficiently large $N_*$ depending on $R, T$ and $K$,
\[
\|P_{\leq N'}\bigl( \Phi(t) (u_0, u_1) - \Phi(t) (\widetilde{u}_0, \widetilde{u}_1) \bigr) \|_{L_t^\infty \cH_x^{1/2}([0,T) \times \bT^3)} \lesssim \left( \log\frac{N_*}{N'} \right)^{-\theta},
\]
with implicit constant depending on $R, T$, $K$.
\end{thm}

To prove Theorem \ref{thm:local_compare}, we demonstrate that under the above assumptions, the low frequency component of the solutions, defined by $u_{lo} = P_{\leq M} u$ for some $M \in \bN$, satisfies a perturbed cubic Klein-Gordon equation given by
\begin{align}
\label{equ:lov_err}
\square u_{lo} + u_{lo} = P_{\leq M} F(u_{lo}, u_{lo}, u_{lo}) + err.
\end{align}
where $err$ is an error term which we can control by the well-posedness theory. Such an expression allows us to prove Theorem \ref{thm:local_compare} by a stability type argument.  The ideas used in this proof are similar to those used in \cite{CKSTT05} however we mention again that in the current setting, the problem is at the critical regularity. This is also reflected in the fact that the constants in the above theorem depend on the Strichartz norm of the solutions and not just the Sobolev norm, as in \cite[Theorem 1.3]{CKSTT05}. The point of Theorem \ref{thm:approx_conditional} is that if one only needs to compare low frequencies, as is the case when proving non-squeezing, some decay can be regained even though we are in the critical setting.

\medskip
To prove these results, we work with the $U^p$ and $V^p$ function spaces.  These spaces have previously been used in the context of critical problems by Hadac, Herr and Koch \cite{HHK} for the KP-II equation, and by Herr, Tataru and Tzvetkov \cite{HTT} for the quintic nonlinear Schr\"odinger equation on $\bT^3$. See \cite{KTV} or \cite{HHK} and references therein for a more complete overview of these function spaces. We record the basic definitions and properties of these spaces in Appendix \ref{sec:up_vp}. 

\medskip
The key benefit of these function spaces is that they recover the endpoint embeddings which fail in $X^{s, \frac{1}{2}}$ but still enable us to exploit the same type of multilinear estimate machinery, and are thus are a suitable setting for critical problems. Unfortunately, however, these spaces miss out on a key property of the Strichartz spaces which is exploited in the stability theory, namely that if a space-time norm is finite on a time interval, then one can isolate subintervals where the norm is small. 

\medskip
To overcome this difficulty, we introduce a weaker norm \eqref{equ:zs_norm} which recovers the necessary properties in order to prove stability, but which still controls the well-posedness theory. This approach was used by Ionescu and Pausader in \cite{IP12} to prove global well-posedness of the energy critical nonlinear Schr\"odinger equation on $\bT^3$, although the norm we introduce is slightly different than the one used in that work.

\medskip
The proof of Theorem \ref{thm:non-squeezing_cond} follows from this approximation result using similar argument to those in the proof of Theorem \ref{thm:non-squeezing}, once one uses Theorem \ref{thm:local_compare} to restrict to a finite dimensional subspace of initial data. We refer the reader to the proof of Theorem 1.5 in \cite{CKSTT05} for a similar argument. Alternatively, to conclude non-squeezing from Theorem \ref{thm:approx_conditional}, we can use the fact that $\Phi_N$ preserves symplectic capacities, as we do for the proof of Theorem \ref{thm:weaknon-squeezing} in Section \ref{sec:non-squeeze_prob}

\subsection{Probabilistic non-squeezing}

\medskip
Finally, we present one last version of the non-squeezing theorem which is a direct application of Theorem \ref{thm:approx_conditional}. In contrast to Theorem \ref{thm:non-squeezing}, we do not need to consider the finite dimensional projection of the ball and the result we obtain is for large times. In exchange for this, however, we must restrict ourselves to initial data sufficiently close to elements of $\Sigma_\lambda$ and the control we obtain over the diameter of the cylinder is not as good.
\begin{thm}
\label{thm:weaknon-squeezing}
Fix $\mu \in \cM^{1/2}$ and let $\Phi$ denote the flow of the cubic nonlinear Klein-Gordon equation \eqref{equ:cubic_nlkg}. Let $R, T > 0$, $k_0 \in \bZ^3$, $z \in \bC$, and $u_* \in \cH^{1/2}(\bT^3)$. Then there exists $\theta > 0$ such that for every $\varepsilon > 0$ there exists an open set $U_\varepsilon$ with
\[
\mu(U_\varepsilon) \geq 1 - e^{ - 1/ \varepsilon^\theta}
\]
and such that
\begin{equation}
\Phi(T)\bigl(U_\varepsilon \cap \textup{\textbf{B}}_R(u_*) \bigr) \not \subseteq \textup{\textbf{C}}_{r}(z; k_0),
\end{equation}
for all $r > 0$ with $ \pi r^2 <  \textup{cap}(U_\varepsilon \cap \textup{\textbf{B}}_R(u_*))$.
\end{thm}

\begin{rmk}
\label{rmk:capacity}
The capacity $\textup{cap}(U_\varepsilon \cap \textup{\textbf{B}}_R(u_*))$ is positive since it is the capacity of an open set. In practice, the sets $U_\varepsilon$ will be constructed by taking the $\rho$-fattening of the subsets $\Sigma_\lambda$, see Section \ref{sec:conv} for details. At the moment, we do not have a better bound for this capacity than the trivial bound for open sets. In the case that the flow can be defined on the interval $[0,T]$ for all initial data in $\textbf{B}_R(u_*)$. The proof of Theorem \ref{thm:weaknon-squeezing} is the only place where we need to use the infinite dimensional symplectic capacity. As a result, the critical stability theory alone is insufficient because the trivial lower bound for $\textup{cap}(U_\varepsilon \cap \textup{\textbf{B}}_R(u_*))$ depends on $\varepsilon$ which is the parameter which yields the control for long-time approximations in the stability theory.
\end{rmk}

\subsubsection*{Notation} We write $X \lesssim Y$ to denote $X \leq C Y$ for some $C > 0$ which depends only on fixed parameters. Occasionally we will also use the notation $\cH^s := H^s \times H^{s-1}$, where this product space is endowed with the obvious norm. We let $\textbf{B}_R(u_*)$ denote the ball of radius $R$ centered at $u_*$ and occasionally we will use the shorthand $\textbf{B}_R := \textbf{B}_R(0)$. We let $\textbf{C}_r(z;k)$ denote the cylinder centered at $z \in \bC$ of radius $r$ in the $k$-th frequency defined in \eqref{equ:cylinder}. We define the Fourier projection $\Pi_K$ on $\cH^{s}$ as the projection onto frequencies $|k| \leq K$ and we define $\Pi_{\geq K}= I - \Pi_K$. We define the smooth projection operator
\begin{align}
\label{equ:trunc_op}
P_{\leq N}(u)(x) \equiv \psi( -N^{-2} \Delta)(u)(x) = \widehat{u}(0) + \sum_{n \in \bZ_*^3} \psi\left( \frac{|n|^2}{N^2}\right) \widehat{u}(n) e^{i (n \cdot x)}, 
\end{align}
for $\psi$ a smooth cut-off function

\subsubsection*{Organization of Paper}
In Section \ref{sec:sym_hilb}, we provide some background on symplectic Hilbert spaces and the relation of non-squeezing to the energy transition problem and we introduce the capacity we will work with. In Section \ref{sec:prelim}, we collect some deterministic and probabilistic facts. In Section \ref{sec:well-posedness}, we prove local and global properties of solutions to the full equation \eqref{equ:cubic_nlkg} and similarly for the equation with truncated nonlinearity \eqref{equ:cubic_nlkg_trun}. In Section \ref{sec:compactness}, we prove the boundedness assumptions on the flow maps of these equations. In Section \ref{sec:decomp}, we prove Proposition \ref{prop:conv}. Finally in Section \ref{sec:conv}, we prove Theorem \ref{thm:trunc_approx} and Theorem \ref{thm:weaknon-squeezing}. 

\subsubsection*{Acknowledgements}
This work is part of the author's thesis and she would like to thank her advisor, Gigliola Staffilani, for suggesting this line of work and for all her help and patience throughout the project. The author would also like to thank Larry Guth and Emmy Murphy for several discussions on symplectic capacities and Jonas L\"uhrmann for helpful comments on an earlier draft.

\section{Symplectic Hilbert spaces}
\label{sec:sym_hilb}

We begin with some background on symplectic Hilbert spaces. We follow the exposition in \cite{Kuk, CKSTT05}. Consider a Hilbert space $\cH$ with scalar product $\langle \cdot , \cdot \rangle$ and a symplectic form $\omega_0$ on $\cH$. Let $J$ be an almost complex structure on $\cH$ which is compatible with the Hilbert space inner product, that is, a bounded self-adjoint operator with $J^2 = -1$ such that for all $\textup{\bf{u}},\textup{\bf{v}} \in \cH$, $\omega_0(\textup{\bf{u}},\textup{\bf{v}}) = \langle \textup{\bf{u}}, J\textup{\bf{v}} \rangle.$

\begin{defn}
\label{defn:symplectic}
We say the pair $(\cH, \omega_0)$ is the symplectic phase space for a PDE with Hamiltonian $H[\textup{\bf{u}}(t)]$ if the PDE can be written as $\dot{\textup{\bf{u}}}(t) = J \nabla H[\textup{\bf{u}}(t)]$.
\end{defn}
\noindent Here, $\nabla$ is the usual gradient with respect to the Hilbert space inner product, defined by
\[
\langle \textup{\bf{v}}, \nabla H[\textup{\bf{u}}] \rangle \equiv \frac{d}{d\varepsilon} H[\textup{\bf{u}} + \varepsilon \textup{\bf{v}}] \Big|_{\varepsilon = 0}.
\]
Definition \ref{defn:symplectic} is equivalent to the condition
\begin{align}
\label{equ:symp_pde}
\omega_0 (\textup{\bf{v}}, \dot{\textup{\bf{u}}}(t) ) = \omega_0(\textup{\bf{v}}, J \nabla H[\textup{\bf{u}}(t)]) =  - \langle \textup{\bf{v}}, \nabla H[\textup{\bf{u}}] \rangle = - \frac{d}{d\varepsilon} H[\textup{\bf{u}} + \varepsilon \textup{\bf{v}}]  \Big|_{\varepsilon = 0}.
\end{align}
Let $\langle \nabla \rangle:=(1 - \Delta)^{1/2}$ and consider the Hilbert space $ \cH^{1/2}(\bT^3)$ with the usual scalar product
 \[
\langle (\textup{\bf{u}}_1, \textup{\bf{u}}_2), (\textup{\bf{v}}_1, \textup{\bf{v}}_2) \rangle_{\frac{1}{2}} :=  \int_{\bT^3} \textup{\bf{u}}_1 \cdot \langle \nabla \rangle \, \textup{\bf{v}}_1  + \int_{\bT^3}  \textup{\bf{u}}_2 \cdot  \langle \nabla \rangle^{-1} \, \textup{\bf{v}}_2 .
 \]
For $(u, u_t) = (\textup{\bf{u}}_1, \textup{\bf{u}}_2)$ we can rewrite \eqref{equ:cubic_nlkg} as the system of first order equations
\begin{equation} \label{equ:cubic_nlw_sys}
 \left\{ \begin{split}
(\textup{\bf{u}}_1)_t  &=  \textup{\bf{u}}_2 \\
(\textup{\bf{u}}_2)_t &=  - (1-  \Delta) \textup{\bf{u}}_1  -   (\textup{\bf{u}}_1)^3.
 \end{split} \right.
\end{equation}
Define the skew symmetric linear operator
 \begin{align*}
J :  \cH^{1/2}(\bT^3) \to  \cH^{1/2}(\bT^3), \quad
J = \begin{pmatrix}
0 & \langle \nabla \rangle^{-1} \\
 -\langle \nabla \rangle & 0
\end{pmatrix},
 \end{align*}
then $J$ is an almost complex structure on $\cH^{1/2}(\bT^3)$ compatible with the symplectic form
\begin{align*}
\omega_{\frac{1}{2}}\bigl(\textup{\bf{u}}, \textup{\bf{v}}\bigr) :=  \int_{\bT^3}  \textup{\bf{u}}_1 \cdot \textup{\bf{v}}_2  - \int_{\bT^3}  \textup{\bf{u}}_2 \cdot  \textup{\bf{v}}_1,
\end{align*}
that is, setting $\textup{\bf{u}}:=(\textup{\bf{u}}_1, \textup{\bf{u}}_2)$ and $\textup{\bf{v}}= (\textup{\bf{v}}_1, \textup{\bf{v}}_2)$, we have $\omega_{\frac{1}{2}} \bigl(\textup{\bf{u}},\textup{\bf{v}}\bigr)  = \langle \textup{\bf{u}}, J \textup{\bf{v}} \rangle_{\frac{1}{2}}$.
Then we can write $\dot{\textup{\bf{u}}}   = J \nabla H(\textup{\bf{u}})$ for the Hamiltonian
\[
H(\textup{\bf{u}}) = \frac{1}{2} \int |\nabla \textup{\bf{u}}_1 |^2 + \frac{1}{2} \int | \textup{\bf{u}}_1|^2  + \frac{1}{2} \int |\textup{\bf{u}}_2|^2 +   \frac{1}{4} \int |\textup{\bf{u}}_1|^4 .
\]
Up to modifying the Hamiltonian, roughly the same computations hold for the nonlinear Klein-Gordon equation with any nonlinearity and in any dimension.

\subsection{An infinite dimensional symplectic capacity}
\label{susec:sympl_cap}
Kuksin's construction of a symplectic capacity for an open set $\cO$ is based on finite dimensional approximations of this set. It is an infinite dimensional analogue of the Hofer-Zehnder capacity \cite{HZ}. Before defining this capacity, we first recall the definition of a symplectic capacity on a symplectic phase space $(\cH, \omega)$.
\begin{defn}
\label{def:cap}
A symplectic capacity on $(\cH, \omega)$ is a function $\textup{cap}$ defined on open subsets $\cO \subset \cH$ which takes values in $[0, \infty]$ and has the following properties:
\begin{enumerate}
\item[1)] Translational invariant: $\textup{cap}(\cO) = \textup{cap}(\cO + \xi)$ for $\xi \in \cH$.
\item[2)] Monotonicity: $\textup{cap}(\cO_1) \geq \textup{cap}(\cO_2)$ if $\cO_1 \supseteq \cO_2$.
\item[3)] 2-homogeneity: $\textup{cap}(\tau \cO) = \tau^2 \textup{cap}(\cO)$.
\item[4)] Non-triviality: $0 < \textup{cap}(\cO) < \infty$ if $\cO  \neq \emptyset$ is bounded.
\end{enumerate}
\end{defn}
In finite dimensions, the symplectic capacity is an important symplectic invariant, namely for a given symplectomorphism, $\varphi$, and an open subset, $\cO \subset \cH$, we have that $\textup{cap}(\varphi(\cO)) = \textup{cap}(\cO)$. In \cite{Kuk}, Kuksin shows that the infinite dimensional symplectic capacity he constructs is invariant under the flow of certain nonlinear dispersive Hamiltonian equations. We will now define Kuksin's infinite dimensional capacity. For a given Darboux basis of $\cH$, let $\cH_N$ denote the span of the first $N$ basis vectors. Similarly, we use the notation $\cO_N$ for any subset $\cO$ projected onto these basis vectors. We collect a few definitions.
\begin{defn}[Admissible function]
Consider a smooth function $f \in C^\infty(\cO)$ and let $m > 0$. The function $f$ is called $m$-admissible if
\begin{enumerate}
\item[i)] $0 \leq f \leq m$ everywhere.
\item[ii)] $f \equiv 0$ in a nonempty subdomain of $\cO$.
\item[iii)] $f |_{\partial \cO} \equiv m$ and the set $\{ f < m\}$ is bounded and the distance from this set to $\partial \cO$ is $d(f) > 0$.
\end{enumerate}
\end{defn}

\begin{defn}[Fast function] Let $f_N := f |_{\cO_N}$ and consider the corresponding Hamiltonian vector field $U_{f_N}$, that is, for $z, v \in \cH_N$, we have
\[
\omega(U_{f_N}(z), v) = \nabla f_N(z) v.
\]
A periodic trajectory of $U_{f_N}$ is called fast if it is not a stationary point and its period $T$ satisfies $T \leq 1$. An admissible function $f$ is called fast if there exists $N_0(f)$ such that for all $N \geq N_0$, the vector field $U_{f_N}$ has a fast trajectory.
\end{defn}

\begin{rmk}
In light of the fact that $J^2 = -I$, we also have the representation
\[
U_{f_N}(z) = J \nabla f_N(z).
\]
\end{rmk}
With these definitions, we are now ready to state the definition of Kuksin's infinite dimensional capacity.
\begin{defn}
For an open, nonempty domain $\cO \subset \cH$, its capacity $\textup{cap}(\cO)$ equals
\[
\textup{cap}(\cO) = \inf \{m_* \,|\, \textup{each $m$-admissible function with $m > m_*$ is fast}\}.
\]
\end{defn}
In \cite{Kuk}, it is shown that this definition satisfies the axioms of a capacity, that is the criteria of Definition \ref{def:cap}, and while the construction of this capacity depends on the choice of Darboux basis, if one chooses another basis which is quadratically close to the first, then the capacity does not change.

\section{Preliminaries}
\label{sec:prelim}

\subsection{Deterministic preliminaries}

We recall the definition of the $X^{s,b}$ spaces with norm
\begin{align}
\label{equ:xsb_def_chap3}
\|u\|_{X^{s,b}(\bR \times \bT^3)} = \|\langle n \rangle^s \langle |\tau| - \langle n \rangle \rangle^b \widehat u(n,\tau) \|_{L^2_\tau \ell_n^2}.
\end{align}
We will also work with the local-in-time restriction spaces $X^{s,b,\delta}$, which are defined by the norm
\begin{align}
\|u\|_{X^{s,b,\delta}} = \inf \bigl\{\|\widetilde{u}\|_{X^{s,b}(\bR \times \bT^3)} : \widetilde{u} |_{[-\delta, \delta]} = u \bigr\}.
\end{align}
We refer the reader to Appendix \ref{sec:harmonic} for more details.

We will record some facts about the projection operator \eqref{equ:trunc_op}. In \cite{BT5}, these operators were used to define an approximating equation for the cubic nonlinear wave equation and we similarly use them to define the equation with truncated nonlinearity \eqref{equ:cubic_nlkg_trun}. We use this smoothed projection instead of the standard truncation because we will need to exploit the fact that this family of operators has uniform $L^p$ bounds.
\begin{lem}
\label{lem:lp_bds}
Let $M$ be a compact Riemannian manifold and let $\Delta$ be the Laplace Beltrami operator on $M$. Let $1 \leq p \leq \infty$. Then $P_N \equiv  \psi( -N^{-2} \Delta) : L^p(K_1) \to L^p(K_1)$ is continuous and there exists $C > 0$ such that for every $N \geq 1$, 
\[
\|P_N \|_{L^p \to L^p} \leq C.
\]
Moreover, for all $f \in L^p(K_1)$, $P_N f \to f$ in $L^p$ as $n \to \infty$.
\end{lem}

\begin{proof}
See \cite[Theorem 2.1]{sogge}.
\end{proof}

We need the following identity to prove the symplectic properties for the truncated nonlinear Klein-Gordon equation.

\begin{lem}
\label{lem:proj_identity}
Let $K$ be large enough so that $\Pi_K P_N = P_N$, then
\[
\int_{\bT^3} P_N\bigl[(P_Nu )^3\bigr] \,\Pi_K v = \int_{\bT^3} (P_Nu)^3 P_N  v.
\]
\end{lem}

In the sequel, we let $\langle \nabla \rangle$ be the operator with symbol $\sqrt{1 + |\xi|^2}$ and we let $F(u) = u^3$. We rewrite \eqref{equ:cubic_nlkg} as a system of first order equations by factoring
\[
\partial_{tt} - \Delta + 1 = ( \langle \nabla \rangle + i\partial_t)( \langle \nabla \rangle - i\partial_t).
\]
For a sufficiently regular solution $u$ to the NLKG \eqref{equ:cubic_nlkg} we can define
\[
u^{\pm} = \frac{ ( \langle \nabla \rangle \mp i\partial_t) }{ 2 \langle \nabla \rangle } u,
\]
then $u = u^+ + u^-$ and the functions $u^\pm$ solve the equations
\begin{align}
\label{equ:soln_kg_sys}
 ( \langle \nabla \rangle \pm i\partial_t) u^\pm = - \frac{F(u^+ + u^-)}{2 \langle \nabla \rangle} , \qquad u^{\pm}(0) = \frac{1}{2}\left(u_0 \mp i \frac{u_1}{\langle \nabla \rangle} \right).
\end{align}
We set
\[
I^\pm(f) = \int_0^t e^{\pm i (t-s) \langle \nabla \rangle} \frac{f}{2 \langle \nabla \rangle} ds
\]
and we obtain a Duhamel's formula for solutions of \eqref{equ:soln_kg_sys} given by
\begin{align}
\label{equ:duhamel}
u^{\pm} = e^{\pm i t  \langle \nabla \rangle} u_0^{\pm} \pm i I^\pm (F(u)).
\end{align}
This formulation is equivalent to \eqref{equ:cubic_nlkg} and since $u =  u^+ + u^-$, we can reconstruct a solutions to \eqref{equ:cubic_nlkg} from this system and bounds for $u^\pm$ imply the same bounds for $u$. We will not be using the specific structure of $u^\pm$ so we will often drop the notation where it has no impact on the argument.

\medskip
Before proceeding, we recall the definition of the $U^p$ and $V^p$ spaces. Consider partitions given by a strictly increasing finite sequence $-\infty < t_0 < t_2 < \ldots t_{K} \leq \infty$. If $t_K = \infty$ we use the convention $v(t_{K}) := 0$ for all functions $v: \bR \to H$. We will usually be working on bounded intervals $I \subset \bR$. For some additional details about these function spaces, see Appendix \ref{sec:up_vp}. In what follows, $\cB$ will denote an arbitrary Banach space.

\begin{defn}[$U^p$ spaces]
Let $1 \leq p < \infty$. Consider a partition $\{t_0, \ldots, t_K\}$ and let $(\varphi_k)_{k=0}^{K-1} \subseteq \cB$ with $\sum_{ k =0}^{K-1} \|\varphi_k\|_{L^2}^p = 1$. We define a $U^p$ atom to be a function
\[
a := \sum_{k=1}^{K} \mathbbm{1}_{[t_{k-1}, t_{k})} \varphi_{k-1}
\]
and we define the atomic space $U^p(\bR, \cB)$ to be the set of all functions $u: \bR \to \cB$ such that
\[
 u = \sum_{j=1}^\infty \lambda_j a_j,
\]
for $a_j$ $U^p$ atoms, and $\{\lambda_j\} \in \ell^1(\bC)$, endowed with the norm
\[
\|u\|_{U^p} := \inf \left\{ \sum_{j=1}^\infty |\lambda_j|, \,u = \sum_{j=1}^\infty \lambda_j a_j \,: \, a_j \textup{ is a } U^p\textup{ atom} \right\}.
\]
\end{defn}

\begin{defn}[$V^p$ spaces]
Let $1 \leq p < \infty$. We define $V^p(\bR, \cB)$ as the space of all functions, $v : \bR \to \cB$, such that the norm 
\begin{align}
\label{equ:v_norm}
\|v\|_{V^p(\bR, \cB)} = \sup_{\textup{partitions }} \left( \sum_{i=1}^K \|v(t_{i}) - v(t_{i-1})\|^p_{\cB} \right)^{1/p}< \infty
\end{align}
with the convention $v(\infty) = 0$.  We let $V_- (\bR , \cB)$ denote the subspace of all functions satisfying $\lim_{t \to - \infty} v(t) = 0$ and $V_{rc}^p(\bR , \cB) $ denote the subspace of all right continuous functions in $V_- (\bR , \cB)$, both endowed with norm defined above.
\end{defn}

\medskip
We will establish a local well-posedness theory for \eqref{equ:soln_kg_sys} via a contraction mapping argument in the adapted functions space
\begin{align}
\label{equ:x_spaces}
\|u\|_{\cX_\pm^{s}} = \Bigl( \sum_{k} \langle k \rangle^{2s} \|\widehat{u^\pm}(k) \|^2_{U_\pm^2} \Bigr)^{1/2}, \qquad \|f\|_{U_\pm^2} = \| e^{\mp i t \langle \nabla \rangle } f\|_{U^2L^2(\mathbb{R} \times \bT^3)}.
\end{align}
Similarly we define
\begin{align}
\label{equ:y_spaces}
\|u\|_{\cY_\pm^{s}} = \Bigl( \sum_{k} \langle k \rangle^{2s} \|\widehat{u^\pm}(k) \|^2_{V_{rc,\pm}^2} \Bigr)^{1/2}, \qquad \|f\|_{V_{rc,\pm}^2} = \| e^{\mp i t \langle \nabla \rangle } f\|_{V_{rc}^2L^2(\mathbb{R} \times \bT^3)}.
\end{align}
We will simplify the notation and let $V_{\pm}^2 := V_{rc,\pm}^2$. We define
\[
\cX^s = \cX_+^s \times X_-^s \quad \textup{and} \quad \cY^s = Y_+^s \times Y_+^s
\]
endowed with the obvious norm. These spaces are at a slightly finer scale than the spaces used in \cite{HHK}, and one should view these as analogous to the spaces used by Herr, Tataru and Tzvetkov in \cite{HTT} for the quintic nonlinear Schr\"odinger equation on $\bT^3$. This choice of scale does not affect the multilinear estimates, as we only need the weaker Strichartz estimates at dyadic scales, however it allows us to make use of the following important orthogonality property.
\begin{cor}[Corollary 2.9, \cite{HTT}]
\label{cor:orthog}
Let $\bZ^3 = \cup \, C_k$ be a partition of $\bZ^3$ or $\bR^k$. Then
\begin{align}
\biggl( \sum_k  \|P_{C_k} u\|^2_{V_{\pm}^2 H^s(I \times \bT^3)} \biggr)^{1/2} \lesssim \|u\|_{\cY_\pm^s(I \times \bT^3)}.
\end{align}
\end{cor}

\noindent We will work with the restriction spaces
\[
\cX^s(I) := \bigl\{ u \in C(I; H^s(\bT^3) \times H^s(\bT^3)) \,|\,\widetilde{u}(t) =  u(t) \textup{ for } t \in I, \widetilde{u} \in \cX^s \bigr\}
\]
endowed with the norm 
\[
\|u\|_{\cX^s(I)} = \inf \bigl \{\|\widetilde{u}\|_{\cX^s} \,|\, \widetilde{u} \,:\, \widetilde{u}(t) =  u(t) \textup{ for } t \in I \bigr\},
\]
and similarly for $\cY^s(I)$. In our estimates we will often implicitly multiply functions by a sharp time cut-off. When $I = [0,T)$ we will use the notation $\cX_T^s$. Perhaps most importantly, bounds for solutions of \eqref{equ:soln_kg_sys} in these spaces implies the same bounds in $L_t^\infty \cH_x^{s}$ for solutions of \eqref{equ:cubic_nlkg}, which was precisely the endpoint embedding we were missing in the $X^{s,\frac{1}{2}}$ spaces. Finally, for $I \subset J$, we have the embedding $\cX^s(I) \hookrightarrow \cX^s(J)$ which can be seen via extension by zero. 

\medskip
As mentioned above, we will need a weaker norm which controls the local well-posedness theory for the equation in order to prove the necessary stability theory. This norm is similar to the norm introduced in \cite[Section 2]{IP12} for the same purpose and should be thought of as the appropriate substitute for the $L_{t,x}^4$ Strichartz norm. We define the $Z(I)$ norm by
\begin{align}
\label{equ:zs_norm}
\|f\|_{Z(I)} 
& = \sup_{ J \subseteq I,\, |J| \leq 1} \left( \sum_N  \|P_N f\|^2_{L_{t, x}^4(J \times \bT^3)} \right)^{1/2}.
\end{align}
For $u = (u^+, u^-)$, we obtain as a consequence of Corollary \ref{cor:orthog} and the Strichartz estimates in Corollary \ref{cor:strichartz} below that
\[
\|u^\pm\|_{L_{t,x}^4(I \times \bT^3)} \lesssim \|u^\pm\|_{Z(I)} \lesssim \|u^\pm\|_{\cY_\pm^{1/2}(I)} \lesssim \|u^\pm\|_{\cX_\pm^{1/2}(I)},
\]
hence this indeed defines a weaker norm. This norm only plays a role in the stability theory and is not necessary to prove the low-frequency component satisfies the perturbed equation \eqref{equ:lov_err}.

\subsection{Strichartz estimates}
In the next sections we will repeatedly make use of Strichartz estimates.

\begin{defn}
A pair of real numbers $(q,r)$ is called wave-admissible provided $2 < q \leq + \infty$, and $2 \leq r < + \infty$, and it satisfies the scaling condition 
\[
\frac{1}{q} + \frac{1}{r} \leq \frac{1}{2}.
\]
\end{defn}

By finite speed of propagation, the Strichartz estimates for the torus are the same as those in Euclidean space, provided one localizes in time. These estimates are classical and due in parts to \cite{Strichartz}, \cite{Pecher}, \cite{GV2}, \cite{KeelTao}.

\begin{prop}[Strichartz estimates]
\label{prop:strichartz2}
Let $u$ be a solution to the equation
\begin{align}
u_{tt} - \Delta u + u = F(u), \qquad (u, \partial_t u)\big|_{t= 0} = (u_0, u_1)
\end{align}
on $I \subset [0, T]$ for $T$ fixed. Then
\begin{align}
\label{equ:strichartz_estimates}
\|(u, \partial_t u)\|_{L_t^\infty  \cH^{\gamma}(I \times \bT^3)} + \|u \|_{L_t^q L_x^r(I \times \bT^3)} &\lesssim \|(u_0, u_1)\|_{\cH^{\gamma}(\bT^3)}  +  \|F\|_{L^{\tilde{p}'}_t L^{\tilde{q}' }_x (I\times \bT^3)} ,
\end{align}
under the scaling assumption
\begin{align}
\label{equ:scaling}
\frac{1}{q} + \frac{3}{r} = \frac{3}{2} - \gamma
\end{align}
and the gap condition
\begin{align}
\label{equ:gap}
\frac{1}{\tilde{q}} + \frac{1}{\tilde{r}} - 2 = \frac{3}{2} - \gamma
\end{align}
where the implicit constant depends on the choice of $T$ but is uniform for any subinterval $I \subset[0,T]$.
\end{prop}
An exponent pair $(q,r)$ is called $\gamma$ admissible provided it satisfies the scaling assumptiong \eqref{equ:scaling} for some $0 < \gamma < \frac{3}{2}$. Similarly, we call an exponent pair $(\tilde{q}', \tilde{r}')$ conjugate admissible provided $\frac{1}{q } + \frac{1}{\tilde{q}'} = 1 = \frac{1}{\tilde{r}} + \frac{1}{\tilde{r}'}$ for some wave-admissible pair $(p,q)$ and it satisfies the gap condition \eqref{equ:gap} for some $0 < \gamma < \frac{3}{2}$.

\begin{rmk}
In fact, a larger range of exponents are admissible for the nonlinear Klein-Gordon equation than the range stated above, which only coincides with the admissible exponents for the nonlinear wave equation. Since the above estimates are all we will need in our arguments, we refrain from stating the full range of Strichartz exponents for simplicity. For a full formulation, see for instance \cite{NO}, and references therein.
\end{rmk}

We state the following interpolation estimate for the Strichartz estimates for the nonlinear Klein-Gordon equation which we will use to estimate frequency localized functions in $X^{s,b}$ spaces, see \cite{B95} for this formulation of the estimates.

\begin{prop}[Strichartz]
\label{prop:Strichartz}
Let $I \subset \bR$. There exists a constant $C \equiv C(I) > 0$ such that
\[
\left\| \sum_{|n| \sim N} a(n) e^{i (x \cdot n + t \langle n \rangle)} \right\|_{L^4(I \times \bT^3)} \leq C(I) N^{1/2} \left( \sum |a(n)|^2  \right)^{1/2}.
\]
By H\"older's inequality we have
\[
\left\|  \sum_{|n| \sim N} \int d\tau \frac{c(n, \tau)}{ \langle|\tau| - \langle n \rangle\rangle^{\frac{1}{2} +} } e^{ i (x \cdot n + t \tau )}  \right\|_{L^4(I \times \bT^3)} \leq C(I)  N^{\,1/2} \left( \int \sum |c(n, \tau)|^2  d\tau \right)^{1/2}
\]
and by interpolating with Parseval's identity we obtain that for $2 \leq r \leq 4$ 
\begin{align}
\label{equ:interp}
\left\|  \sum_{|n| \sim N} \int d\tau \frac{c(n, \tau)}{ \langle|\tau| - \langle n \rangle\rangle^{\frac{\theta}{2}  +} } e^{ i (x \cdot n + t\tau )}  \right\|_{L^r(I \times \bT^3)} \leq C(I)  N^{\,\theta/2} \left( \int \sum |c(n, \tau)|^2  d\tau \right)^{1/2}
\end{align}
where $\theta = 2 - \frac{4}{r}$. 
\end{prop}

We also obtain the following formulation of Strichartz estimates in $U^p$  and $V^p$ spaces. The reasoning is identical to the proof of Corollary 2.21 in \cite{HHK}.

\begin{cor}
\label{cor:strichartz}
 Let  $1 \leq p < 4$ and $2 \leq q \leq 4$. Let $T > 0$, and let $u_N = P_N u_N$. Then,
\begin{enumerate}
\item[(i)] $\| u_N \|_{L^4_{t,x}([0,T] \times \bT^3)} \lesssim C_T N^{1/2} \| u_N \|_{U_\pm^{4}}$
\item[(ii)] $\| u_N \|_{L^4_{t,x}([0,T] \times \bT^3)} \lesssim C_T N^{1/2} \| u_N \|_{V_{-, \pm}^{p}}$
\item[(iii)] $\| u_N \|_{L^q_{t,x}([0,T] \times \bT^3)} \lesssim C_T N^{\frac{q-2}{q}} \| u_N \|_{U_\pm^{q}}$ .
\end{enumerate}
\end{cor}
\begin{proof}
The first claim follows from Proposition \ref{prop:strichartz2} and the transfer principle, Proposition \ref{prop:transfer}. The second claim follows from \eqref{equ:v_in_u} and the fact that $u$ agrees with its right-continuous variant almost everywhere. The final claim follows from interpolation with the trivial $L_{t,x}^2$ estimate, since we are on bounded time intervals.
\end{proof}

We have the following refined bilinear Strichartz estimate for the Klein-Gordon equation. Such estimates were treated in full for wave equations in \cite{FK}. For a proof of this statement for the Klein-Gordon equation on Euclidean space, see \cite[Appendix A]{Schottdorf} which adapts the geometric proofs from \cite{Selberg08} for the nonlinear wave equation, see also \cite{KSV}. One can also adapt the proof from \cite{Selberg08} to the compact setting to prove Proposition \ref{prop:refined_l2} directly, replacing the volume estimates with fairly straightforward lattice counting.

\begin{prop}
\label{prop:refined_l2}
Fix $T > 0$ and let  $O, M, N$ be dyadic numbers and $\phi_M, \psi_N$ functions in $L^2(\bT^3)$ localized at frequencies $M,N$ respectively. Define $u_M = e^{ \pm_1 it \langle \nabla \rangle } \phi_M$ and $v_M = e^{ \pm_2 it \langle \nabla \rangle } \psi_N$ Denote $L = \min(O,M,N)$, and $H = \max(O,M,N)$. Then
\begin{align}
\|P_O(u_M v_N) \|_{L_{t,x}^2} \lesssim C_T
\begin{cases}
L \|\phi_M \|_{L^2(\bT^3)} \|\psi_N \|_{L^2(\bT^3)}  & if \,\,\, M \ll N \\
H^{\frac{1}{2}} L^{\frac{1}{2}} \|\phi_M \|_{L^2(\bT^3)} \|\psi_N \|_{L^2(\bT^3)}  &if \,\,\, M \sim N .
\end{cases}
\end{align}
\end{prop}

By the transfer principle and an orthogonality argument and Corollary \ref{cor:strichartz}, we obtain the following version of the above estimates.

\begin{prop}[Proposition 7, \cite{Schottdorf}]
\label{prop:multilinear}
Fix $T > 0$ and let $L$, respectively $H$, denote the highest and lowest frequencies of $M, N, O$. Let $u_M \in U_{\pm_1}^2$, $u_N \in U^2_{\pm_2}$. Then
\begin{align}
\|P_O(u_M v_N) \|_{L_{t,x}^2} \lesssim C_T
\begin{cases}
L \|u_M \|_{U_{\pm_1}^2} \|v_N \|_{U_{\pm_2}^2}  &if \,\,\, M \ll N \\
H^{\frac{1}{2}} L^{\frac{1}{2}} \|u_M \|_{U_{\pm_1}^4} \|v_N \|_{U_{\pm_2}^4}  &if \,\,\, M \sim N .
\end{cases}
\end{align}
\end{prop}

\begin{rmk}
One can convert these estimates to bounds in $V_{\pm}^2$ with a logarithmic loss in the first estimate, and no loss in the second estimate, see Proposition \ref{prop:log_loss}.
\end{rmk}

Finally, in order to prove the multilinear estimates required for the stability theory, we will make use of the following refined Strichartz estimate which is proved using Proposition \ref{prop:refined_l2}.

\begin{prop}[Proposition 10, \cite{Schottdorf}]
\label{prop:l4_strichartz}
Let $n \geq 3$ and let $M \lesssim N$ and let $u_{M,N}$ have Fourier support in a ball of radius $\sim M$ centered at frequency $N$. Then
\[
\|u_{M,N} \|_{L_{t,x}^4([0,T) \times \bT^3)} \lesssim C_T \, N^{1/4} M^{\frac{n-2}{4}} \|u_{M,N} \|_{U_{\pm}^4}.
\]
In particular if $u_N = P_N u_N$ then the result holds with $M= N$.
\end{prop}

\subsection{Probabilistic preliminaries}
\label{sec:prob_prelim}
We will record some of the basic probabilistic results about the randomization procedure. Most of these estimates are consequences of the classical estimates of Paley-Zygmund for random Fourier series on the torus. These estimates were used heavily in the works of Burq and Tzvetkov, see especially \cite{BT5} for proofs. We begin with the standard large deviation estimate.

\begin{prop}[Large deviation estimate; Lemma 3.1 in \cite{BT1}] \label{prop:large_deviation_estimate}
 Let $\{h_n\}_{n=1}^{\infty}$ be a sequence of complex-valued independent random variables with associated distributions $\{\mu_n\}_{n=1}^{\infty}$ on a probability space $(\Omega, {\cA}, \bP)$. Assume that the distributions satisfy the property that there exists $c > 0$ such that
 \begin{equation*}
  \bigl| \int_{-\infty}^{+\infty} e^{\gamma x} d\mu_n(x) \bigr| \leq e^{c \gamma^2} \text{ for  all } \gamma \in \bR \text{ and for all } n \in \mathbb{N}.
 \end{equation*}
 Then there exists $c > 0$ such that for every $\lambda > 0$ and every sequence $\{c_n\}_{n=1}^{\infty} \in \ell^2(\bN)$ of complex numbers, 
 \begin{equation*}
  \bP \Bigl( \omega : |\sum_{n=1}^{\infty} c_n h_n(w)| > \lambda \Bigr) \leq 2 e^{-\frac{c \lambda^2}{\sum_n |c_n|^2}}.
 \end{equation*}
 As a consequence there exists $C > 0$ such that for every $p \geq 2$ and every $\{c_n\}_{n=1}^{\infty} \in \ell^2(\bN)$,
 \begin{equation*}
  \bigl\| \sum_{n=1}^{\infty} c_n h_n(\omega) \bigr\|_{L^p(\Omega)} \leq C \sqrt{p} \bigl( \sum_{n=1}^{\infty} |c_n|^2 \bigr)^{1/2}.
 \end{equation*}
\end{prop}

\begin{rmk}
Gaussian iid or Bernoulli random variables satisfy the exponential moment assumption in the statement of Proposition \ref{prop:large_deviation_estimate}.
\end{rmk}

\medskip
This large deviation estimate is the key component in the proof of the following corollary, which states that the free evolution of randomized initial data satisfies almost surely better integrability properties. There is the minor modification in the following that we are dealing with complex random variables $\{h_{k}\}$ which satisfy a symmetry condition. Given that the functions we are randomizing are real-valued, and thus have Fourier coefficients which satisfy an analogous symmetry condition, the arguments for the following results go through unchanged. Recall in the sequel that $S(t)$ which we defined in \eqref{equ:free} denotes the free evolution for the Klein-Gordon equation.

\begin{cor}[Corollary A.5, \cite{BT4}]
\label{cor:averaging_effects}
Fix $\mu \in \cM^s$ and suppose $\mu$ is induced via the map \eqref{equ:bighsrandomization_chap5} for $(f_0, f_1) \in \cH^{s}$. Let $2 \leq p_1 < \infty$, $2 \leq p_2 < \infty$ and $\delta > 1 + \frac{1}{p_1}$ and $0 < \sigma \leq s$ Then there exist constants $C, c > 0$ such that for every $\lambda > 0$,
\begin{align}
\label{equ:cor_averaging}
\mu\Bigl( \{ (u_0, u_1) \in \cH^{s} \,: \, \|\langle t \rangle^{-\delta} S(t)(u_0, u_1)\|_{L_t^{p_1}L_x^{p_2}(\bR \times \bT^3)} > \lambda \} \Bigr) \leq C e^{ - c \lambda^2 / \| (f_0, f_1) \|_{\cH^{\sigma}}^2}.
\end{align}
\end{cor}

\begin{rmk}
We can include the endpoint $p_2 = \infty$ if we restrict to bounding $0 < \sigma < s$ in the statement above, or equivalently to bounding.
\[
\|\langle t \rangle^{-\delta} (1 - \Delta)^{\gamma/2} S(t)(u_0, u_1)\|_{L_t^{p_1}L_x^{p_2}}
\]
for any $0 < \gamma < s$.  By Sobolev embedding, we are also be able to include the endpoint $p_1 = \infty$ in this case. We use this in the proof of Proposition \ref{prop:unif_trunc}, see for instance \cite[Lemma 2.2]{BT5}.
\end{rmk}

Given these large deviation estimates, the following result, which enables us to construct the subset $\Sigma$ is a simple corollary.

\begin{cor}
\label{cor:asbdd}
Let $T > 0$ fixed, $\mu \in \cM^{s}$ and $2 \leq p < \infty$. Then there exists $C \equiv C(T)$ such that for any $0 \leq \gamma < \gamma_1 \leq 1/2$,
\[
\mu\Bigl( \{ (v_0, v_1) \in \cH^{s} \,: \, \| S(t)(u_0, u_1)\|_{L_t^{p_1} W_x^{\gamma, p_2}(\bR \times \bT^3)} > \lambda \} \Bigr) \leq C e^{ - c \lambda^2 / \| (f_0, f_1) \|_{\cH^{\gamma_1}}^2}.
\]
\end{cor}

Additionally, Corollaries \ref{cor:averaging_effects} and \ref{cor:asbdd} imply that the set of initial data which satisfy good local $L^p$ bounds has full $\mu$ measure.

\begin{cor}
\label{cor:asbdd2}
Fix $\mu \in \cM^{s}$ and let $2 \leq p < \infty$ and $0 < \gamma < s$. Then for a set of full $\mu$ measure,
\begin{align*}
\|(1 - \Delta)^{\gamma /2} S(t)(u_0, u_1)\|^p_{L_x^p(\bT^3)}  \in L^1_{\textup{loc}}(\bR_t),
 \qquad \| S(t)(u_0, u_1)\|_{L_x^\infty (\bT^3)}  \in L^1_{\textup{loc}}(\bR_t).
\end{align*}
Consequently, for $\Sigma$ as defined in \eqref{equ:sigma_def}, we have $\mu(\Sigma) = 1$.
\end{cor}

\section{Well-posedness theory}
\label{sec:well-posedness}

We record some global bounds on the solution to the cubic nonlinear Klein-Gordon \eqref{equ:cubic_nlkg}. The only new component in this statement is the bounds on the $L^4$ norm of the solution, which follows by a similar argument to the proof of \cite[Proposition 4.1]{BT4}. We include the proofs of the following statement for completeness but we omit the proof for the truncated equation as it follows in precisely the same manner.

\begin{prop}
\label{prop:unif_global2}
Let $0 < s < 1$ and let $\mu \in \cM^s$. Then for any $\varepsilon > 0$, there exist $C, c, \theta > 0$ such that for every $(u_0, u_1) \in \Sigma$, there exists $M= M(u_0, u_1) > 0$ such that the global solution $u$ to the cubic nonlinear Klein-Gordon equation\eqref{equ:cubic_nlkg} satisfies
\begin{align}
u(t) &= S(t)(u_0, u_1) + w(t) \\
\|(w(t), \partial_t w(t)) \|_{\cH^1} &\leq C\bigl(M + |t|\bigr)^{1 + \varepsilon}\label{equ:h1_bds} \\
\|u(t)\|_{L^4(\bT^3)} &\leq C \bigl(M + |t|\bigr)^{\frac{1}{2} + \varepsilon}\label{equ:L4_bds} 
\end{align}
and furthermore $\mu\bigl((u_0, u_1) \in \Sigma \,:\, M > \lambda\bigr) \leq C e^{- c\lambda^\theta}$.
\end{prop}

\begin{proof}
Fix $\mu \in \cM^s$. Following the proofs of Proposition 4.1 in \cite{BT4} and Lemma 2.2 from \cite{BT5}, fix $\varepsilon > 0$, $\rho > \frac{1}{2}$, $\tilde{\rho} > 1/3$, and $\check{\rho} > 0$ and define
\begin{align*}
U_M &:= \bigl\{ (v_0, v_1) \in \Sigma \,: \, \|\Pi_M(v_0, v_1)\|_{\cH^1} \leq N^{1 - s + \varepsilon} \bigr\}\\
G_M &:= \bigl\{ (v_0, v_1) \in \Sigma \,: \, \|\Pi_Mv_0\|_{L^4} \leq N^{ \varepsilon}\bigr\}\\
H_M &:= \bigl\{ (v_0, v_1) \in \Sigma \,: \, \|\langle t \rangle^{-\rho} \, \Pi_{\geq N}(v_0, v_1)\|_{L_t^2L^\infty_x(\bR \times \bT^3)} \leq N^{ \varepsilon - s}\bigr\}\\
K_M &:= \bigl\{ (v_0, v_1) \in \Sigma \,: \, \|\langle t \rangle^{-\tilde{\rho}} \, \Pi_{\geq N}(v_0, v_1)\|_{L_t^3L^6_x(\bR \times \bT^3)} \leq N^{ \varepsilon- s}\bigr\}\\
R_M &:= \bigl\{(v_0, v_1) \in \Sigma \,:\, \|\langle t \rangle^{-\check{\rho}}\, \Pi_{\geq N} S(t)(v_0, v_1)\|_{L_t^\infty L^4_{x}(\bR_t \times \bT^3)} \leq N^{\varepsilon - s} \bigr\}.
\end{align*}
We let
 
\begin{align}
\label{equ:en_def}
E_M = U_M \cap G_M \cap H_M \cap K_M \cap R_M.
\end{align}
The bounds on the measure of $E_M$ follow from Proposition 4.1 in \cite{BT4} and Lemma 2.2 in \cite{BT5}. We consider the inhomogeneous energy functional \eqref{equ:inhomog_energy}
\[
\cE(w(t)) = \frac{1}{2} \int |\nabla w|^2 + \frac{1}{2}\int |w|^2 + \frac{1}{2} \int |w_t|^2 + \frac{1}{4} \int (w)^4.
\]
Fix $(v_0, v_1) \in E_M$ and let $w_M$ denote the solution to
\[
(w_M)_{tt} - \Delta w_M + w_M + (w_M + S(t) \Pi_{\geq M} (v_0, v_1))^3 = 0, \quad (w_M ,\partial_t w_M) \big|_{t= 0} = \Pi_M(v_0, v_1)
\]
then
 
\[
\cE(w_M(t))^{1/2} \leq C M^{1 - s + \varepsilon}.
\]
Since $\cE$ controls the $\cH^1$ norm, we no longer need to project away from constants to obtain \eqref{equ:h1_bds}. To prove \eqref{equ:L4_bds}, note that by the definition of $R_M$, we have
\begin{align*}
\|u_M(t)\|_{L^4(\bT^3)} &\leq \|w_M(t)\|_{L^4(\bT^3)} + \|S(t) \Pi_{ \geq M} (v_0, v_1)\|_{L^4(\bT^3)} \\
& \leq \cE(w_M(t))^{1/4} + CM^{-s + \varepsilon}\\
& \leq CM^{\frac{1- s + \varepsilon}{2}}.
\end{align*}
The conclusion then follows as in Proposition 4.1 in \cite{BT4}.
\end{proof}

We now turn to studying the global well-posedness and symplectic properties of the approximating equation \eqref{equ:cubic_nlkg_trun}. Let $P_N$ be the smooth projection operator defined in \eqref{equ:trunc_op}. For $K$ sufficiently large so that $\Pi_K P_N = P_N$, this equation is equivalent to the uncoupled system
\begin{equation}
\label{equ:hamiltonian_wave_fin_dim}
\left\{\begin{split}
&  \partial_{tt} \Pi_K u_N + (1 - \Delta) \Pi_K u_N + P_N [ (P_N u_N)^3 ]  = 0, \quad (t,x) \in \bR \times \bT^3 \\
   & \Pi_K (u_N, \partial_t  u_N) \big|_{t=0} = \Pi_K ( u_0, u_1) \\
  & (1 - \Pi_K) (u_N) = S(t)(1 - \Pi_K)(u_0,  u_1)
\end{split} \right.
\end{equation}
which we can write as a first order system as in \eqref{equ:cubic_nlw_sys}.
\begin{rmk}
As remarked in the introduction, \eqref{equ:hamiltonian_wave_fin_dim} is a nonlinear flow for low frequencies and a decoupled linear evolution for high frequencies. In particular, the solution $u_N$ is supported on all frequencies. We will nonetheless call this the truncated flow for simplicity even though this defines a flow on the whole space.
\end{rmk}

Global well-posedness for \eqref{equ:hamiltonian_wave_fin_dim} follows from local well-posedness by observing that  the linear evolution is globally defined and the energy functional
\begin{align}
\label{equ:fin_dim_hamil}
\qquad  H_N(\Pi_K (u_N, (u_N)_t ) )= \frac{1}{2} \int_{\bT^3} |\nabla_x  \Pi_K u_N|^2 + ( \Pi_K u_N)^2 + ( \Pi_K (u_N)_t)^2  + \frac{1}{4} \int_{\bT^3} (P_N (u_N))^4 
\end{align}
is well defined and conserved under the flow of \eqref{equ:hamiltonian_wave_fin_dim} for bounded frequency components. Note that the bounds on the solution depend on the energy of the initial data and consequently are not uniform in the truncation parameter. Nonetheless, Burq-Tzvetkov \cite{BT5} proved that if one restricts to initial data $(u_0, u_1) \in \Sigma$, then the nonlinear components of the solutions to the cubic nonlinear wave equation satisfy uniform bounds. As was the case in Theorem \ref{thm:bt_global}, the proof of these uniform bounds follow for the nonlinear Klein-Gordon equation from the arguments in \cite{BT5} with only minor modifications. 

\begin{prop}[Proposition 3.1, \cite{BT5}]
\label{prop:unif_trunc}
Let $0 < s < 1$ and let $\mu \in \cM^s$. Then for any $\varepsilon > 0$, there exist $C, c, \theta > 0$ such that for every $(u_0, u_1) \in \Sigma$, there exists $M = M(u_0, u_1)> 0$ such that the family of global solutions $(u_N)_{N \in \bN}$ to \eqref{equ:hamiltonian_wave_fin_dim} satisfies
\begin{align*}
u_N(t) &= S(t)(u_0, u_1) + w_N(t) \\
\|(w_N(t), \partial_t w_N(t)) \|_{\cH^1} &\leq C\bigl(M + |t|\bigr)^{1 + \varepsilon}\\
\|P_N u_N(t)\|_{L^4(\bT^3)} &\leq C \bigl(M + |t|\bigr)^{\frac{1}{2} + \varepsilon}
\end{align*}
and furthermore $\mu((u_0, u_1) \in \Sigma \,:\, M > \lambda) \leq C e^{- c\lambda^\theta}$.
\end{prop}

The truncated flow maps also preserve symplectic capacities. This is an easy consequence of the fact that, when restricted to bounded frequencies, these maps are finite dimensional symplectomorphisms.

\begin{prop}
\label{prop:preserve_capacity_trunc}
The flow maps $\Phi_N(t)$ preserve symplectic capacities $\textup{cap}(\cO)$ for any domain $\cO \subset \cH^{1/2}(\bT^3)$.
\end{prop}
\begin{proof} 
Let $K$ be large enough so that $\Pi_K P_N = P_N$ and consider the Hamiltonian \eqref{equ:fin_dim_hamil}. We abuse notation slightly and we also use $\Pi_K \textbf{u}_j$ to denote the Fourier projection for the first or second coordinate. For $(\textup{\bf{v}}_1, \textup{\bf{v}}_2) \in \Pi_K \cH^{1/2}(\bT^3)$ and $(\textup{\bf{u}}_1, \textup{\bf{u}}_2)$ which solve \eqref{equ:hamiltonian_wave_fin_dim}, we have
\begin{align*}
& \frac{d}{d\varepsilon} H_N(\Pi_K(\textup{\bf{u}}_1, \textup{\bf{u}}_2) + \varepsilon (\textup{\bf{v}}_1, \textup{\bf{v}}_2))\,\, \Big|_{\varepsilon = 0} \\
& = \frac{d}{d\varepsilon}  \Big[ \frac{1}{2} \int_{\bT^3} ( \Pi_K\textup{\bf{u}}_1 + \varepsilon \textup{\bf{v}}_1)^2 + |\nabla_x  (\Pi_K\textup{\bf{u}}_1 + \varepsilon \textup{\bf{v}}_1)|^2 + ( \Pi_K\textup{\bf{u}}_2 + \varepsilon \textup{\bf{v}}_2)^2 + \frac{1}{4} \int_{\bT^3} (P_N (\textup{\bf{u}}_1 + \varepsilon \textup{\bf{v}}_1))^4 \Big]\Big|_{\varepsilon = 0} \\
& =  \int_{\bT^3} (\Pi_K\textup{\bf{u}}_1) \textup{\bf{v}}_1  + \nabla_x  (\Pi_K\textup{\bf{u}}_1) \nabla_x \textup{\bf{v}}_1 + (\Pi_K\textup{\bf{u}}_2) \textup{\bf{v}}_2+  (P_N \textup{\bf{u}}_1)^3 P_N(\textup{\bf{v}}_1)  
\end{align*}
thus by Lemma \ref{lem:proj_identity} we obtain
\begin{align*}
& \int_{\bT^3} (\Pi_K\textup{\bf{u}}_1) \textup{\bf{v}}_1  + \nabla_x  (\Pi_K\textup{\bf{u}}_1) \nabla_x \textup{\bf{v}}_1 +   (\Pi_K\textup{\bf{u}}_2) \textup{\bf{v}}_2+ P_N (P_N \textup{\bf{u}}_1)^3 \Pi_K \textup{\bf{v}}_1   \\
& =  \int_{\bT^3} \,  - \:\partial_t (\Pi_K\textup{\bf{u}}_2)\, \textup{\bf{v}}_1 +  \partial_t (\Pi_K\textup{\bf{u}}_1) \,\textup{\bf{v}}_2   \\
& = - \omega_{\frac{1}{2}} \Bigl( (\textup{\bf{v}}_1, \textup{\bf{v}}_2), \bigl(\Pi_K (\textup{\bf{u}}_1,   \textup{\bf{u}}_2 )\bigr)_t \Bigr).
\end{align*}
Thus the maps $\Phi_N $ are symplectomorphisms on $\Pi_K \cH^{1/2}$, denote this restricted map by $\Psi_N$. Since the flow decouples low and high frequencies, we can write
\[
\Phi_N(u) = \Psi_N(\Pi_K u) + e^{tJA} \Pi_{\geq K} u = e^{tJA} \bigl( e^{-tJA} \Psi_N(\Pi_K u) +  \Pi_{\geq K} u \bigr)
\]
where $A$ denotes the free evolution of the Klein-Gordon equation on $\cH^{1/2}$. 

\medskip
The invariance of the symplectic capacity under the flow follows from \cite[Lemma 5]{Kuk}, noting that $e^{-tJA} \Psi_N$ is also a symplectomorphism and that $e^{tJA}$ preserves admissible and fast functions, see also \cite[Theorem 3]{Kuk}.
\end{proof}

\medskip

\subsection{Definition and properties of $\Sigma_\lambda$}
\label{sec:sigma_props}
We recall the definition of $\Sigma$. We let $0 < \gamma < \frac{1}{2}$ to be fixed later and define
\begin{equation}
\begin{split}
\Theta_1 &:= \bigl\{(u_0, u_1) \in \cH^{1/2} \,:\, \|S(t)(1 - \Delta)^{\gamma / 2} (u_0, u_1) \|^6_{L_x^6(\bT^3)}  \in L_{\text{loc}}^1(\bR_t)  \bigr\} \\
\Theta_2 &:= \bigl\{(u_0, u_1) \in \cH^{1/2} \,:\ \|S(t)(u_0, u_1) \|_{L_x^\infty(\bT^3)} \in L_{\text{loc}}^1(\bR_t)  \bigr\}.
\end{split}
\end{equation}
Set $\Theta := \Theta_1 \cap \Theta_2$ and let $\Sigma = \Theta + \cH^1$.
\medskip
For $(u_0, u_1) \in \Sigma$ we denote by $u(t) = S(t)(u_0, u_1) + w(t)$ the global solution to \eqref{equ:cubic_nlkg}. Fix $\varepsilon > 0$ and let $C > 0$ be as in the statements of Proposition \ref{prop:unif_global2} and Proposition \ref{prop:unif_trunc}. Define the subsets
\begin{align*}
E_\lambda &:= \bigl\{(u_0, u_1) \in \Sigma \,:\, \|(w(t), \partial_t w(t)) \|_{\cH^1} \leq C\bigl(\lambda + |t|\bigr)^{1 + \varepsilon}\bigr\},\\
H_\lambda &:= \bigl\{(u_0, u_1) \in \Sigma \,:\, \|(w_N(t), \partial_t w_N(t)) \|_{\cH^1} \leq C\bigl(\lambda + |t|\bigr)^{1 + \varepsilon}\bigr \}, \\
J_\lambda &:= \bigl\{(u_0, u_1) \in \Sigma \,:\, \|u \|_{L_x^4} \leq C\bigl(\lambda + |t|\bigr)^{1 + \varepsilon}\bigr\}, \\
K_\lambda &:=\bigl\{(u_0, u_1) \in \Sigma \,:\, \|u_N \|_{L_x^4}\leq C\bigl(\lambda + |t|\bigr)^{1 + \varepsilon}\bigr\},
\end{align*}
and let $M_\lambda$ be the set of $(u_0, u_1) \in \Sigma$ such that for $\gamma$ as above we can find $C = C(T)$ as in Corollary \ref{cor:asbdd} such that
 \begin{align}
\label{equ:lp_bds}
\|(1 - \Delta)^{\gamma /2 }S(t)(u_0, u_1)\|_{L_{x,t}^6([0, T] \times \bT^3)} \leq   C \lambda.
\end{align}
Setting
 \begin{align}
\label{equ:sigma_lambda}
\Sigma_\lambda := E_\lambda \cap H_\lambda \cap J_\lambda \cap K_\lambda \cap M_\lambda
\end{align}
we have the following bounds for the measure of this set. The choice of $\gamma > 0$ does not affect the following proposition.

\begin{prop}
\label{prop:lambda_props}
Fix $\mu \in \cM^{1/2}$ and let $\Sigma_\lambda$ be as defined in \eqref{equ:sigma_lambda}. Then there exists $C, c, \theta > 0$ such that for all $\lambda > 0$ we have
 \begin{align}
\label{equ:lambda_meas}
\mu(\Sigma_\lambda) \geq 1 - C e^{ -c \lambda^\theta}.
\end{align}
\end{prop}
\begin{proof}
Suppose that $\mu$ is induced by the randomization of $(f_0, f_1) \in \cH^{1/2}$. By Proposition \ref{prop:unif_global2} and Proposition \ref{prop:unif_trunc}, there exists $C, c > 0$ such that
\[
\mu(E_\lambda \cap H_\lambda \cap J_\lambda \cap K_\lambda ) \geq 1 - C e^{-c \lambda^\theta},
\]
and by Corollaries \ref{cor:averaging_effects} and \ref{cor:asbdd},
\begin{align*}
\mu(M_\lambda ) \geq 1- C e^{ - c \lambda^2 / \| (f_0, f_1) \|_{\cH^{s}}^2}.
\end{align*}
Taking intersections and using that these bounds are exponential yields \eqref{equ:lambda_meas}.
\end{proof}

\begin{rmk}
In the above proof we are implicitly taking advantage of the fact that we have exponential bounds on the measure of the sets in question.
\end{rmk}

\section{Multilinear estimates in adapted spaces}
\label{sec:multi}

Multilinear estimates for nonlinear Klein-Gordon equations in the form we need for well-posedness were proven in \cite{Schottdorf},  where more general non-linearities were treated, but we state only the estimates we require. The statements are slightly modified for our setting and we include the proofs for completeness. Ultimately, however, we will need slightly stronger estimates for the stability theory, which we prove in Proposition \ref{prop:multi_refine}. We begin this section with a few preliminary estimates.

\medskip
The following lemma is an immediate consequence of the atomic structure of $U^2$, see, for instance \cite[(17)]{Schottdorf} for the computation.

\begin{lem}[Linear solutions lie in $\cX^s$]
\label{lem:lin_est}
Let $s \geq 0$, $0 < T \leq \infty$ and $u_0^{\pm} \in H^s(\mathbb{T}^3)$. Then
\[
\|e^{\pm i t  \langle \nabla \rangle} u_0^{\pm} \|_{\cX_\pm^s([0,T))} \leq \| u_0^{\pm} \|_{H^s(\bT^3)}.
\]
\end{lem}

We also have the following duality estimate. The proof of this result is a straightforward adaptation of the proof of  Proposition 2.11 in \cite{HTT}.

\begin{prop}
\label{prop:dual}
Let $s \geq 0$ and $T > 0$. For $f \in L_t^1 H^s([0,T) \times \bT^3)$ we have
\[
\|I^{\pm}(f) \|_{\cX_\pm^{s}} \leq \sup_{w \in Y_\mp^{1- s} \,: \, \|w\|_{Y_\mp^{1- s}} = 1} \left| \int_0^T \int_{\bT^3} f(t,x) \overline{v(t,x)} dt dx \right|.
\]
\end{prop}

\begin{proof} 
We let $(a_n)_{n \in \bZ^3} \in \ell^2(\bZ^3)$ be with $\|(a_n)\|_{\ell^2} = 1$ such that
\[
\|I^\pm(f)\|_{\cX_\pm^{s}} \leq \sum_{n \in \bZ^3} a_n\, \langle n \rangle^{s} \left\| \int_0^t \chi_{[0,T)}  \, e^{\pm i (t -s) \langle n \rangle} \frac{f(t,x)}{2 \langle \nabla \rangle} dt \right\|_{U_{\pm}^2} + \varepsilon.
\]
We use the definition of $U^2_\pm$, and we use duality from Theorem \ref{thm:duality} and Proposition \ref{prop:duality2} to estimate each piece by
\[
\left\| \int_0^t \chi_{[0,T)}  \, e^{\mp i s \langle n \rangle} \frac{f(t,x)}{2 \langle \nabla \rangle} dt \right\|_{U^2} \leq \left| \int_0^T \widehat{f(s)}(n) \, \overline{ \frac{ e^{\pm i s \langle n \rangle} v_n(s)}{2 \langle n \rangle} }ds \right| + 2^{-|n|^2} \varepsilon
\]
for a sequence $v_n \in V_{rc}^2$ with $\|v_n\|_{V_{rc}^2} = 1$ supported on $[0, T)$. We then define
\[
v(t,x) \simeq \sum_{n \in \bZ^3} \Bigl( a_n \langle n \rangle^{s-1} e^{\pm it \langle n \rangle} \, v_n(t) \Bigr) e^{i x \cdot n },
\]
and we observe that $v \in Y_{\mp}^{1-s}([0,T))$ with $\|v\|_{Y_{\mp}^{1-s}} \leq 1$. Hence
\[
\|I^{\pm}(f) \|_{\cX_\pm^{s}} \leq \sum_{n \in \bZ^3} \left| \int_0^T \widehat{f(t)}(n) \overline{\widehat{v(t)}(n)} \right| + c \,\varepsilon,
\]
and the claim follows by dominated convergence theorem and Plancherel.
\end{proof}

The following proposition demonstrates that a priori bounds on the Strichartz norm control the norm of solutions in the adapted function spaces.

\begin{prop}
\label{prop:xs_solN_1ds}
Let $u$ be a solution to the cubic nonlinear Klein-Gordon equation for initial data $(u_0, u_1) \in \textbf{B}_R \subset \cH^{1/2}$ which satisfies
\begin{align}
\label{equ:uniform_l4}
\|u\|_{L_{t,x}^4(([0,T) \times \bT^3)} \leq K.
\end{align}
Then $\|u\|_{\cX_T^{1/2}} \lesssim C(K, R)$.
\end{prop}
\begin{proof}
Let $u$ solve the cubic nonlinear Klein-Gordon equation \eqref{equ:soln_kg_sys} and suppose that $u$ satisfies the uniform bound \eqref{equ:uniform_l4}. Fix $T > 0$ then by Lemma \ref{lem:lin_est} we estimate
\[
\|u\|_{\cX_T^{1/2}} \lesssim \|(u_0, u_1)\|_{\cH_x^{1/2}} + \|I(F) \|_{\cX_T^{1/2}}, \qquad F(u) = u^3 .
\]
We expand the nonlinear term and we deal with $I^+$ as the other term is handled analogously. By Proposition \ref{prop:dual} and H\"older's inequality
\begin{align}
\| I^+(F) \|_{\cX_+^{1/2}}  &\leq  \sup_{v \in \cY_-^{1/2} \,: \, \|v\|_{\cY_-^{1/2}} = 1} \left| \int_0^T \int_{\bT^3} F(t,x) \overline{v(t,x)} dt dx \right|\\
& =  \sup_{v \in \cY_-^{1/2} \,: \, \|v\|_{\cY_-^{1/2}} = 1} \sum_{N} \left| \int_0^T \int_{\bT^3} P_N F\,\cdot \, \overline{P_N v(t,x)} dt dx \right|\\
& \leq \sup_{v \in \cY_-^{1/2} \,: \, \|v\|_{\cY_-^{1/2}} = 1} \sum_N  \| P_N F \|_{L_{t,x}^{4/3}} \, \| P_N v(t,x) \|_{L_{t,x}^4}.
\end{align}
By complex interpolation, we have the dual square function type inequality
\[
\biggl( \sum_{N}  \|P_{N} F \|^2_{L_{t,x}^{4/3}} \biggr)^{1/2} \lesssim \|F\|_{L_{t,x}^{4/3}}.
\]
Applying Cauchy-Schwarz, and noting that Remark \ref{rmk:duality_improv} applies to $I^+$, we use Corollary \ref{cor:strichartz} part $(ii)$ and Corollary \ref{cor:orthog}, to obtain
\begin{align}
&\sup_{v \in \cY_-^{1/2} \,: \, \|v\|_{\cY_-^{1/2}} = 1} \biggl( \sum_{N}  \|P_{N} f \|^2_{L_{t,x}^{4/3}} \biggr)^{1/2} \|v\|_{\cY_-^{1/2}} \lesssim \| f \|_{L_{t,x}^{4/3}}  \leq \| u \|^3_{L_{t,x}^{4}},
\end{align}
which yields the result.
\end{proof}

We are now ready to state our first multilinear estimate.

\begin{thm}[Theorem 3, \cite{Schottdorf}]
\label{thm:multilinear_ests}
Suppose that the signs $\pm_i$ $(i=0, 1,2,3)$ are arbitrary and $H \sim H'$. Then
\begin{align}
\label{equ:multi1}
\frac{1}{H} \Bigl| \sum_{L_i \lesssim H} \iint  \prod_{i=1}^2  u_{L_i} u_{H'} u_{H} dx dt \Bigr| \lesssim \prod_{i=1}^2 \Biggl( \sum_{L_i\lesssim H} L_i \|u_{L_i}\|_{V_{\pm_i}^2}^2 \Biggr)^{1/2} \|u_{H'}\|_{V_{\pm_3}^2} \|v_{H}\|_{V_{\pm_{0}}^2}
\end{align}
and 
\begin{align}
\label{equ:multi2}
& \sup_{\|w \|_{\cY_\pm^{1/2}} = 1} \sum_{L \lesssim H} \Bigl| \sum_{L_1 \lesssim H} \iint u_{L_1} u_{H'} u_{H} v_{L} dx dt \Bigr| \lesssim   (H H')^{1/2} \|u_{1}\|_{\cY_{\pm_1}^{1/2}} \|u_{H'}\|_{V_{\pm_2}^2} \|u_{H}\|_{V_{\pm_{3}}^2}. \qquad
\end{align}

\end{thm}

\begin{proof}
We estimate
\begin{align}
\frac{1}{H} \Bigl| \iint \prod_{i=1}^2 u_{L_i} u_{H'} v_{H} dx dt \Bigr| &\lesssim \frac{1}{H} \|u_{L_1} u_{H'} \|_{L_{t,x}^2} \|u_{L_2} v_{H} \|_{L_{t,x}^2} \\
& \lesssim \frac{1}{H} \left( \frac{H^2}{L_1 L_2} \right)^\delta L_1 L_2 \,\|u_{L_1} \|_{V_{\pm_1}^2}  \|u_{L_2} \|_{V_{\pm_2}^2} \|u_{H'} \|_{V_{\pm_3}^2}  \|v_{H} \|_{V_{\pm_0}^2}  
\end{align}
where we used the improved bilinear estimates from Proposition \ref{prop:multilinear} and Remark \ref{rmk:log_loss} in each term. Summing over $L_i \lesssim H$, we use Cauchy-Schwarz on the terms
\[
 L_i^{1-\delta} \|u_{L_i} \|_{V_{\pm_i}^2} = L_i^{\frac{1}{2} - \delta}  L_i^{\frac{1}{2}} \|u_{L_i}\|_{V_{\pm_i}^2}
\]
which yield the $i=1,2$ factors in \eqref{equ:multi1}. Ultimately, it suffices to bound
\begin{align}
H^{-1 + 2 \delta} \Bigl( \sum_{L_1 \leq H} \sum_{L_2 \leq L_1} L_2^{1 - 2\delta} L_1^{1 - 2\delta}  \Bigr)^{1/2} \lesssim H^{-1 + 2 \delta} \Bigl( \sum_{L_1 \leq H} L_1^{2-4\delta} \Bigr)^{1/2} \lesssim 1,
\end{align}
and since $1-2\delta > 0$ for $\delta$ sufficiently small, this yields \eqref{equ:multi1}. 

\medskip
The second estimate is treated similarly, with the roles of $u_{L_2}$ and $v_L$ swapped. That is, we bound
\begin{align}
\Bigl|  \iint u_{L_1} u_{H'} u_{H} v_{L} dx dt \Bigr| &\lesssim  \|u_{L_1} u_{H} \|_{L_{t,x}^2} \|u_{H'} v_{L} \|_{L_{t,x}^2} \\
& \lesssim  \left( \frac{H^2}{L_1 L} \right)^\delta L_1 L\, \|u_{L_1} \|_{V_{\pm_1}^2}   \|u_{H} \|_{V_{\pm_2}^2} \|u_{H'} \|_{V_{\pm_3}^2} \|v_{L} \|_{V_{\pm_0}^2}.
\end{align}
Collecting terms, summing in $L_1, L \lesssim H$, and applying Cauchy-Schwarz yields
\begin{align}
\eqref{equ:multi2} &\lesssim H^{2 \delta} \Bigl( \sum_{L \lesssim H}  \sum_{L_1 \lesssim H} L^{1- 2 \delta} L_1^{1 -  2\delta} \Bigr)^{1/2} \|u_1\|_{\cY_{\pm_1}^{1/2}}  \|u_{H} \|_{V_{\pm_2}^2} \|u_{H'} \|_{V_{\pm_3}^2} \\
& \lesssim (H H')^\frac{1}{2}   \|u_1\|_{\cY_{\pm_1}^{1/2}} \|u_{H} \|_{V_{\pm_2}^2} \|u_{H'} \|_{V_{\pm_3}^2},
\end{align}
as required.
\end{proof}

We will now briefly review the well-posedness theory for \eqref{equ:soln_kg_sys} in the adapted function spaces. We need to estimate the cubic nonlinearity $F(u) = u^3$ for $u = u^+ + u^-$. Hence we can decompose
\[
F(u) = \sum_{\{i, j, k\} \in \{1,2,3\}} u^{(i)}\,u^{(j)}\,u^{(k)}
\]
for $u^{(i)} = u^\pm$. Thus, it suffices to estimate these eight cubic terms, which we do in the sequel. Due to finite speed of propagation, the arguments from \cite{Schottdorf} apply if one allows implicit constants to depend on the time interval.

\begin{thm}[Theorem 4, \cite{Schottdorf}]
\label{thm:nonlinear_est}
Fix $T > 0$. There exists a constant $C$ depending only on $T > 0$ such that
\begin{align}
\label{equ:ys_est}
\|I(u^{(1)}, u^{(2)}, u^{(3)})\|_{\cX_T^{1/2}} \leq C \prod_{i=1}^3 \|u^{(i)}\|_{\cY_T^{1/2}}.
\end{align}
Since $\cX^s \hookrightarrow \cY^s$, this implies
\[
\|I(u^{(1)}, u^{(2)}, u^{(3)})\|_{\cX_T^{1/2}} \leq C \prod_{i=1}^3 \|u^{(i)}\|_{\cX_T^{1/2}}.
\]
\end{thm}

\begin{rmk}
The proof works similarly for any $s \geq 1/2$ with the obvious modifications, however we omit this generalization for simplicity of presentation. The wellposedness of \eqref{equ:soln_kg_sys} follows from these estimates via a straightforward contraction mapping argument.
\end{rmk}

\begin{proof}
We only treat $I^+$ as $I^-$ follows analogously, and $T > 0$ will be fixed throughout. In the usual manner, we take extensions of the $u^{(i)}$ to $\bR$, which we still denote by $u^{(i)}$, and \eqref{equ:ys_est} follows by taking infimums over all such extensions. We also suppress the notation $\pm_i$ on each function. We do not repeat these considerations. By Proposition \ref{prop:dual}, we estimate
\begin{align}
\left\| I^+(u^{(1)}, u^{(2)}, u^{(3)})  \right\|_{\cX_+^{1/2}} \leq  \sup_{\|w\|_{\cY_-^{1/2}} = 1} \sum_{N_0} \, \Biggl| \sum_{N_3+ N_2 + N_1 = N_0} \iint u_{N_3} u_{N_2} u_{N_1} v_{N_0} dx dt \,\, \Biggr|.
\end{align}
The convolution requirement, implies that the right-hand side above vanishes unless $N_i \sim N_j$ for some $i \neq j$, and hence we may assume without loss of generality that
\[
N_3 \lesssim N_2 \lesssim N_1, \qquad N_1 \sim \max\{N_0, N_2\}.
\]
In the case that $N_0 \sim N_1$, we use \eqref{equ:multi1}, and we bound
\begin{align}
&\Biggl| \sum_{N_3+ N_2 + N_1 = N_0} \iint u_{N_3} u_{N_2} u_{N_1} v_{N_0} dx dt \,\, \Biggr| \\
& \lesssim \sum_{N_0 \sim N_1} \prod_{i=2}^3 \Biggl( \sum_{N_i\lesssim N_0} N_i \|u_{N_i}\|_{V_{\pm_i}^2}^2 \Biggl)^{1/2} N_0\,  \|u_{N_1}\|_{V_{\pm_1}^2} \|v_{N_0}\|_{V_{\pm_0}^2} .
\end{align}
By Corollary \ref{cor:orthog} and Cauchy-Schwarz in $N_0 \sim N_1$,
\begin{align}
\left\| I^+(u^{(1)}, u^{(2)}, u^{(3)})  \right\|_{\cX_+^{1/2}} 
& \lesssim \|u^{(1)}\|_{\cY_{\pm_1}^{1/2}} \|u^{(2)}\|_{\cY_{\pm_2}^{1/2}} \|u^{(3)}\|_{\cY_{\pm_3}^{1/2}}.
\end{align}
When $N_2 \sim N_1$, we use \eqref{equ:multi2} and we can similarly bound
\begin{align}
\left\| I^+(u^{(1)}, u^{(2)}, u^{(3)})  \right\|_{\cX_+^{1/2}}  &\leq  \sup_{\|w\|_{\cY_-^{1/2}} = 1} \,\,  \Biggl| 
\sum_{N_3+ N_2 + N_1 = N_0} \iint u_{N_3} u_{N_2} u_{N_1} v_{N_0} dx dt \,\, \Biggr| \\
& \lesssim \|u_{3}\|_{\cY_{\pm_3}^{1/2}} \sum_{N_1 \sim N_2}  (N_1 N_2)^{1/2}   \|u_{N_1}\|_{V_{\pm_{1}}^2} \|u_{N_2}\|_{V_{\pm_2}^2}.
\end{align}
Once again, we obtain the desired bound using Corollary \ref{cor:orthog} and Cauchy-Schwarz.
\end{proof}

Finally, we need to following refinement to Theorem \ref{thm:nonlinear_est} which says that we can replace some of the factors with the mixed norm 
\[
\|u^\pm\|_{Z_\pm'(I)} =  \|u^\pm\|^{3/4}_{Z(I)} \|u^\pm\|^{1/4}_{\cX_\pm^{1/2}(I)},
\]
for $Z$ the norm defined in \eqref{equ:zs_norm}, and we set $Z' = Z_+ \times Z_-$.

\medskip
As in the proof of Theorem \ref{thm:nonlinear_est}, we need to consider cubic expressions in $u^{+}$ and $u^-$. It will be clear from the proof that we could have, instead, relied only on bounding $u^+ + u^-$ in  the $Z(I)$ norm, which is more consistent with the $L_{t,x}^4$ Strichartz norm from the standard well-posedness theory. Ultimately, however, such considerations do not affect our arguments given Proposition \ref{prop:xs_solN_1ds} which allows us to control the $Z(I)$ norm of solutions to \eqref{equ:soln_kg_sys} via Strichartz bounds.

\begin{prop}
\label{prop:multi_refine}
Under the above assumptions,
\[
\left\| I(u^{(1)}, u^{(2)}, u^{(3)})  \right\|_{X^{1/2}(I)} \lesssim \sum_{\{i, j , k\}\, \in\, \{1,2,3\}} \|u^{(i)}\|_{\cY^{1/2}(I)} \|u^{(j)}\|_{Z'(I)} \|u^{(k)}\|_{Z'(I)}.
\]
\end{prop}

\begin{proof}
We fix $T > 0$ and we only treat the $I^+$ term. By Proposition \ref{prop:dual}, we estimate
\begin{align}
\left\| I^+(u^{(1)}, u^{(2)}, u^{(3)})  \right\|_{\cX_+^{1/2}} \leq  \sup_{\|w\|_{\cY_-^{1/2}} = 1} \Biggl| \sum_{N_3+ N_2 + N_1 = N_0} \iint u_{N_3} u_{N_2} u_{N_1} v_{N_0} dx dt \,\, \Biggr|.
\end{align}
The convolution requirement implies that the right-hand side above vanishes unless $N_i \sim N_j$ for some $i \neq j$, and hence we may without loss of generality assume that
\[
N_3 \lesssim N_2 \lesssim N_1, \qquad N_1 \sim \max\{N_0, N_2\}.
\]
In the case that $N_0 \sim N_1$, we apply H\"older's inequality and obtain
\begin{align}
&\Biggl| \sum_{N_3+ N_2 + N_1 = N_0} \iint u_{N_3} u_{N_2} u_{N_1} v_{N_0} dx dt \,\, \Biggr| \lesssim \sum_{N_3+ N_2 + N_1 = N_0} \|u_{N_3} u_{N_1} \|_{L_{t,x}^2} \|u_{N_2} v_{N_0} \|_{L_{t,x}^2} .
\end{align}
Let us consider the first term. Let $C$ be a cube of size $N_3$ centered in frequency space at $\xi_0 \in \bZ^3$ with $|\xi_0| \sim N_1$ and let $P_C$ denote the (sharp) Fourier projection onto this cube. Since the spatial Fourier support of $(P_C u_{N_1}) u_{N_3}$ is contained in a fixed dilate of $C$,
\[
 \|u_{N_3} u_{N_1} \|_{L_{t,x}^2} \lesssim \bigg( \sum_{C}  \| (P_{C} u_{N_1}) u_{N_3} \|_{L_{t,x}^2}^2\bigg)^{1/2}
\]
and hence by H\"older's inequality and Proposition \ref{prop:l4_strichartz} on the term with $P_{C} u_{N_1}$, we can bound
\begin{align}
 \|u_{N_3} u_{N_1} \|^2_{L_{t,x}^2} &\lesssim  (N_3 N_1)^{\frac{1}{2}}  \|u_{N_3} \|^2_{L_{t,x}^4(I \times \bT^3)}  \sum_{C}  \| P_{C} u_{N_1} \|^2_{V_{\pm_1}^2}.
\end{align}
Thus for fixed $N_1$, we obtain
\begin{align}
\sum_{N_3 \lesssim N_1} \|u_{N_3} u_{N_1} \|_{L_{t,x}^2} &\lesssim \sum_{N_3 \lesssim N_1} (N_3 N_1)^{\frac{1}{4}} \|u_{N_3} \|_{L_{t,x}^4(I \times \bT^3)}  \bigg(\sum_{C}  \| P_{C} u_{N_1} \|^2_{V_{\pm_1}^2} \bigg)^{1/2} \\
&\lesssim \sum_{N_3 \lesssim N_1}  \left(\frac{ N_3 }{ N_1 } \right)^{\frac{1}{4} } \|u_{N_3} \|_{L_{t,x}^4(I \times \bT^3)}    \bigg(\sum_{C} N_1 \| P_{C} u_{N_1} \|^2_{V_{\pm_3}^2} \bigg)^{1/2},
\end{align}
hence by resumming over $C$ and using Cauchy-Schwarz in $N_3$, we can bound
\begin{align}
\sum_{N_3 \lesssim N_1} \|u_{N_3} u_{N_1} \|_{L_{t,x}^2} &  \lesssim \sum_{N_3 \lesssim N_1}  \left(\frac{ N_3 }{ N_1 } \right)^{\frac{1}{4} } \|u_{N_3} \|_{L_{t,x}^4(I \times \bT^3)} \| u_{N_1} \|_{\cY_{\pm_1}^{1/2}} \\
&\lesssim \bigg( \sum_{N_3}   \|u_{N_3} \|^4_{L_{t,x}^2(I \times \bT^3)} \bigg)^{1/2}  \| u_{N_1} \|_{\cY_\pm^{1/2}},
\end{align}
which yields
\[
\sum_{N_3 \lesssim N_1} \|u_{N_3} u_{N_1} \|_{L_{t,x}^2} \lesssim \|u_1\|_{Z}\| u_{N_1} \|_{\cY_{\pm_1}^{1/2}}.
\]
We perform the same analysis on the second term. By Cauchy-Schwarz in $N_0 \sim N_1$ and symmetrizing we obtain 
\[
\left\| I(u^{(1)}, u^{(2)}, u^{(3)})  \right\|_{X^{1/2}(I)} \lesssim \sum_{\{i, j , k\}\, \in\, \{1,2,3\}} \|u^{(i)}\|_{X^{1/2}(I)} \|u^{(j)}\|_{Z(I)} \|u^{(k)}\|_{Z(I)}.
\]
Using the definition of the $Z'(I)$ norm, and combining this bound with the estimates from Theorem \ref{thm:nonlinear_est} yields the result in this case.

The case when $N_1 \sim N_2$ and $N_0 \lesssim N_2$ requires a bit more care. We estimate
\begin{align}
&\Biggl| \sum_{N_3+ N_2 + N_1 = N_0} \iint u_{N_3} u_{N_2} u_{N_1} v_{N_0} dx dt \,\, \Biggr| \lesssim \sum_{N_3+ N_2 + N_1 = N_0} \|u_{N_3} u_{N_1} \|_{L_{t,x}^2} \|u_{N_2} v_{N_0} \|_{L_{t,x}^2} .
\end{align}
As above, for the first term, we obtain
\begin{align}
 \sum_{N_3 \lesssim N_1 } \|u_{N_3} u_{N_1} \|_{L_{t,x}^2}  &\lesssim \sum_{N_3 \lesssim N_1}  \left(\frac{ N_3 }{ N_1 } \right)^{\frac{1}{4} } \|u_{N_3} \|_{L_{t,x}^4(I \times \bT^3)}    \biggl(\sum_{C_1} N_1 \| P_{C_1} u_{N_1} \|^2_{V_{\pm_1}^2} \biggr)^{1/2}\\
 &\lesssim \|  u_{N_1} \|_{\cY_{\pm_1}^{1/2}} \sum_{N_3 \lesssim N_1}  \left(\frac{ N_3 }{ N_1 } \right)^{\frac{1}{4} } \|u_{N_3} \|_{L_{t,x}^4(I \times \bT^3)}.
\end{align}
For the second term we use H\"older's inequality to bound
\begin{align}
 \sum_{N_0 \lesssim N_2 } \|u_{N_2} v_{N_0} \|_{L_{t,x}^2} & =  \sum_{N_0 \lesssim N_2 }  \|u_{N_2} v_{N_0} \|^{1/4}_{L_{t,x}^2} \|u_{N_2} v_{N_0} \|^{3/4}_{L_{t,x}^2} 
 \end{align}
 
\noindent and we use Propositions \ref{prop:multilinear} and \ref{prop:l4_strichartz} to bound this by
 \begin{align}
 & \sum_{N_0 \lesssim N_2 }  (N_0)^{1/4} \,  \left(\frac{N_2}{N_0} \right)^\delta    \|u_{N_2} \|^{1/4}_{V_{\pm_2}^2}\| v_{N_0} \|^{1/4}_{V_{\pm_0}^2 }\|u_{N_2}\|^{3/4}_{L_{t,x}^4}  N_0^{3/8} \| v_{N_0} \|^{3/4}_{V_{\pm_0}^2 } \\
 & \lesssim \sum_{N_0 \lesssim N_2 }   \,  \left(\frac{N_0}{N_2} \right)^{\frac{1}{8} - \delta}    (N_2)^{1/8} \|u_{N_2} \|^{1/4}_{V_{\pm_2}^2} \|u_{N_2}\|^{3/4}_{L_{t,x}^4}  \,N_0^{1/2} \| v_{N_0} \|_{V_{\pm_0}^2 }.
\end{align}
We have split this term using H\"older's inequality in order to gain some term which enables us to sum in $N_0$ without loss. Using that
\[
\sum_{N_0 \lesssim N_2} \left(\frac{N_0}{N_2} \right)^{\frac{1}{8} - \delta}N_0^{1/2} \| v_{N_0} \|_{V_{\pm}^2 } \lesssim \|w\|_{\cY_\pm^{1/2}}
\]
and
\[
\sum_{N_3 \lesssim N_1}  \left(\frac{ N_3 }{ N_1 } \right)^{\frac{1}{4} } \|u_{N_3} \|_{L_{t,x}^4(I \times \bT^3)} \lesssim \|u_3\|_{Z(I)},
\]
we are left with
\[
\|u_{N_1} \|_{\cY_{\pm_1}^{1/2}} \sum_{N_3 \sim N_2}   (N_2)^{1/8} \|u_{N_2} \|^{1/4}_{V_{\pm}^2} \,\|u_{N_2}\|^{3/4}_{L_{t,x}^4} \, ,
\]
and we once again conclude by Cauchy-Schwarz with $\frac{1}{8} + \frac{3}{8} + \frac{1}{2} = 1$ and Corollary \ref{cor:orthog}.
\end{proof}

\section{Stability theory in adapted function spaces}
\label{sec:stab}
In this section, we prove the necessary stability theory for the nonlinear Klein-Gordon equation in the adapted function spaces. As discussed in the introduction, the key difficulty in proving a satisfactory stability theory in this setting is that even if $X^{1/2}(I)$ is bounded, we cannot isolate a small interval on which the norm is small. Ultimately, however, using the intermediate $Z'(I)$ norm we are able to recover the desired stability theory.  We record the following results which are, for the most part, straightforward adaptations of the analogous results for the nonlinear Schr\"odinger equation from \cite[Section 3]{IP12}.

\begin{prop}
\label{equ:local_criteria}
Suppose that $R > 0$ is fixed and let $u_0 = (u_0^+, u_0^-)$ with $\|u^\pm_0 \|_{H^{1/2}(I \times \bT^3)} \leq R$. Then there exists $\delta_0 = \delta_0(R) > 0$ such that if 
\[
\|e^{\pm i t \langle \nabla \rangle} u^\pm_0 \|_{Z_\pm'(I)} < \delta
\]
for some $\delta \leq \delta_0$, on some interval $I$ with $0 \in I$ and $|I| \leq 1$, then there exists a strong solution to \eqref{equ:soln_kg_sys} in $X^{1/2}(I)$ with initial data $u(0) = u_0$ and
\begin{align}
\label{equ:nonliN_1ound}
\| u^\pm - e^{\pm i t \langle \nabla \rangle} u^\pm_0 \|_{\cX_\pm^{1/2}(I)} \leq \delta^{\frac{5}{3}}.
\end{align}
\end{prop}

\begin{rmk}
The choice $5/3$ is arbitrary and in fact the statement holds for any $\alpha$ with $1 < \alpha < 2$, and we merely require $\alpha > 1$ for our applications. This proposition can be thought of as an version of a small data result in the adapted function spaces which does not require that the initial data be small in the $H^{1/2}$ norm.
\end{rmk}

\begin{proof}
The statement about existence follows from a standard fixed point argument. Indeed, Let $R, a > 0$ and consider
\[
\cS = \{ u  \in X^{1/2}(I) \,:\, \|u\|_{X^{1/2}(I)} \leq 4R, \,\, \|u\|_{Z'(I)} \leq 2a\}
\]
and the mapping
\[
\Phi^\pm (u) = e^{ \pm i t \langle \nabla \rangle} u_0^\pm \pm i I^\pm(F(u)).
\]
By Proposition \ref{prop:multi_refine}, 
\begin{align}
&\|\Phi^\pm (u)\|_{\cX_\pm^{1/2}(I)} \leq \| u^\pm_0 \|_{H^{1/2}} + C R a^2 \leq R + CR a^2\\
&\|\Phi^\pm (u) \|_{Z_\pm'(I)} \leq \| e^{\pm i t \langle \nabla \rangle} u^\pm_0 \|_{Z_\pm'(I)} + C R a^2 \leq \delta + CRa^2 ,
\end{align}
and similarly for the difference expression. Choosing $a = 2\delta$ for $0 < \delta \leq \delta_0$ and $\delta_0 = \delta_0(R)$ small enough so that $4CR\delta < 1$, we find that $\Phi = \Phi^+ + \Phi^-$ posesses a unique fixed point $u$ in $\cS$. Finally, to prove \eqref{equ:nonliN_1ound}, we obtain by another application of Proposition \ref{prop:multi_refine} that
\[
\| u^\pm - e^{\pm i t \langle \nabla \rangle} u^\pm_0 \|_{\cX_\pm^{1/2}(I)} \lesssim R \delta^2,
\]
hence taking $\delta_0$ even smaller if necessary, we obtain the statement.
\end{proof}

In light of the fact that the $Z'$ norm controls the $L_{t,x}^4$ norm of solutions, the standard blow-up criterion in Strichartz spaces together with Proposition \ref{prop:xs_solN_1ds} imply that this weaker norm controls the global existence theory.

\medskip
For simplicity, we define the norm
\begin{align}
\label{equ:norm_nonlin}
\|h\|_{N_{\pm} (I)} = \left\| \int_a^t e^{ \pm i t \langle \nabla \rangle} \frac{h(s)}{2 \langle \nabla \rangle} ds \right\|_{\cX_\pm^{1/2}(I)}.
\end{align}

The following is the main result of this section.

\begin{prop}
\label{prop:xs_stability}
Let $I \subset \bR$ a compact time interval and $t_0 \in I$. Let $v$ be a solution defined on $I \times \bT^3$ of the Cauchy problem
\begin{equation*}
\left\{ 
\begin{split}
&v_{tt} - \Delta v + v + F(v) = e \\
&\bigl(v, \partial_tv \bigr)\big|_{t=t_0} = (v_0, v_1) \in  \cH^{1/2}(\bT^3),
\end{split}
\right.
\end{equation*}
and identify the solution $v$ with $(v^+, v^-)$. Suppose that 
\begin{align}
\label{equ:z_bds}
\|v^\pm\|_{Z(I)} + \|v^\pm\|_{L_t^\infty H_x^{1/2}(I \times \bT^3)} \leq K.
\end{align}
Let $(u, \partial_t u) \big|_{t=t_0} = (u_0, u_1) \in \cH^{1/2}(\bT^3)$ and suppose we have the smallness condition 
\begin{align}
\label{equ:small_diff}
\|(v_0 - u_0, v_1 - u_1)\|_{ \cH^{1/2}(\bT^3)} + \|e\|_{N_\pm(I)} \leq \varepsilon < \varepsilon_1  
\end{align}
for some $0 < \varepsilon < \varepsilon_1$ where $\varepsilon_1 \leq 1$ is a small constant $\varepsilon_1 = \varepsilon_1(K, I) > 0$. Then there exists a unique solution $(u, \partial_t u)$ to the cubic nonlinear Klein-Gordon equation on $I \times \bT^3$ with initial data $(u_0, u_1)$ at time $t_0$ and $C \equiv C(K, I) \geq 1$ which satisfies $\|v - u\|_{X^{1/2}(I)} \leq C \,\varepsilon$.
\end{prop}

\medskip
\begin{rmk}
\label{rmk:bound_requirements}
In particular, \eqref{equ:z_bds} holds if we have $X^{1/2}(I)$ bounds on the solution $v$, and consequently, by Proposition \ref{prop:xs_solN_1ds}, if we have $L_{t,x}^4(I \times \bT^3)$ bounds. Additionally, the computations in Proposition \ref{prop:xs_solN_1ds} also imply that $L_{t,x}^{4/3}(I \times \bT^3)$ bounds on the error imply the $N_\pm(I)$ bounds on the error in \eqref{equ:small_diff}. Hence Proposition \ref{prop:xs_stability} can be seen as a refined version of the long-time stability theory from Appendix \ref{ap:pert}.
\end{rmk}

\medskip
We include a proof of this fact for completeness, although it follows almost identically to the corresponding statement for the NLS in \cite[Section 3]{IP12}. The main idea is to mimic the proof of the standard Strichartz space stability, exploiting the extra properties of the $Z'(I)$ norm. Roughly speaking, we will work on small intervals where $\|v\|_{Z'(I_k)}$ is sufficiently small and then, in spirit, the computations which yield the standard stability theory yield the result. We only have to check that at each step we can guarantee that the assumptions still hold, namely that the difference between the solutions remains sufficiently small. This is possible since the number of steps depends only on $K$ and $\varepsilon_1$, hence we can iterate such an argument to cover the whole interval in order to obtain the result.

\begin{proof}
Without loss of generality, we may assume $|I| \leq 1$.  As in the proof of Proposition \ref{equ:local_criteria}, there exists some $\delta_1(K)$ such that if for some $J \ni t_0$,
\[
\|e^{\pm i (t- t_0) \langle \nabla \rangle} v(t_0)\|_{Z_\pm'(J)} + \|e\|_{N_\pm(J)} \leq \delta_1,
\]
then there exists a unique solution $v$ to \eqref{equ:soln_kg_sys} on $J$ and
\begin{align}
\label{equ:nonlin_bd}
\|v^\pm - e^{\pm i (t- t_0) \langle \nabla \rangle} v(t_0)\|_{X^{1/2}(J)} \leq \|e^{\pm i (t- t_0) \langle \nabla \rangle} v(t_0)\|^{\frac{5}{3}}_{Z_\pm'(J)} + 2 \|e\|_{N_\pm(J)}.
\end{align}
Next we claim that there exists $\varepsilon_1 = \varepsilon_1(K)$ such that if for some $I_k = (t_k, t_{k+1})$ it holds that
\begin{align}
\label{equ:eps_bds}
\|e\|_{N_\pm(I_k)}  \leq \varepsilon_1 \qquad \textup{and} \qquad \|v^\pm\|_{Z(I_k)} \leq \varepsilon \leq \varepsilon_1
\end{align}
then
\begin{align}
\label{equ:eps_bds2}
\|e^{\pm i (t - t_k) \langle \nabla \rangle} v^\pm(t_k) \|_{Z_\pm'(I_k)} \leq C(1 + K)( \varepsilon + \|e\|_{N_\pm(I_k)})^{\frac{3}{4}}\\
\| v^\pm \|_{Z_\pm'(I_k)} \leq C(1 + K)( \varepsilon + \|e\|_{N_\pm(I_k)})^{\frac{3}{4}}.
\end{align}
Indeed, we let $h(s) := \|e^{\pm i (t - t_k) \langle \nabla \rangle} v^\pm(t_k) \|_{Z_\pm'(t_k , t_k + s)} $. Let $J_k = [t_k, t') \subset I_k$ be the largest interval such that $h(s) \leq \delta_1 /2$, for $\delta_1(K)$ as above. Then by Duhamel's formula
\begin{align}
\|e^{\pm i (t - t_k) \langle \nabla \rangle} v^\pm(t_k) \|_{Z(t_k , t_k + s)} &\leq \|v^\pm \|_{Z(t_k , t_k + s)} + \|v^\pm - e^{\pm i (t - t_k) \langle \nabla \rangle} v^\pm(t_k) \|_{\cX_\pm^{1/2}(t_k , t_k + s)} \\
& \leq \varepsilon + h(s)^{\frac{5}{3}} + 2 \|e\|_{N_\pm(I_k)}.
\end{align}
By definition,
\[
h(s) \leq  \|e^{\pm i (t - t_k) \langle \nabla \rangle} v^\pm(t_k) \|^{\frac{3}{4}}_{Z(t_k , t_k + s)} \, \|e^{\pm i (t - t_k) \langle \nabla \rangle} v^\pm(t_k) \|^{\frac{1}{4}}_{\cX_\pm^{1/2}(t_k , t_k + s)} 
\]
hence  by \eqref{equ:nonlin_bd}, the boundedness of the free evolution in $\cX_\pm^{1/2}$ and \eqref{equ:z_bds},
\begin{align}
h(s) & \leq \biggl(\varepsilon + h(s)^{\frac{5}{3}} + 2 \|e\|_{N_\pm(I_k)} \biggr) \, K^{\frac{1}{4}} \\
&\leq C(1 + K)(\varepsilon +\|e\|_{N_\pm(I_k)})^{3/4} + C(1 + K) h(s)^{\frac{5}{4}},
\end{align}
and we can conclude the claim provided $\varepsilon_1$ is chosen sufficiently small. Let now $I_k$ be an interval such that
\begin{equation}
\label{equ:eps_bds3}
\begin{split}
\|e^{\pm i (t - t_k) \langle \nabla \rangle} v^\pm(t_k)\|_{Z_\pm'(I_k)} + \|v^\pm \|_{Z_\pm'(I_k)} \leq \varepsilon \leq \varepsilon_0\\
\|e\|_{N_\pm(I_k)} \leq \varepsilon,
\end{split}
\end{equation}
it holds by the above considerations that $\|v^\pm\|_{\cX_\pm^{1/2}(I_k)} \leq K + 1$.  Fix such an interval and let $u$ be a solution to \eqref{equ:soln_kg_sys} defined on an interval $J_u \ni t_k$ with 
\[
\|(u(t_k) - v(t_k), \partial_t u(t_k) - \partial_t v(t_k) \|_{\cH^{1/2}} \leq \varepsilon_0.
\]
Set $\phi = u - v$, then $\phi$ solves the difference equation
\begin{equation}
\label{equ:cubic_nlkg_chap3_diff}
\left\{\begin{split}
&v_{tt} - \Delta \phi + \phi + (v + \phi)^3 - v^3 - e = 0, \quad u: \bR \times \bT^3 \to \bR  \\
&(\phi, \partial_t \phi)\big|_{t=t_k} = \bigl(u(t_k) - v(t_k), \partial_t u(t_k) - \partial_t v(t_k)\bigr) \in \cH^{1/2}(\bT^3),
\end{split} \right.
\end{equation}
and we can identify $\phi$ with $(\phi^+, \phi^-)$ as before. Let $J_k = [t_k, t_k + s] \cap I_k \cap J_u$ be the maximal interval such that
\[
\|\phi^\pm\|_{Z_\pm'(J_k)} \leq 10 C \varepsilon_0 \leq \frac{1}{10(K+1)}.
\]
Such an interval exists and is non-empty since $s \mapsto \|\phi^\pm\|_{Z_\pm'(t_0, t_0 + s)} $ is a continuous function which vanishes at $s = 0$. Similarly to the argument used in the standard Strichartz stability theory, by Proposition \ref{prop:multi_refine} we can then estimate
\begin{equation}
\label{equ:v_x_bds}
\begin{split}
&\|\phi^\pm\|_{\cX_\pm^{1/2}(J_k)} \\
&\leq \|e^{\pm i (t - t_k) \langle \nabla \rangle} (u^\pm(t_k) - v^\pm(t_k))\|_{\cX_\pm^{1/2}(J_k)} + \|(v + \phi)^3 - v^3 \|_{N_\pm(J_k)} + \|e\|_{N_\pm(J_k)} \\
& \lesssim \|(u^\pm(t_k) - v^\pm(t_k), \partial_t u^\pm(t_k) - \partial_t v^\pm(t_k) \|_{\cH^{1/2}} + C\bigl( \varepsilon_0 \|\phi^\pm\|_{\cX_\pm^{1/2}(J_k)} + \|e\|_{N_\pm(J_k)}\bigr).
\end{split}
\end{equation}
Thus, if $\varepsilon_0$ is sufficiently small, 
\begin{align}
\label{equ:v_bds}
\|\phi^\pm\|_{Z_\pm'(J_k)} \leq C \|\phi^\pm\|_{\cX_\pm^{1/2}(J_k)} \leq 8 C\varepsilon_0,
\end{align}
and in fact $J_k = I_k \cap J_u$. Moreover, $u$ can be extended to all of $I_k$ by the remark after Proposition \ref{equ:local_criteria}, and further \eqref{equ:v_x_bds} and \eqref{equ:v_bds} hold on all of $I_k$. We conclude the proof by splitting $I$ into subintervals such that
\[
\|v^\pm\|_{Z(I_k)} \leq \varepsilon_2, \qquad \|e\|_{N_\pm(I_k)} \leq \kappa \varepsilon_2.
\]
On each interval \eqref{equ:eps_bds} holds, hence by \eqref{equ:eps_bds2}, \eqref{equ:eps_bds3} also holds, and as above we conclude that \eqref{equ:v_x_bds} and \eqref{equ:v_bds} hold. This concludes the proof of the proposition.
\end{proof}

\begin{rmk}
\label{rmk:stab_trunc}
If we consider the nonlinear Klein-Gordon equation with truncated nonlinearity,
\begin{equation}
\left\{\begin{split}
&(u_N)_{tt} - \Delta u_N + u_N + P_{\leq N}(P_{\leq N} u_N)^3 = 0, \quad u: \bR \times \bT^3 \to \bR  \\
&(u_N, \partial_tu_N)\big|_{t=0} = (u_0, u_1) \in \cH^{1/2}(\bT^3),
\end{split} \right.
\end{equation}
then all the estimates from Sections \ref{sec:multi} and \ref{sec:stab} go through with constants uniform in the truncation parameter. Indeed, set $G(u) = F(P_N u)$ and with $v_N = P_{\leq N} v_N$, we can estimate
\begin{align}
\|P_{\leq N} G(u)\|_{N(I)} \leq \sup_{\|v_N\|_{\cY^{1/2}} = 1} \left| \iint P_N G(u) \, v_N \right| & =  \sup_{\|v_N\|_{\cY^{1/2}} = 1} \left| \iint G(u) \, v_N \right|,
\end{align}
and all of the previous multilinear estimates go through as before. Moreover, to repeat the stability argument above, we only need
\[
\|P_{\leq N} v\|_{Z(I)} + \|v\|_{L_t^\infty \cH_x^{1/2}(I \times \bT^3)} \leq K.
\]
Moreover, if one only needs to compare low frequencies, it suffices to require
\[
\|P_{\leq N} (v_0 - u_0, v_1 - u_1)\|_{ \cH^{1/2}(\bT^3)} + \|P_{\leq N} e\|_{N(I)} \leq \varepsilon < \varepsilon_1.
\]
This statement should be compared with Remark \ref{rmk:low_freq_bds} and Lemma \ref{lem:pert_long_trunc}.
\end{rmk}

\section{A low frequency equation}
\label{sec:approx}

The next propositions should be compared with Proposition 5.1 in \cite{CKSTT05}. As we are at the critical level, we cannot hope to achieve the gain of derivatives for the Strichartz estimates, as was obtained in Theorem 4.3 in \cite{CKSTT05}. On the other hand, we can still obtain some decay by exploiting the improved Strichartz estimates from Proposition \ref{prop:multilinear} provided we are able to create a scale separation between low and high-frequencies. In the sequel, where we denote errors of a given order, this is always understood to be in the $N(I)$ norm defined in \eqref{equ:norm_nonlin}. In the sequel we will write $F(u, v , w) = u v w$ as a means to record to various combinations in the nonlinearity.

\begin{prop}
\label{prop:error_eqn}
Let $R, T > 0$ and let $u$ be a solution to the cubic nonlinear Klein-Gordon equation for initial data $(u_0, u_1) \in \textbf{B}_R \subset \cH^{1/2}$ which satisfies
\begin{align}
\label{equ:uniform_l42}
\|u\|_{L_{t,x}^4(([0,T) \times \bT^3)} \leq K.
\end{align}
Let $1 \leq N' \ll N_*$. Then there exists $M \in [N', N_*]$ such that the low frequency component $u_{lo} = P_{\leq M} u$ satisfies the perturbed cubic nonlinear Klein-Gordon equation
\begin{align}
\label{equ:low}
\square u_{lo} + u_{lo} = P_{\leq M} F(u_{lo}, u_{lo}, u_{lo}) + \cO_{K, R, T}\biggl( \Bigl(\log \frac{N_*}{N'}\Bigr)^{-\theta}\biggr).
\end{align}
\end{prop}

\begin{proof}
In this proof we allow the value of the parameter $\theta > 0$ to change from line to line, and we allow implicit constants to depend on the various parameters involved. Given dyadic frequencies $N'$, $N_*$, the frequency interval $[N', N_*]$ contains  $O\left(\frac{\log (N_* / N')}{\log\log (N_* / N') }\right)$ intervals of the form 
\[
I_\alpha := [N'(\log (N_* / N'))^{\alpha}, N'(\log (N_* / N'))^{\alpha + 2}] \qquad \textup{for  }\alpha = 2k, k \in \mathbb{N}.
\]
Since $L_{t,x}^4$ bounds imply $\cX^{1/2}$ bounds by Proposition \ref{prop:xs_solN_1ds}, and we have the embedding $\cX^{1/2} \hookrightarrow \cY^{1/2}$, we obtain $\cY^{1/2}$ bounds for solutions of the nonlinear Klein-Gordon given the assumption \eqref{equ:uniform_l42}. Consequently, by \eqref{cor:orthog} and we conclude that there must be some frequency interval $I_\alpha$ where
\begin{align}
\label{equ:med_smallness}
&\|(P_{\leq N' (\log (N_* / N'))^{\alpha + 2}} - P_{\leq N' (\log (N_* / N'))^\alpha}) \,u \|_{\cY_T^{1/2}} 
\lesssim \left(\log \frac{N_*}{N'}\right)^{-\theta} .
\end{align}
Fix this $\alpha$. We introduce the notation $M := N' (\log (N_* / N'))^\alpha$, and  we define
\begin{align}
&u_{lo} = P_{\leq M} , \qquad  u_{med} = \bigl (P_{\leq M (\log (N_* / N'))^{2}} - P_{\leq M} \bigr) u, \qquad  u_{hi} = \bigr (1 - P_{\leq M (\log (N_* / N'))^{2}}\bigr)u.
\end{align}
Let $a,b,c \in \{lo,\, med,\, hi\}$, then we can decompose the nonlinearity as
\begin{align}
P_{\leq M} F(u, u ,u) = P_{\leq M} F(u_{lo}, u_{lo}, u_{lo}) + \sum_{\max(a,b,c)\, \geq\, med} P_{\leq M} F(u_a, u_b, u_c).
\end{align}
By \eqref{equ:med_smallness} and \eqref{equ:ys_est}, any term involving $u_{med}$ will be an error term. Hence, we only need to consider terms with $u_{hi}$ and $u_{lo}$ components. We will use largely the same analysis used to prove Theorem \ref{thm:nonlinear_est}, but in this case, the dual function will always be localized to low frequencies. First we consider
\[
P_{\leq M} F(u_{lo}, u_{hi}, u_{hi}) = \underset{N_3, N_0 \lesssim M}{\sum_{ N_1, N_2} } P_{N_0} P_{\leq M} F(u_{N_3}, u_{N_2}, u_{N_1}),
\]
which we will treat with similar estimates to those used to prove \eqref{equ:multi2}. By the convolution requirement, the term above will vanish unless $N_i \sim N_j$ for some $i \neq j$, hence we may assume without loss of generality that
\[
N_3 \lesssim N_2 \lesssim N_1, \qquad N_1 \sim N_2,
\]
and we bound
\begin{align}
\left| \iint u_{N_3} u_{N_2} u_{N_1} v_{N_0} dx dt \, \right| &\leq \|u_{N_3} u_{N_1} \|_{L_{t,x}^2} \| u_{N_1} v_{N_0}\|_{L_{t,x}^2} \\
& \lesssim \left( \frac{N_1 N_2}{N_3 N_0} \right)^\delta N_3 N_0 \,  \|u_{N_1} \|_{V_{\pm_1}^2}   \|u_{N_2} \|_{V_{\pm_2}^2}\|u_{N_3} \|_{V_{\pm_3}^2} \|v_{N_0} \|_{V_{\pm_0}^2} .
\end{align}
Applying Cauchy-Schwarz in $N_0, N_3 \lesssim M$, we obtain
\begin{align}
&  \sup_{\|w\|_{\cY_\pm^{1/2}} = 1}  \sum_{N_0 \lesssim M}   \underset{N_1, N_2 \sim hi}{ \sum_{N_3 \lesssim \, M } }\left( \frac{N_1 N_2}{N_3 N_0} \right)^\delta N_3 N_0 \,  \|u_{N_1} \|_{V_{\pm_1}^2}   \|u_{N_2} \|_{V_{\pm_2}^2} \|u_{N_3} \|_{V_{\pm_3}^2} \|v_{N_0} \|_{V_{\pm_0}^2}    \\
& \lesssim \|u_3\|_{\cY_{\pm_3}^{1/2}}  \sum_{N_1, N_2 \sim hi }  \,\frac{ M^{1 - 2\delta}}{(N_1 N_2)^{\frac{1}{2} - \delta}} (N_1 N_2)^{\frac{1}{2}}  \|u_{N_1} \|_{V_{\pm_1}^2} \|u_{N_2} \|_{V_{\pm_2}^2}  .
\end{align}
In particular, by the restriction
\begin{align}
\label{equ:freq_sep}
N_2 , N_1 \gtrsim M (\log (N_* / N'))^{2},
\end{align}
we can bound the above multilinear estimate by
\[
 \|u_3\|_{\cY_{\pm_3}^{1/2}}  \sum_{N_1, N_2 \sim hi }  \frac{1}{(\log(N_* / N'))^{2 - 4\delta}} \,(N_1 N_2)^{\frac{1}{2}} \|u_{N_1} \|_{V_{\pm_1}^2} \|u_{N_2} \|_{V_{\pm_2}^2},
\]
and by applying Cauchy-Schwarz in $N_2 \sim N_1$, we see this term is part of the error in \eqref{equ:low}.

\noindent For the term with high-frequencies
\begin{align}
\label{equ:high_high}
P_{\leq M} F(u_{hi}, u_{hi}, u_{hi}) 
\end{align}
we can once again assume without loss of generality that $N_3 \lesssim N_2 \sim N_1$. We obtain 
\begin{align}
\left| \iint u_{N_3} u_{N_2} u_{N_1} v_{N_0} dx dt \, \right| &\leq \|u_{N_3} u_{N_2} \|_{L_{t,x}^2} \| u_{N_1} v_{N_0}\|_{L_{t,x}^2} \\
& \lesssim \left( \frac{N_1 N_2}{N_0 N_3} \right)^\delta N_3\, N_0 \,  \|u_{N_1} \|_{V_{\pm_1}^2}   \|u_{N_2} \|_{V_{\pm_2}^2} \|u_{N_3} \|_{V_{\pm_3}^2} \|v_{N_0} \|_{V_{\pm_0}^2} .
\end{align}
By Cauchy-Schwarz in $N_0$ and $N_3$, we obtain
\begin{align}
& \sup_{\|w\|_{\cY_\pm^{1/2}} = 1}   \sum_{N_0 \lesssim M}  \,\, \sum_{N_1, N_2, N_3 \sim hi}  \left( \frac{N_1 N_2}{N_0 N_3} \right)^\delta N_3\, N_0 \,  \|u_{N_1} \|_{V_{\pm_1}^2}   \|u_{N_2} \|_{V_{\pm_2}^2} \|u_{N_3} \|_{V_{\pm_3}^2} \|v_{N_0} \|_{V_{\pm_0}^2} \\
& \lesssim   \|u_3\|_{\cY_{\pm_3}^{1/2}}  \sum_{N_1, N_2 \sim hi}  \,\frac{ M^{\frac{1}{2} - \delta} }{N_1^{\frac{1}{2} - \delta}} \, (N_1 N_2)^{\frac{1}{2}}   \|u_{N_1} \|_{V_{\pm_1}^2} \|u_{N_2} \|_{V_{\pm_2}^2} ,
\end{align}
and once again by \eqref{equ:freq_sep} and Cauchy-Schwarz in $N_2 \sim N_1$ we conclude that such expressions contribute to the error term in \eqref{equ:low}. Finally, we note that by the convolution requirement, the term
\[
P_{\leq M} F(u_{lo}, u_{lo}, u_{hi})
\]
vanishes provided $\log(N_* / N')^2 > 8$, which concludes the proof.
\end{proof}

\begin{rmk}
The key difference between the proofs of Propositions \ref{prop:error_eqn} and \ref{thm:multilinear_ests} is that the low frequency is always bounded by $M$ in Propositions \ref{prop:error_eqn}, hence when we apply Cauchy-Schwarz in the low frequency, we only lose a factor of $M$, instead of the next largest frequency. This fact, together with the manufactured frequency separation, enables us to achieve the necessary decay.
\end{rmk}

Finally we arrive at the proof of the main result of this section.

\begin{proof}[Proof of Theorem \ref{thm:local_compare}]
Let $u$ and $\widetilde{u}$ denote the solutions to \eqref{equ:cubic_nlkg} with initial data $(u_0, u_1)$ and $(\widetilde{u}_0, \widetilde{u}_1)$, respectively. Let $K > 0$ be such that
 \[
\|u\|_{L_{t,x}^4([0,T] \times \bT^3)} \,\, + \,\, \|\widetilde{u}\|_{L_{t,x}^4([0,T] \times \bT^3)} \leq K.
 \] 
 Using the arguments above, we can find $M$ such that $u_{lo} = P_{\leq M} u_{lo}$ such that Proposition \ref{prop:error_eqn}, the low frequency component satisfies the equation
\begin{align}
\square \,u_{lo} + u_{lo} = P_{\leq M} F(u_{lo}, u_{lo}, u_{lo}) + \cO_{K, R, T}( (\log (N_* / N'))^{-\theta})
\end{align}
and similarly for $\widetilde{u}_{lo} = P_{\leq M} \widetilde{u}_{lo}$ for the same cut-off $M$, with slightly different error terms. Since $(u_{lo}  - \widetilde{u}_{lo} ) \big|_{t= 0 } = 0$, let $N_*$ be chosen sufficiently large depending on $R, K$ and $T$ so that the smallness requirement of Proposition \ref{prop:xs_stability} is satisfied. Then the result follows from the stability theory, Remark \ref{rmk:bound_requirements} and the fact that the bounds for the equation with truncated nonlinearity are uniform in $N$.
\end{proof}

\section{Probabilistic bounds for the nonlinear component of the flow}
\label{sec:compactness}

\medskip
In this section, we will show boundedness for the nonlinear component of the cubic nonlinear Klein-Gordon equation on the subsets $\Sigma_\lambda \subset \Sigma$. In the sequel we let $F(u) = u^3$. We will begin with the proof of a local boundedness property for the nonlinearity $F(u)$. The argument is based on Strichartz estimates together with the improved averaging effects for the free evolution of initial data $(u_0, u_1) \in \Sigma_\lambda$, where $\Sigma_\lambda$ was defined in \eqref{equ:sigma_lambda}, as well as the uniform bounds on the nonlinear component of those global solutions from Proposition \ref{prop:unif_global2} and Proposition \ref{prop:unif_trunc}. 

\medskip
We wish to obtain bounds on $\|F[u(t)] \|_{X^{s_2 -1,- \frac{1}{2}+, \delta}}$ for solutions to the cubic nonlinear Klein-Gordon equation \eqref{equ:cubic_nlkg} with initial data $(u_0, u_1) \in \Sigma_\lambda$. We will perform the estimates with $b = \frac{1}{2}+$ and hence, we need to estimate the expression
\begin{align}
\label{equ:nonlinearity_estimate}
\left[ \sum_{n} \int d\tau \frac{|\widehat F(n, \tau) |^2}{\langle n\rangle^{2(1- s_2)} \langle |\tau| - \langle n\rangle\rangle^{2(1 - b)}} \right]^{\frac{1}{2}}
\end{align}
for $F(u) = u^3$ and $b > 1/2$. If we expand $F(u)$ for $u(t) = S(t)(u_0, u_1) + w(t)$, then we need to consider terms of the form $u^{(1)} u^{(2)} u^{(3)}$ for $u^{(i)}$ either
\begin{enumerate}
\item[(I)\,] the free evolution $S(t)(u_0, u_1)$ of initial data $(u_0, u_1) \in \Sigma_\lambda \subset \cH^{1/2}$, or
\item[(II)\:] the nonlinear component, $w(t)$, of a solution to \eqref{equ:cubic_nlkg} with initial data $(u_0, u_1) \in \Sigma_\lambda$.
\end{enumerate}
We will refer to these as type (I) or type (II) functions. We define
\[
c_i(n_i, \tau_i) = \langle n_i \rangle^{s_1} \langle| \tau_i| - \langle n_i\rangle \rangle^{b} |\widehat{u}^{(i)} (n_i, \tau_i) |,
\]
then $\|c\|_{L_\tau^2 \ell_n^2  (\mathbb{R} \times \bT^3)} = \|u^{(i)}\|_{X^{s_1, b}  (\mathbb{R} \times \bT^3)}$.  By duality, \eqref{equ:nonlinearity_estimate} can be estimated by
\begin{align}
\label{equ:duality}
\sum_{n = n_1 + n_2 + n_3} \int_{ \tau = \tau_1 + \tau_2 + \tau_3} \,\prod_{i=1}^3 \frac{c_i(n_i, \tau_i)}{\langle n_i \rangle^{s_1} \langle |\tau_i| - \langle n_i \rangle \rangle^{b}}\frac{v(n, \tau)}{\langle n \rangle^{1-s_2} \langle|\tau|- \langle n\rangle\rangle^{1- b} } \,d\tau
\end{align}
where $\|v\|_{L_\tau^2 \ell_n^2  (\mathbb{R} \times \bT^3)} \leq 1$. We remark that this notation should not be confused with the initial data for \eqref{equ:cubic_nlkg}, and it will be clear from the context which we are considering.

\medskip
We restrict the $n_i$ and $n$ to dyadic regions $|n_i| \sim N_i$ and $|n| \sim N$. We will implicitly insert  a time cut-off with each function but we will omit the notation, since, in the usual way we can take extensions of the $u_i$ and then take infimums. The ordering of the size of the frequencies will not play a role in this argument. We do not repeat these considerations. Letting
\begin{align}
\label{equ:viwi}
\widehat U_{N_i}(n_i, \tau_i) = \frac{c_i(n_i, \tau_i)}{\langle |\tau_i| - \langle n_i \rangle\rangle^{b}} \chi_{\{|n_i| \sim N_i\}}, \qquad \widehat V_N(n, \tau) = \frac{v(n, \tau)}{\langle |\tau| - \langle n\rangle \rangle^{1- b}} \chi_{\{|n| \sim N\}}
\end{align}
we will need to estimate expressions of the form
\begin{align}
\label{equ:local_est_expression}
(N_1 N_2 & N_3)^{-s_1} N^{-(1-s_2)} \int_{\bR_t} \int_{\bT^3_x} \,\, \prod_{i=1}^3 U_{N_i} \cdot V_N \, dx\, dt.
\end{align}
We will use expressions \eqref{equ:duality} and \eqref{equ:local_est_expression} as starting points in proofs of the subsequent propositions. In the sequel, we will always take the constant $\gamma = s_1$ in our definition of $\Sigma$.

\subsection{Boundedness of the flow map} 

\begin{prop}[Local boundedness]
\label{prop:bddness}
Consider the cubic nonlinear Klein-Gordon equation \eqref{equ:cubic_nlkg}. Then there exists $s_1 < \frac{1}{2} < s_2$ with $s_1, s_2$ sufficiently close to $1/2$ such that for any $\lambda, R, T > 0$, for every $(u_0, u_1) \in \Sigma_\lambda \cap \textbf{B}_R$ and for any interval $I \subset [0,T]$ with $|I| = \delta$, the nonlinearity satisfies the bound
\begin{equation}
\label{equ:local_bds}
\begin{split}
\|&F(u) \|_{X^{s_2 -1, - \frac{1}{2}+, \delta}(I \times \bT^3)} \leq C (\lambda,R,T) \, \delta^c \,\Bigl( 1 +  \|u\|^{9/4}_{X^{s_1, \frac{1}{2}+, \delta}(I \times \bT^3)}   \Bigr).
\end{split}
\end{equation}
where $(u, \partial_t u)$ is the global solution to the cubic nonlinear Klein-Gordon equation \eqref{equ:cubic_nlkg} with initial data $(u_0, u_1)$.
\end{prop}

\begin{proof}
For any solution $u(t) = S(t)(u_0, u_1) + w(t)$, our computations will yield \eqref{equ:local_bds} with the nonlinear component of the solutions, $w$, on the right-hand side instead of $u$. Now, for any $s, b \in \bR$ and any interval $I \subset [0,T]$ with $|I| = \delta$ and $\inf I = t_0$, and $\eta(t)$ a Schwartz time-cutoff adapted to that interval we claim that we it sufficed to obtain \eqref{equ:local_bds} with the nonlinear components of the solutions on the right-hand side. Indeed,
\begin{align*}
\| \eta(t) w \|_{X^{s, b}(I \times \bT^3)} &= \| \eta(t)(u - S(t)(u_0, u_1) )\|_{X^{s, b}(I \times \bT^3)} \\&
\lesssim \|\eta(t) S(t) (u_0, u_1)\|_{X^{s, b}(I \times \bT^3)} + \|\eta(t) u\|_{X^{s, b}(I \times \bT^3)} \\
& \lesssim \| (u(t_0), \partial_t u(t_0))\|_{\cH^{s}} + \|\eta(t) u\|_{X^{s, b}(I \times \bT^3)} .
\end{align*}
Since the free evolution is bounded in $\cH^s$, we have
\[
\| (u(t_0), \partial_t u(t_0))\|_{\cH^{s}} \lesssim \|(u_0, u_1) \|_{\cH^s} + \|(w, \partial_t w) \|_{L^\infty\cH^s([0,T] \times \bT^3)},
\]
and since the terms on the right-hand side of \eqref{equ:local_bds} are computed with $s = s_1 < 1/2$, then by the choice of $\Sigma_\lambda$, we have that for any such $I \subset [0,T]$,
\[
\| \eta(t) w \|_{X^{s, b}(I \times \bT^3)} \lesssim \|\eta(t) u\|_{X^{s, b}(I \times \bT^3)}  + C(\lambda, R, T)
\]
which will yield \eqref{equ:local_bds}. The key point here is that on any interval $[0,T]$, the choice of $\Sigma_\lambda$ yields uniform control on bounds of the Sobolev norm of solutions. 
We will not repeat these considerations. 

\medskip
We analyze the different combinations of $u_i $ systematically. The argument only depends on the number of $u_i$ which are of type (I) or type (II), so we will only present one combination from each case.

\noindent $\bullet$\textbf{ Case (A): All $u^{(i)}$ of type (II)}. Since
\[
\| U_{N_i}^\alpha U_{N_i}^{1-\alpha} \|_{L^{4+ \varepsilon_2}} \lesssim \| U_{N_i}\|_{L^{p\alpha}}^\alpha \|U_{N_i}\|_{L^{(1-\alpha)5}}^{1-\alpha} 
\]
for $p = \frac{(4+\varepsilon_2)5}{1 - \varepsilon_2}$, we use H\"older's inequality to estimate \eqref{equ:duality} by
\begin{align*}
(N_1 N_2 & N_3)^{-s_1} N^{-(1-s_2)} \int_{\bR_t} \int_{\bT^3_x} \,\, \prod_{i=1}^3 U_{N_i} \cdot V_N \, dx\, dt \\
& \lesssim (N_1 N_2 N_3)^{-s_1} N^{-(1-s_2)} \prod_{i=1}^3 \| U_{N_i}\|_{L^{p\alpha}}^\alpha \|U_{N_i}\|_{L^{(1-\alpha)5}}^{1-\alpha}  \|V_N\|_{L^{4-3\varepsilon_1}} .
\end{align*}
Taking $\alpha = 1/4 $, we have $(1-\alpha) 5 < 4$ and for $\varepsilon_2$ sufficiently small, $p \alpha < 6$.  Provided
\begin{align}
\label{equ:rho_cond}
 1 - b >  \frac{4 - 6 \varepsilon_1 }{ 2(4 - 3 \varepsilon_1)} =: \frac{\theta_2}{2},  
\end{align}
 
we can bound
\[
\frac{1}{\langle |\tau| - \langle n \rangle\rangle^{1- b}} = \frac{1}{\langle |\tau| - \langle n \rangle\rangle^{1- b - \frac{\theta_2}{2}+}} \frac{1}{\langle |\tau| - \langle n \rangle\rangle^{\frac{\theta_2}{2} + }} \lesssim \frac{1}{\langle |\tau| - \langle n \rangle\rangle^{\frac{\theta_2}{2} + }}
\]
and we can apply Strichartz estimates \eqref{equ:interp} with $r = \frac{15}{4}$ and $\theta_1 = 14/15$ for the $U_{N_i}$ and Strichartz estimates with $r = 4 - 3 \varepsilon_1$ and $\theta_2 = \frac{4 - 6 \varepsilon_1 }{ 4 - 3 \varepsilon_1}$ for $V_N$. By Sobolev embedding, accounting for the $N_i^{s_1}$ factors in the expressions for the $U_{N_i}$, we obtain 
\[
\lesssim (N_1 N_2 N_3)^{\frac{21}{60}  - \frac{3}{4}s_1 } N^{-(1 -\frac{1}{2}\cdot  \frac{4 - 6 \varepsilon_1 }{ 4 - 3 \varepsilon_1} - s_2)} \prod_{i=1}^3\|u^{(i)}\|^{1/4}_{H^{1}(\bT^3)}  \prod_{i=1}^3 \|c_i\|^{3/4}_{L^2} .
\]
The expression on the right-hand side is summable for dyadic values of $N_i$ and $N$. By the definition of $\Sigma_\lambda$ \eqref{equ:sigma_lambda}, we obtain
 
\begin{align}
\label{equ:local_bds1}
\eqref{equ:duality}  \lesssim (\lambda + T)^{\frac{3}{4} +} \prod_{i=1}^3 \|u^{(i)}\|^{\frac{3}{4} }_{X^{s_1 , b,\, \delta}}.
\end{align}

\medskip
\medskip
\noindent $\bullet$\textbf{ Case (B):  $u^{(1)}$ of type (I) and $u^{(2)}, u^{(3)}$ of type (II)}.
In this case, we use H\"older's inequality and Strichartz estimates \eqref{equ:interp} with $r = \frac{18}{5}$ and $\theta = \frac{8}{9} $, and provided
\[
1 - b > \frac{4}{9},
\]
which holds if \eqref{equ:rho_cond} holds, we can bound
\begin{align*}
\eqref{equ:duality}& \lesssim (N_1 N_2 N_3)^{-s_1} N^{-(1-s_2)} \int U_{N_1} \,U_{N_2} \, U_{N_3}  \cdot V_N \, dx dt \\
& \lesssim (N_1N_2 N_3)^{-s_1} N^{-(1-s_2)} \|U_{N_1}\|_{L^6}  \|U_{N_2}\|_{L^{\frac{18}{5}}} \|U_{N_3}\|_{L^{\frac{18}{5}}} \|V_N\|_{L^{\frac{18}{5}}}\\
& \lesssim (N_1)^{- s_1} (N_2 N_3)^{ \frac{4}{9} - s_1}\, N^{-(\frac{5}{9} -s_2)}\, \|U_{N_1}\|_{L^6} \,  \|u^{(2)}\|_{X^{s_1 , b}} \|u^{(3)}\|_{X^{s_1, b}} .
\end{align*}
Noting how we defined $U_{N_1}$, we use the fact that $s_1 < \frac{1}{2}$ and that $(u_0, u_1) \in \Sigma_\lambda$, to estimate that piece, exploiting the uniform boundedness of the frequency projections in $L_x^p$ spaces. Once again, we obtain an estimate which summable for dyadic $N$ and $N_i$, yielding
\begin{align}
\label{equ:local_bds2}
\eqref{equ:duality} \lesssim \lambda \, \|u^{(2)}\|_{X^{s_1 ,b, \delta}} \|u^{(3)}\|_{X^{s_1,b, \delta}} .
\end{align}

\medskip
\medskip
\noindent $\bullet$\textbf{ Case (C): $u^{(1)}, u^{(2)}$ of type (I) and $u^{(3)}$ of type (II)}.
In this case, we use H\"older's inequality and Strichartz estimates \eqref{equ:interp} with $r = 3$ and $\theta = \frac{2}{3}$ and
\[
1- b > \frac{1}{3},
\]
which holds if \eqref{equ:rho_cond} holds and estimate
\begin{align*}
\eqref{equ:duality}& \lesssim (N_1 N_2N_3)^{-s_1} N^{-(1-s_2)} \int U_{N_1} \,U_{N_2} \, U_{N_3}  \cdot V_N \, dx dt \\
& \lesssim (N_1 N_2N_3)^{-s_1} N^{-(1-s_2)} \|U_{N_1}\|_{L^{6}} \|U_{N_2}\|_{L^{6}}  \|U_{N_3}\|_{L^{3}}  \|V_N\|_{L^{3}}\\
& \lesssim (N_1 N_2)^{-s_1} (N_3 )^{\frac{1}{3} - s_1}\, N^{-(\frac{2}{3} -s_2 -)} \,  \|U_{N_1}\|_{L^{6}} \|U_{N_2}\|_{L^{6}}\, \|u^{(3)}\|_{X^{s_1 , b}}   
\end{align*}
which is again summable for dyadic $N$ and $N_i$, yielding
\begin{align}
\label{equ:local_bds3}
\eqref{equ:duality} \lesssim \lambda^2 \,  \|u^{(3)}\|_{X^{s_1, b}} .
\end{align}

\medskip
\medskip
\noindent $\bullet$\textbf{ Case (D): All $u^{(i)}$ of type (I)} .
In this case we estimate
\begin{align*}
\eqref{equ:duality}& \lesssim (N_1 N_2N_3)^{-s_1} N^{-(1-s_2)} \int U_{N_1} \,U_{N_2} \, U_{N_3}  \cdot V_N \, dx dt \\
& \lesssim (N_1 N_2N_3)^{-s_1} N^{-(1-s_2)} \|U_{N_1}\|_{L^{6}} \|U_{N_2}\|_{L^{6}}  \|U_{N_3}\|_{L^{6}}  \|V_N\|_{L^{2}}\\
& \lesssim (N_1 N_2N_3)^{-s_1}\, N^{-(1 -s_2 )} \,  \|U_{N_1}\|_{L^{6}} \|U_{N_2}\|_{L^{6}}\, \|U_{N_3}\|_{L^{6}}   %
\end{align*}
which is again summable for dyadic $N$ and $N_i$, yielding
\begin{align}
\label{equ:local_bds4}
\eqref{equ:duality} \lesssim \lambda^3.
\end{align}
Finally, we note that for $- \frac{1}{2} < - (1 - b)$, we have
\begin{align}
\label{equ:time_loc}
\|F(u) \|_{X^{s_2 -1,  -(1 - b),\: \delta}} \lesssim \delta^{\varepsilon} \, \|F(u) \|_{X^{s_2 -1, - (1 - b) + \varepsilon,\, \delta}}
\end{align}
where the implicit constant depends only on $\varepsilon$. To obtain the small time factor in the estimates, we can perform the previous estimates replacing $b$ with $b + \varepsilon$ on the $V_N$ factor for some small $\varepsilon > 0$. Provided $b = \frac{1}{2}+$ is chosen sufficiently close to $\frac{1}{2}$ and $\varepsilon > 0$ is sufficiently small so as to ensure that \eqref{equ:rho_cond} continues to hold for $b + \varepsilon$, we obtain the desired estimate. We will not repeat this consideration in the following Propositions.

\medskip
Combining \eqref{equ:local_bds1}, \eqref{equ:local_bds2}, \eqref{equ:local_bds3},  \eqref{equ:local_bds4} and \eqref{equ:time_loc} yields 
\begin{align}
\|&F(u) \|_{X^{s_2 -1, - \frac{1}{2}+, \delta}}  \leq C (\lambda, R) \, \delta^c \,\Bigl( 1 +   \|u\|^{9/4}_{X^{s_1, \frac{1}{2}+, \delta}}   \Bigr). \qedhere
\end{align}
\end{proof}

\begin{rmk}
\label{rmk:trunc_bds}
If we consider $(u_0, u_1) \in \Sigma_\lambda \cap \textbf{B}_R$ and instead we look at solutions to the cubic nonlinear Klein-Gordon equation with truncated nonlinearity, $u_N$, then we obtain the same bounds
\begin{equation*}
\label{equ:local_bds_trunc}
\begin{split}
\| &F(u_N) \|_{X^{s_2 -1, \frac{1}{2}+, \delta}} \leq C (\lambda, R, T) \, \delta^c \,\Bigl( 1  +  \|u_N\|^{9/4}_{X^{s_1, \frac{1}{2}+, \delta}}   \Bigr),
\end{split}
\end{equation*}
and similarly for $P_N u_N$. Indeed, by the choice of $\Sigma_\lambda$, the nonlinear component of the solution $w_N$ will satisfy the same bounds as the nonlinear component, $w$, of the solution to the full equation. Since the linear components of $u$ and $u_N$ are the same, certainly we have the same bounds on $S(t)(u_0, u_1)$ and we can repeat the arguments in the previous proof to obtain \eqref{equ:local_bds_trunc}. 
\end{rmk}

\subsection{Boundedness of the flow with truncated nonlinearity}

We defined the truncated nonlinearity $F_N(u_N) = P_N [ (P_Nu_N)^3]$. A direct consequence of Proposition \ref{prop:unif_trunc}, Proposition \ref{prop:bddness} and our choice of $\Sigma_\lambda$ is the following local boundedness result for the truncated nonlinear Klein-Gordon equation. It is important to note that the bounds we obtain are uniform in the truncation parameter.

\begin{prop}[Local boundedness for the truncated equation]
\label{prop:trunc_bddness}
Consider the cubic nonlinear Klein-Gordon equation with truncated nonlinearity \eqref{equ:cubic_nlkg_trun}. Then there exists $s_1 < \frac{1}{2} < s_2$ with $s_1, s_2$ sufficiently close to $1/2$ such that for any $\lambda, R, T > 0$, for every $(u_0, u_1) \in \Sigma_\lambda \cap \textbf{B}_R$ and for any interval $I \subset [0,T]$ with $|I| = \delta$, the truncated nonlinearity satisfies the bound
\begin{equation}
\label{equ:local_bds10}
\begin{split}
\| &F_N(u_N) \|_{X^{s_2 -1 , -\frac{1}{2}+, \delta}(I \times \bT^3)} \leq C (\lambda, R,T) \, \delta^c \,\Bigl( 1+  \|u_N\|^{9/4}_{X^{s_1, \frac{1}{2}+, \delta}(I \times \bT^3)}   \Bigr).
\end{split}
\end{equation}
where $(u_N, \partial_t u_N)$ is the global solution to \eqref{equ:cubic_nlkg_trun} with initial data $(u_0, u_1)$.
\end{prop}

\subsection{Continuity estimates for the flow map}
We need the following continuity-type estimate when we compare the full nonlinearity on solutions of the full flow to solutions of the truncated equation.
\begin{prop}
\label{prop:fin_bddness}
Consider the cubic nonlinear Klein-Gordon equation \eqref{equ:cubic_nlkg} and the cubic nonlinear Klein-Gordon equation with truncated nonlinearity \eqref{equ:cubic_nlkg_trun}.  Then there exists $s_1 < \frac{1}{2} < s_2$ with $s_1, s_2$ sufficiently close to $1/2$ such that for any $\lambda, R, T > 0$, for every $(u_0, u_1) \in \Sigma_\lambda \cap \textbf{B}_R$ and for any interval $I \subset [0,T]$ with $|I| = \delta$, the nonlinearity satisfies the bound
\begin{equation*}
\begin{split}
& \|F(u)-  F(u_N) \|_{X^{-\frac{1}{2}, - \frac{1}{2}+,\delta}(I \times \bT^3)} \\
&\leq c(\lambda, R, T) \, \delta^c \, \| u - u_N \|_{X^{\frac{1}{2}, \frac{1}{2}+,\delta}(I \times \bT^3)} \Bigl(1 +  \|u\|^{6/4}_{X^{s_1, \frac{1}{2}+,\delta}(I \times \bT^3)} +  \|u_N\|^{6/4}_{X^{s_1, \frac{1}{2}+,\delta}(I \times \bT^3)} \Bigr),
\end{split}
\end{equation*}
where $(u, \partial_t u)$ and $(u_N, \partial_t u_N)$ to the full and truncated equations, respectively, with data $(u_0, u_1)$.
\end{prop}
\begin{proof}
As in the proof of the Proposition \ref{prop:bddness}, in light of \eqref{equ:time_loc} we will take $b = \frac{1}{2}+$ for $b$ sufficiently close to $\frac{1}{2}$ so that the time localization yields the desired $\delta^c$ factor. We first note that 
\[
| (u)^3 -  (u_N)^3 | \lesssim  | u - u_N | \bigl( |u|^2 + |u_N|^2 \bigr),
\]
hence these estimates are similar to those in Proposition \ref{prop:bddness} but we will always estimate $u_1 = | u - u_N |$ in $X^{\frac{1}{2}, \frac{1}{2} + , \,\delta}$. More precisely, once again we estimate the expression
\begin{align}
\label{equ:duality_2} \qquad
\sum_{n = n_1 + n_2 + n_3} \int_{ \tau = \tau_1 + \tau_2 + \tau_3} d\tau \frac{c_1(n_1, \tau_1)}{\langle n_1 \rangle^{\frac{1}{2}} \langle |\tau_1| - \langle n_1 \rangle \rangle^{b}} \,\prod_{i=2}^3 \frac{c_i(n_i, \tau_i)}{\langle n_i \rangle^{s_1} \langle |\tau_i| - \langle n_i\rangle \rangle^{\frac{1}{2}+}}\frac{v(n, \tau)}{\langle n\rangle^{\frac{1}{2}} \langle|\tau|- \langle n\rangle\rangle^{1 - b} } \, \quad 
\end{align}
where $\|v\|_{L_\tau^2 \ell_n^2 ( \mathbb{R} \times \bT^3)}  \leq 1$ and as before the functions $u$ or $u_N$ in the expression for $c$ are either of type (I) or type (II), with $u_1$ always of type (II).

\medskip
We define $U_{N_i}$ and $V_N$ as in \eqref{equ:viwi}. The key difference between this proof and the proof of Proposition \ref{prop:bddness} is that in each case, we will estimate $u_1$ in $X^{\frac{1}{2} , \frac{1}{2}+, \delta}$ instead of using $X^{s_1 , \frac{1}{2}+, \delta}$ which we use for the other functions.

\medskip
\noindent $\bullet$\textbf{ Case (A): All $u^{(i)}$ of type (I)}.
Once again, we recall that since
\[
\| U_{N_i}^\alpha U_{N_i}^{1-\alpha} \|_{L^{4+ \varepsilon_2}} \lesssim \| U_{N_i}\|_{L^{p\alpha}}^\alpha \|U_{N_i}\|_{L^{(1-\alpha)5}}^{1-\alpha} 
\]
for $p = \frac{(4+\varepsilon_2)5}{1 - \varepsilon_2}$, we use H\"older's inequality to estimate \eqref{equ:duality_2} by
\begin{align*}
N_1^{- \frac{1}{2}} (N_2 & N_3)^{-s_1} N^{-\frac{1}{2}} \int_{\bR_t} \int_{\bT^3_x} \,\, \prod_{i=1}^3 U_{N_i} \cdot V_N \, dx\, dt \\
& \lesssim N_1^{-\frac{1}{2}} (N_1 N_2 N_3)^{-s_1} N^{-\frac{1}{2}} \|U_{N_1}\|_{L^{4-\varepsilon_1}} \prod_{i=2}^3 \| U_{N_i}\|_{L^{p\alpha}}^\alpha \|U_{N_i}\|_{L^{(1-\alpha)5}}^{1-\alpha}  \|V_N\|_{L^{4-\varepsilon_1}} .
\end{align*}
Taking $\alpha = 1/4 $, we have $(1-\alpha) 5 < 4$ and for $\varepsilon_2$ sufficiently small, $p \alpha < 6$. By Sobolev embedding and Strichartz estimates \eqref{equ:interp} with $r = \frac{15}{4}$ and $\theta_1 = 14/15$ for the $U_{N_2}, U_{N_3}$ and by Strichartz estimates with $r = 4 -  \varepsilon_1$ and $\theta_2 = \frac{4 - 2 \varepsilon_1 }{ 4 -  \varepsilon_1}$ for $U_{N_1}$ and $V_N$ we obtain
\[
\eqref{equ:duality_2} \lesssim (N_1N)^{-(1 - \frac{1}{2} \cdot \frac{4 - 2 \varepsilon_1 }{ 4 -  \varepsilon_1}) } (N_2 N_3)^{\frac{21}{60} - \frac{3}{4}s_1 } \|U_{N_1}\|_{L^{4-\varepsilon_1}} \prod_{i=2}^3\|u^{(i)}\|^{1/4}_{H^{1}(\bT^3)}  \prod_{i=2}^3 \|c_i\|^{3/4}_{L^2} .
\]
The expression on the right-hand side of the inequality is summable for dyadic values of $N_i$ and $N$ provided $\varepsilon_1 > \theta$. By the definition of $\Sigma_\lambda$ \eqref{equ:sigma_lambda}, we obtain
\[
\eqref{equ:duality_2} \lesssim (\lambda + T)^{\frac{1}{2} +} \|u^{(1)}\|_{X^{\frac{1}{2} , \frac{1}{2}+, \delta}} \prod_{i=2}^3 \|u^{(i)}\|^{\frac{3}{4} }_{X^{s_1 , \frac{1}{2}+, \delta}}.
\]
The other cases follow analogously.
\end{proof}

Finally, the last continuity type estimate we will need demonstrates that if one of the functions in the multilinear estimates satisfies $\Pi_N u^{(i)} = 0$, that is if it is only supported on high frequencies, then one gains some additional decay in $N$.

\begin{prop}
\label{prop:fin_bddness_n}
Consider the cubic nonlinear Klein-Gordon equation \eqref{equ:cubic_nlkg}. Then there exists $ s_2 > \frac{1}{2}$ with $ s_2$ sufficiently close to $1/2$ and $\theta > 0$ such that for any $\lambda, R, T > 0$, for every $(u_0, u_1) \in \Sigma_\lambda \cap \textbf{B}_R$ and for any interval $I \subset [0,T]$ with $|I| = \delta$, the nonlinearity satisfies the bound
\begin{equation*}
\begin{split}
\|F(u) -  F(P_N u) &\|_{X^{s_2, -\frac{1}{2}+,\delta}(I \times \bT^3)} \leq C(\lambda, R,T) \, \delta^c  \, N^{-\theta} \Bigl(1 +   \|u\|^{9/4}_{X^{\frac{1}{2}, \frac{1}{2}+,\delta}(I \times \bT^3)} \Bigr).
\end{split}
\end{equation*}
where $(u, \partial_t u)$ is the global solution to the cubic nonlinear Klein-Gordon equation \eqref{equ:cubic_nlkg} with initial data $(u_0, u_1)$.
\end{prop}
\begin{proof}
Once again, we use the inequality.
\[
|F(u) -  F(P_N u)| \lesssim  | (I - P_N) u| \bigl( |u|^2 + |P_N u|^2 \bigr),
\]
Let $s_1$ be as in Proposition \ref{prop:bddness}, and we repeat those arguments but we will always estimate the high-frequency term, $(I - P_N) u$, in $X^{s_1, \frac{1}{2} + , \,\delta}$, even for the linear evolution. For instance, consider the case where all the $u^{(i)}$ are of type (I). We once again estimate the expression
\begin{align}
\label{equ:duality3}
\sum_{n = n_1 + n_2 + n_3} \int_{ \tau = \tau_1 + \tau_2 + \tau_3}  
\prod_{i=1}^3 \frac{c_i(n_i, \tau_i)}{\langle n_i \rangle^{s_1} \langle |\tau_i| -\langle n_i\rangle\rangle^{\frac{1}{2}+}}\frac{d(n, \tau)}{\langle n\rangle^{\frac{1}{2}} \langle|\tau|- \langle n\rangle\rangle^{\frac{1}{2}-} } \,d\tau
\end{align}
and with the definitions for $U_{N_i}$ and $V_N$ as in \eqref{equ:viwi}, we obtain
\begin{align*}
\eqref{equ:duality3} & \lesssim (N_1 N_2N_3)^{-s_1} N^{- \frac{1}{2}} \int U_{N_1} \,U_{N_2} \, U_{N_3}  \cdot V_N \, dx dt \\
& \lesssim (N_1 N_2N_3)^{-s_1} N^{-\frac{1}{2}} \|U_{N_1}\|_{L^{3}} \|U_{N_2}\|_{L^{6}}  \|U_{N_3}\|_{L^{6}}  \|V_N\|_{L^{3}}\\
& \lesssim N_1^{ - s_1 + \frac{1}{3} } (N_2 N_3)^{-s_1}  N^{-\frac{1}{6} } \,  \|u^{(1)} \|_{X^{s_1, \frac{1}{2}+, \delta}} \|U_{N_2}\|_{L^{6}}\, \|U_{N_3}\|_{L^{6}} 
\end{align*}
which is summable for dyadic $N$ and $N_i$. By Strichartz estimates, recalling that we set
\[
u^{(1)} = (1 - P_N) S(t) (u_0, u_1),
 \]
we obtain
\begin{align*}
\|u^{(1)} \|_{X^{s_1, \frac{1}{2}+, \delta}} &= \|(1 - P_N) S(t) (u_0, u_1) \|_{X^{s_1, \frac{1}{2} + , \delta}} \\
& \lesssim N^{-\theta} \| S(t) (u_0, u_1) \|_{X^{\frac{1}{2}, \frac{1}{2} + , \delta}} \lesssim N^{-\theta} \|(u_0, u_1)\|_{\cH^{1/2}(\bT^3)},
\end{align*}
which yields the desired estimate. The other cases follow analogously to the previous propositions, with the modification that when we take $u^{(1)} = (1 - P_N) u $ we obtain
\begin{align*}
\|u^{(1)} \|_{X^{s_1, \frac{1}{2}+, \delta}}  \lesssim N^{-\theta} \| u \|_{X^{\frac{1}{2}, \frac{1}{2} + , \delta}},
\end{align*}
which yields the result.
\end{proof}

\section{Probabilistic approximation of the flow of the NLKG}
\label{sec:decomp}
This section is devoted to the proof of the approximation of the flow map for the cubic nonlinear Klein-Gordon equation by the flow of the nonlinear Klein-Gordon equation with truncated nonlinearity. We will use here the probabilistic boundedness estimates from Section \ref{sec:compactness}.

\begin{prop}
\label{prop:approx}
Let $\Phi$ denote the flow of the cubic nonlinear Klein-Gordon equation \eqref{equ:cubic_nlkg}, and $\Phi_N$ the flow of the cubic nonlinear Klein-Gordon equation with truncated nonlinearity \eqref{equ:cubic_nlkg_trun}. Fix $R, T, \lambda > 0$. Then for every $(u_0, u_1) \in \Sigma_{\lambda} \cap \textbf{B}_R $,
\[
\sup_{t \in [0,T]} \| \Phi(t)(u_0, u_1) -  \Phi_N(t) (u_0, u_1)\|_{\cH_x^{1/2}(\bT^3)} \leq C(\lambda, T, R) \, \varepsilon_1(N)
\]
with $\varepsilon_1(N) \to 0 $ as $N \to \infty$. 
\end{prop}

\begin{proof}
Fix $R, T, \lambda > 0$ and let $\Sigma_{\lambda}$ be as defined in \eqref{equ:sigma_lambda}. We need to estimate the difference $\Phi - \Phi_N$ for initial data $(u_0, u_1) \in \Sigma_{\lambda} \cap \textbf{B}_R$. Fix such a $(u_0, u_1)$ and let $u(t)$ and $u_N(t)$ denote the corresponding solutions to the full and truncated equations, respectively. By the choice of $\Sigma_\lambda$,
\begin{equation}
\label{equ:global_lambda bounds}
\begin{split}
\|(1 - \Delta)^{s_1 / 2} S(t) (u_0, u_1) \|_{L^6([0,T]\,;\, L^6 (\bT^3))} &<  C_T \lambda\\
  \| (w, \partial_t w) \|_{ L^\infty([0,T]; \cH^{1}(\bT^3))} &< C \,(\lambda + T)^{1+ }, \\
 \| (w_N, \partial_t w_N) \|_{ L^\infty([0,T]; \cH^{1}(\bT^3))} &< C \,(\lambda + T)^{1+ } ,
\end{split}
\end{equation}
where, as usual, $w(t)$ and $w_N(t)$ are the nonlinear components of the global solutions $u(t)$ and $u_N(t)$, respectively. Note that for any subinterval $I \subset [0,T]$, these bounds hold uniformly. Furthermore, for $|I| = \delta$, Proposition \ref{prop:bddness}, and the inhomogeneous estimate for the Klein-Gordon equation yield
\begin{align*}
\| w \|_{ X^{s_2, \frac{1}{2}+ ,\delta} } 
& \lesssim \|F(u) \|_{ X^{s_2 -1, - \frac{1}{2}+ ,\delta} } \leq C(\lambda,R,T) \,\delta^c\, \bigl( 1 +   \|u\|^{9/4}_{X^{s_1, \frac{1}{2}+, \delta}}  \bigr),
\end{align*}
hence if $\inf I = t_0$, by Lemma \ref{lem:free_solns} we obtain
\begin{align*}
\| u \|_{ X^{s_1, \frac{1}{2}+, \delta} } & \leq \|S(t)(u_0, u_1)\|_{X^{s_1, \frac{1}{2}+, \delta}} +  \|w \|_{X^{s_2, \frac{1}{2}+ ,\delta} } \\
& \lesssim \|(u(t_0), \partial_tu(t_0)) \|_{\cH^{s_1}} +  \|F(u) \|_{ X^{s_2 -1, - \frac{1}{2}+ ,\delta} } .
\end{align*}
By the uniform bounds \eqref{equ:global_lambda bounds} and the boundedness of the free evolution on $\cH^s$,
\[
\sup_{t_0 \in [0,T]} \|(u(t_0), \partial_tu(t_0) ) \|_{\cH^{s_1}} \leq C(\lambda, R, T),
\]
hence
\[
\| u \|_{ X^{s_1, \frac{1}{2}+, \delta} }  \leq C(\lambda, R, T) +  C(\lambda, R, T) \,\delta^c\,  \bigl( 1 + \|u\|^{9/4}_{X^{s_1, \frac{1}{2}+, \delta}}  \bigr).
\]
Thus, taking $\delta \equiv \delta(\lambda, R, T) > 0$ sufficiently small we obtain that
\begin{align}
\label{equ:unif_soln_bds}
\|\Phi(t)(u_0, u_1) \|_{X^{s_1, \frac{1}{2}+, \delta} (I \times \bT^3)} 
\leq C(\lambda, T, R).
\end{align}
By Proposition \ref{prop:trunc_bddness}, this argument yields the same result for $\Phi_N(t)$ with the same choice of $\delta > 0$. Note, too, that we can obtain the same bounds with $s_1$ replaced by $\frac{1}{2}$. Now, define 
\[
(\phi, f_t) : = (u - u_N, \partial_t u - \partial_t u_N)
\]
then we have
\[
\partial_{tt} \phi -  \Delta \phi + \phi = \Bigl( F(u)  - F(u_N) + F(u_N) -  F(P_N u_N) + F(P_N u_N) - P_N F (P_N u_N)   \Bigr).
\]
Set
\begin{align*}
\Delta_1 &:=  F(u)  - F(u_N),  \qquad \Delta_2 := F(u_N) - F(P_N u_N), \qquad \Delta_3 := (I - P_N) F(P_N u_N),
\end{align*}
and fix an interval $I \subset [0,T]$ with $|I| = \delta$. We will estimate
\begin{align*}
 \|\Delta_1 \|_{X^{-\frac{1}{2}, - \frac{1}{2} +, \delta}(I \times \bT^3)} + \|\Delta_2 \|_{X^{-\frac{1}{2}, - \frac{1}{2} +, \delta}(I \times \bT^3)} + \|\Delta_3 \|_{X^{-\frac{1}{2}, - \frac{1}{2} +, \delta}(I \times \bT^3)}.
\end{align*}
By Proposition \ref{prop:fin_bddness}, we can bound $\Delta_1$ by
\begin{align*}
&\|\Delta_1  \|_{X^{-\frac{1}{2}, - \frac{1}{2} +, \delta}}  \leq 
C( \lambda, R, T) \,\delta^c \,  \| u - u_N \|_{X^{\frac{1}{2}, \frac{1}{2} +, \delta}}  \left(1  + \|u\|^{6/4}_{X^{s_1, \frac{1}{2}+, \delta}} +  \|u_N\|^{6/4}_{X^{s_1, -\frac{1}{2}+,\delta}}\right).
\end{align*}
For the second term, Proposition \ref{prop:fin_bddness_n} and Remark \ref{rmk:trunc_bds} yields
\begin{align*}
 \|\Delta_2  \|_{X^{s_2 -1, - \frac{1}{2} +, \delta}} 
 & \leq C(\lambda, R, T)\, \delta^c \,  N^{-\theta}  \Bigl(1   + \|u_N\|^{9/4}_{X^{\frac{1}{2}, \frac{1}{2}+, \delta}}  \Bigr) .
\end{align*}
Finally, by Remark \ref{rmk:trunc_bds} we have
\begin{align*}
 \| \Delta_3 \|_{X^{-\frac{1}{2}, - \frac{1}{2} +, \delta}} & \leq N^{-\theta}\, \| F(P_N u_N)  \|_{X^{s_2 -1, - \frac{1}{2} +, \delta}} \\
 & \leq C(\lambda, R, T) \,\delta^c  \,N^{-\theta} \Bigl(1 +  \|u_N\|^{9/4}_{X^{s_1, \frac{1}{2}+, \delta}}  \Bigr).
\end{align*}
In the second and third terms, we used the observation that for any $N \in \bN$, $s \in \bR$, we have
\[
\|P_N v\|_{X^{s, \frac{1}{2}+}} \leq \|v\|_{X^{s, \frac{1}{2}+}}
\]
for any $v$ such that the right-hand side is finite. Now let $I  = [0,\delta]$. Then because $u- u_N$ has zero initial data, the inhomogeneous estimate yields
\begin{align}
\label{equ:duhamel_exp}
\| \phi \|_{X^{\frac{1}{2}, \frac{1}{2}+, \delta} } \lesssim  \|\Delta_1 \|_{X^{-\frac{1}{2}, - \frac{1}{2} +, \delta}} + \|\Delta_2 \|_{X^{-\frac{1}{2}, - \frac{1}{2} +, \delta}} + \|\Delta_3 \|_{X^{-\frac{1}{2}, - \frac{1}{2} +, \delta}},
\end{align}
and together with \eqref{equ:unif_soln_bds} and the similar result for $\Phi_N$, we can bound \eqref{equ:duhamel_exp} by
\begin{align}
\label{equ:phi_bd}
\| \phi \|_{X^{\frac{1}{2}, \frac{1}{2}+, \delta} } \leq C(\lambda, T, R) \, \delta^c \, \|\phi\|_{X^{\frac{1}{2}, \frac{1}{2}+, \delta} } + C(\lambda, T, R)\, \delta^c \, N^{-\theta},
\end{align}
and similarly for the time derivative component.  Hence, for $\delta>0$ sufficiently small, 
\begin{align}
\label{equ:phi_bd2}
\|(\phi, \partial_t \phi) \|_{L_t^\infty \cH_x^{1/2}(I \times \bT^3)} \leq C_1(\lambda, R, T) \, \varepsilon_1(N),
\end{align}
with $\lim_{N \to \infty} \varepsilon_1(N) = 0$. On the next subinterval, $I_2 = [\delta, 2\delta]$ we bound
\[
\|(\phi, \partial_t \phi) \|_{L_t^\infty \cH_x^{1/2}( I_2 \times \bT^3)} \lesssim C_1(\lambda, R, T) \, \varepsilon_1(N) + \left\| \int_\delta^t S(t - s) \left[ F(u)- F_N (u_N) \right] \right\|_{L_t^\infty \cH_x^{1/2}}
\]
and once again by the inhomogeneous estimate
\[
\| \phi \|_{X^{\frac{1}{2}, \frac{1}{2}+, \delta} (I_2 \times \bT^3)} \lesssim C_1(\lambda, R, T) \, \varepsilon_1(N)+  \|\left[ F(u)- F_N (u_N)\right] \|_{X^{-\frac{1}{2}, -\frac{1}{2} + , \delta} (I_2 \times \bT^3)}.
\]
Applying the above argument,  we obtain that
\begin{align}
\label{equ:phi_bd3}
\|(\phi, \partial_t \phi) \|_{L_t^\infty \cH_x^{1/2}( I_2 \times \bT^3)} \lesssim C_2 (\lambda, R, T) \, \varepsilon_1(N).
\end{align}
At each stage the coefficient of the nonlinear component is independent of the step number, the constants in \eqref{equ:phi_bd2} are independent of the subinterval and the bounds \eqref{equ:unif_soln_bds} are uniform, hence we can choose $\delta > 0$ sufficiently small to obtain the analogue of \eqref{equ:phi_bd3} at each stage uniformly for all subintervals. While the bound that we obtain grows with each iteration since the constant for the initial data is compounded, the number of steps is controlled by $\lambda, R$ and $T$, which yields the desired result.
\end{proof}

\begin{rmk}
When the estimate is performed on the time derivative, the time localization may increase the left-hand side of \eqref{equ:phi_bd} up to a factor of $\delta^{\frac{1}{2} - b}$ for $b = \frac{1}{2}+$ as above. However, we recall that in Section \ref{sec:compactness}, we the exponent $c$ which we obtain on $\delta$ is some fixed, small constant independent of $b$. By taking $b$ sufficiently close to $\frac{1}{2}$, we still obtain the necessary $\delta$ factor provided we ensure that
\[
\frac{1}{2} - b + c > 0.
\]
\end{rmk}

\begin{rmk}
\label{rmk:no_sobolev}
The term $\Delta_2$ is a key reason why this argument will not work with probabilistic energy estimates alone, as in \cite{BT4}, say. Indeed, this term requires us to bound the nonlinearity by a weaker norm and it does not seem possible to close the Gronwall argument if one needs to derive a bound with respect to some norm below $\cH^1$. Similarly, one encounters similar problems if one attemps to carry out the arguments in the standard Strichartz spaces.
\end{rmk}

\section{Approximation results for the flow}
\label{sec:conv}
The goal of this section is to prove the approximation results presented in the introduction. 

\subsection{Proof of Theorem \ref{thm:trunc_approx}} 
One key component in our argument is the critical stability theory which allows us to upgrade the sets of large measure where the approximation holds to open sets, or at least to general initial data in $\Pi_{2N} \textbf{B}_R$ as is required for Theorem \ref{thm:trunc_approx}. Stability arguments first appeared in the context of the three-dimensional energy critical nonlinear Schr\"odinger equation in \cite{CKSTT08}, see also \cite{TV}. The corresponding results for the nonlinear Klein-Gordon equation follow in a similar manner from the Strichartz estimates of Proposition \ref{prop:strichartz2}, and we present the proofs in Appendix \ref{ap:pert}. One crucial component is that we obtain stability results which are uniform in the truncation parameter for the nonlinear Klein-Gordon equation with truncated nonlinearity. 

\medskip
We will no recall the statement of our approximation theorem.

\begin{customthm}{1.5}
Let $\Phi$ denote the flow of the cubic nonlinear Klein-Gordon equation \eqref{equ:cubic_nlkg} and $\Phi_N$ the flow of \eqref{equ:cubic_nlkg_trun}. Fix $R > 0$, $u_* \in \cH^{1/2}$ and $N', N \in \bN$ with $N'$ sufficiently large, depending on $u_*$. Then there exists a constant $\varepsilon_0(u_*, R) > 0$ such that for all $0 < \varepsilon \leq \varepsilon_0$, there exists $\sigma \equiv \sigma( R, \varepsilon, N')$ such that for any $(u_0, u_1) \in \textbf{B}_R(u_*)$,
\begin{align}
 \sup_{ t \in [0,\sigma]} \| \Phi(t) \Pi_{N'} (u_0, u_1) -  \Phi_N(t) \Pi_{N'} (u_0, u_1) \|_{\cH_x^{1/2}} \leq C(R, u_*) \left[ \, \varepsilon_1(N) + \varepsilon \,\right] \qquad
\end{align}
with $\lim_{N \to \infty} \varepsilon_1(N) = 0 $.
\end{customthm}

\begin{proof}[Proof of Theorem \ref{thm:trunc_approx}]
Fix $R > 0$, $I = [0, \sigma]$ for some $0 < \sigma \leq  1$ to be fixed, and let $\lambda > 0$ be sufficiently large so that we can find $(v_0, v_1)\in \Sigma_{\lambda} \cap \textbf{B}_R(u_*)$. The intersection is non-empty for some $\lambda > 0$ by density and for all $N' \in \bN$ it holds that $P_{N'}(v_0,v_1) \in \Sigma_{\lambda}$ since \eqref{equ:sigma_lambda} and \eqref{equ:en_def} are invariant under smooth projections. Thus, there exists some constant $K_1 > 0$ such that the corresponding global solutions $v$ and $v_N$ to equations \eqref{equ:cubic_nlkg} and \eqref{equ:cubic_nlkg_trun} with initial data $P_{N'} (v_0 ,v_1)$ satisfy
 \begin{align*}
\|v\|_{L^4(I \times \bT^3)} \leq K_1 \quad \textup{and} \quad \|v_N \|_{L^4(I \times \bT^3)}  \leq K_1.
\end{align*}
Let $(u_0, u_1) \in \textbf{B}_R$, and let $N'$ be sufficiently large so that
\[
\|\Pi_{ \geq N'} \, u_* \|_{\cH^{1/2}} < R.
\]
 
Then
 \begin{equation}
\label{equ:trivial_stab}
\begin{split}
 &\|S(t)(P_{N'} (v_0, v_1) - \Pi_{N'}(u_0, u_1) )\|_{L^4_{t,x}(I \times \bT^3)} \\
 &\leq |I|^{1/4} \sup_{t \in I}  \|S(t) (P_{N'} (v_0, v_1) - \Pi_{N'}(u_0, u_1) ) \|_{L^4_{x}(\bT^3)} \\
 &\leq |I|^{1/4} \, (N')^{1/2} \, \Bigl[ \|(P_{N'} (v_0, v_1) - P_{N'} u_* ) \|_{H_x^{1/2}(\bT^3)}  + \|(P_{N'} u_* - \Pi_{N'}(u_0, u_1) ) \|_{H_x^{1/2}(\bT^3)} \Bigr]\\
 & \lesssim |I|^{1/4}\, (N')^{1/2} \, R.
\end{split}
 \end{equation}
Let $\rho_1 = \rho_1(K_1)$ be as in the stability lemma and recall that we can choose $\rho_1$ uniformly for all $\sigma \leq 1$. Let $0 < \varepsilon_0 < \rho_1$, then setting 
 
\[
\sigma \simeq (N')^{-2} R^{-4} \, \varepsilon_0,
\] 
the smallness condition \eqref{equ:cond1} of Lemma \ref{lem:pert_long} is met and we conclude that for $t \in I$, solutions 
\begin{align*}
u(t) & := \Phi(t) \Pi_{N'}(u_0, u_1) \\
u_N(t) & := \Phi_N(t) \Pi_{N'}(v_0, v_1) 
\end{align*}
exist to equations \eqref{equ:cubic_nlkg} and \eqref{equ:cubic_nlkg_trun}, respectively. Moreover, we conclude from \eqref{equ:diff_bds22} that
\begin{align*}
\|F(u) - F(v) \|_{L^{4/3}(I \times \bT^3)} & \leq C(K_1) \, \varepsilon_0.
\end{align*}
Hence, by Duhamel's formula and Strichartz estimates, the nonlinear components  $\widetilde{\Phi}$ of the solutions, which was defined in \eqref{equ:nonlin_com}, satisfy
\begin{align}
\label{equ:conv_bd}
\sup_{t \in I} \| \widetilde{\Phi}(t)\Pi_{N'}(v_0, v_1) - \widetilde{\Phi}(t)\Pi_{N'} (u_0, u_1) \|_{ \cH_x^{1/2}(\bT^3)} &\lesssim  \| F(u) - F(v) \|_{L_{t,x}^{4/3 } (I \times \bT^3)}  \leq C(K_1) \, \varepsilon_0,
\end{align}
and similarly for $\Phi_N$. We can estimate \eqref{equ:open_conv} using the triangle inequality by
\begin{align}
\label{equ:triangle}
\begin{split}
&\sup_{t \in I_0} \| \Phi(t)\Pi_{N'} (u_0, u_1) - \Phi_N(t)\Pi_{N'} (u_0, u_1) \|_{ \cH_x^{1/2}(\bT^3)} \\
& =\sup_{t \in I_0} \| \widetilde{\Phi}(t)\Pi_{N'}(u_0, u_1) - \widetilde{\Phi}_N(t)\Pi_{N'} (u_0, u_1) \|_{ \cH_x^{1/2}(\bT^3)} \\
&\leq \sup_{t \in I_0} \| \widetilde{\Phi}(t)P_{N'} (v_0, v_1) - \widetilde{\Phi}(t)\Pi_{N'} (u_0, u_1)  \|_{ \cH_x^{1/2}(\bT^3)} \\
& \hspace{12mm} +  \sup_{t \in I_0} \| \widetilde{\Phi}(t)P_{N'}(v_0, v_1) - \widetilde{\Phi}_N(t)P_{N'} (v_0, v_1)\|_{ \cH_x^{1/2}(\bT^3)} \\
& \hspace{24mm}+ \sup_{t \in I_0} \| \widetilde{\Phi}_N(t) P_{N'} (v_0, v_1) - \widetilde{\Phi}_N(t)\Pi_{N'}  (u_0, u_1) \|_{ \cH_x^{1/2}(\bT^3)} 
\end{split}
\end{align}
and hence we obtain that for all $(u_0, u_1) \in\,  \textbf{B}_R(u_*)$,
\begin{align*}
 \sup_{t \in [0,\sigma]}&  \| \Phi(t)\Pi_{N'} (u_0, u_1) - \Phi_N(t)\Pi_{N'} (u_0, u_1) \|_{ \cH_x^{1/2}(\bT^3)} <  C(R, u_*) \left[ \, \varepsilon_1(N) + \varepsilon_0 \,\right]. \qedhere
\end{align*}
\end{proof}

\subsection{Proof of Theorem \ref{thm:weaknon-squeezing}} 
\label{sec:non-squeeze_prob}
We define the $\rho$-fattening of $\Sigma_{\lambda}$ by
\begin{align}
\label{equ:sigma_lambda_rho}
\Sigma_{\lambda, \,\rho} := \bigcup_{u \in \Sigma_\lambda} \textbf{B}_\rho(u).
\end{align}
A first step will be to use the critical local theory from Appendix \ref{ap:pert} to conclude that there exists some $\rho \equiv \rho(\lambda, T)$ such that we obtain the same convergence result with uniform bounds on this open set. 

\medskip
We turn to the proof Theorem \ref{thm:weaknon-squeezing}. We consider initial data $(u_0, u_1) \in \Sigma_{\lambda, \rho} \cap \textbf{B}_R$ for some sufficiently small $\rho > 0$ for which we can obtain uniform bounds on the corresponding solutions by the stability theory. We begin by proving the following approximation result. Using this result, we prove Theorem \ref{thm:weaknon-squeezing} similarly to Theorem \ref{thm:non-squeezing}. 
\begin{thm}
\label{thm:approx_prop}
Let $\Phi$ denote the flow of the cubic nonlinear Klein-Gordon equation \eqref{equ:cubic_nlkg} and $\Phi_N$ the flow of the cubic nonlinear Klein-Gordon equation with truncated nonlinearity \eqref{equ:cubic_nlkg_trun}. Fix $T, R > 0$ and let $(u_0, u_0) \in \Sigma_{\lambda, \rho} \cap \textbf{B}_R $, then for all $\varepsilon > 0$ and any $N' \in \bN$,
\[
\sup_{ t \in [0,T]} \|P_{N'} \left(\Phi(t)(u_0, u_1) -  \Phi_N(t) (u_0, u_1) \right)\|_{\cH_x^{1/2}} < \varepsilon
\]
for $N = N ( N', \varepsilon, \lambda, R, T)  \gg N'$ sufficiently large. 
\end{thm}

\begin{proof}
Suppose that $(u_0, u_1) \in \textbf{B}_\rho(v_0, v_1)$ for some fixed $(v_0, v_1) \in \Sigma_\lambda$. Since $P_{N_*}(u_0, u_1) \in  \textbf{B}_\rho P_{N_*}(v_0, v_1)$ for all $N_* \in \bN$, the stability theory yields solutions on $[0,T]$ given by
\begin{align}
u &:= \Phi(t) (u_0, u_1), \qquad \qquad u_N := \Phi_N(t) (u_0, u_1) \\
\widetilde{u} &:= \Phi(t) P_{N_*} (u_0, u_1), \,\,\quad \quad \widetilde{u}_N := \Phi_N(t) P_{N_*}(u_0, u_1)
\end{align}
which satisfy uniform $L_{t,x}^4([0,T] \times \bT^3)$ and $L_t^\infty \cH_x^{1/2}([0,T] \times \bT^3)$ bounds depending only on $\lambda, R, T$. We apply the triangle inequality
\begin{align}
\sup_{ t \in [0,T]} \|P_{N'} \bigl(\Phi(t)(u_0, u_1) -&  \Phi_N(t) (u_0, u_1) \bigr)\|_{\cH_x^{1/2}} \\
& \lesssim \sup_{ t \in [0,T]} \|P_{N'} \left(\Phi(t)(u_0, u_1) -  \Phi(t) P_{N_*} (u_0, u_1) \right)\|_{\cH_x^{1/2}} \\
& \hspace{14mm} \sup_{ t \in [0,T]} \|P_{N'} \left(\Phi(t)P_{N_*} (u_0, u_1) -  \Phi_N(t) P_{N_*} (u_0, u_1) \right)\|_{\cH_x^{1/2}} \\
& \hspace{24mm} \sup_{ t \in [0,T]} \|P_{N'} \left(\Phi_N(t)(u_0, u_1) -  \Phi_N(t) P_{N_*}(u_0, u_1) \right)\|_{\cH_x^{1/2}},
\end{align}
and we estimate each term separately. For the second term, we observe that $P_{N_*}(u_0, u_1)$ is smooth, hence $P_{N_*} (u_0, u_1) \in \Sigma_\lambda \cap \,\textbf{B}_R$ for $\lambda = \lambda (R, K, N_*)$ by energy conservation and Sobolev embedding. Thus Proposition \ref{prop:approx} yields the bound
\[
\sup_{ t \in [0,T]} \|P_{N'} \left(\Phi(t)P_{N_*} (u_0, u_1) -  \Phi_N(t) P_{N_*} (u_0, u_1) \right)\|_{\cH_x^{1/2}} \lesssim C(N_*, R, T) \, \varepsilon_1(N).
\]
By Theorem \ref{thm:local_compare} and Remark \ref{rmk:stab_trunc}, we can bound the first and last terms by
\[
\sup_{ t \in [0,T]} \|P_{N'} \left(\Phi(t)(u_0, u_1) -  \Phi(t) P_{N_*} (u_0, u_1) \right)\|_{\cH_x^{1/2}} \lesssim \left( \log
 \frac{N_*}{N'} \right)^{-\theta}
\]
\[
\sup_{ t \in [0,T]} \|P_{N'} \left(\Phi_N(t)(u_0, u_1) -  \Phi_N(t) P_{N_*}(u_0, u_1) \right)\|_{\cH_x^{1/2}} \lesssim   \left( \log \frac{N_*}{N'} \right)^{-\theta},
\]
where the implicit constants depend on $\lambda, R$, and $T > 0$. Thus for fixed $N' \in \bN$, choosing $N_*$ sufficiently large, and subsequently $N$ sufficiently large yields the result.
\end{proof}

Now we prove the following statement, from which we obtain Theorem \ref{thm:weaknon-squeezing} readily given the bounds on the measure of the subsets $\Sigma_{\lambda}$ from Proposition \ref{prop:lambda_props}.

\begin{thm}
Let $\Phi$ denote the flow of the cubic nonlinear Klein-Gordon equation \eqref{equ:cubic_nlkg}. Fix $T, R > 0$, $k_0 \in \bZ^3$, $z \in \bC$, and $u_* \in \cH^{1/2}(\bT^3)$ and let $\lambda > 0$ be such that $\Sigma_{\lambda} \cap \textup{\textbf{B}}_R(u_*) \neq \varnothing$. Then for all $0 < \rho  < \rho_1(\lambda, T)$ sufficiently small,
\[
\Phi(T)\bigl(\Sigma_{\lambda, \rho} \cap \textup{\textbf{B}}_R(u_*) \bigr) \not \subseteq \textup{\textbf{C}}_{r}(z; k_0)
\]
for all $r > 0$ with $\pi r^2 <  \textup{cap}(\Sigma_{\lambda, \, \rho} \cap \textup{\textbf{B}}_R(u_*))$.
\end{thm}

\begin{proof}
Fix $N' \in \bN$ with $N' > |k_0|$ and let $R_1:= \|u_*\| + R$. Then by Theorem \ref{thm:approx_prop}, we can find $N \in \bN$ sufficiently large so that for any $(u_0, u_1) \in \Sigma_{\lambda, \rho} \cap \textbf{B}_{R_1}$ we have
\begin{align}
\label{equ:lambda_conv}
\sup_{ t \in [0,T]} \|P_{N'} \left(\Phi(t)(u_0, u_1) -  \Phi_N(t) (u_0, u_1) \right)\|_{\cH_x^{1/2}} < \varepsilon.
\end{align}
Since $\Phi_N$ preserves capacities by Proposition \ref{prop:preserve_capacity_trunc}, we have the equality
\[
c\bigl(\Phi_N(t) (\Sigma_{\lambda, \rho} \cap \textbf{B}_R(u_*)) \bigr) = c\bigl(\Sigma_{\lambda, \rho} \cap \textbf{B}_R(u_*)\bigr)
\]
for all $t \in \bR$.  Thus for $z= (z_0, z_0)$ we can find some $(u_0, u_1) \in \Sigma_{\lambda, \rho} \cap \textbf{B}_R(u_*)$ such that
\[
\left( \langle k_0 \rangle | \widehat{\Phi_N(T) (u_0, u_1) }(k_0) - z_0 |^2 + \langle k_0 \rangle^{-1}| \widehat{\partial_t \Phi_N(T) (u_0, u_1) }(k_0) - z_0 |^2 \right)^{1/2} > r + \varepsilon,
\]
and since $(u_0, u_1) \in \Sigma_{\lambda, \rho} \cap \textbf{B}_{R_1}$, we conclude by the triangle inequality and \eqref{equ:lambda_conv} that
\[
\left( \langle k_0 \rangle | \widehat{\Phi(T) (u_0, u_1) }(k_0) - z_0 |^2 + \langle k_0 \rangle^{-1}| \widehat{\partial_t \Phi(T) (u_0, u_1) }(k_0) - z_0 |^2 \right)^{1/2} > r,
\]
which completes the proof.
\end{proof}

\subsection{Proof of Theorem \ref{thm:approx_conditional}}
The goal of this subsection is to prove the conditional global result and the small-data non-squeezing result. The following result demonstrates that at low frequencies, the truncated flow is a good approximated to the full equation.
The proof follows from the same arguments used to prove Theorem \ref{thm:local_compare}, which is unsurprising given that Theorem \ref{thm:local_compare} essentially yields a decoupling between low and high frequencies. In this setting, we do not rely on the probabilistic estimates from Proposition \ref{prop:approx}, however, we are only able to compare the low frequency components of the corresponding solutions. The following proposition immediately yields the large data portion of Theorem \ref{thm:approx_conditional}. 

\begin{prop}
\label{prop:convergence for truncated flow_later}
Fix $T, R > 0$ and $u_* \in \cH^{1/2}(\bT^3)$. Suppose there exists some $K > 0$ such that for all $(u_0, u_1) \in \textbf{B}_R(u_*)$, the corresponding solutions $u$  to \eqref{equ:cubic_nlkg} and $u_N$ to \eqref{equ:cubic_nlkg_trun} exist on $[0,T]$ and satisfy
\[
\|u \|_{L_{t,x}^{4}([0,T] \times \bT^3)} \,\, + \,\,  \sup_N \|P_N u_N\|_{L_{t,x}^{4}([0,T] \times \bT^3)} \leq K.
\]
Let $\Phi$ and $\Phi_N$ denote the flows of the cubic nonlinear Klein-Gordon equation with full \eqref{equ:cubic_nlkg} and truncated \eqref{equ:cubic_nlkg_trun} nonlinearities, respectively. Then for any $N' \in \bN$, and sufficiently large $N$ depending on $R, T$ and $K$,
\[
\sup_{ t \in [0,T]} \|P_{N'} \left(\Phi(t)(u_0, u_1) -  \Phi_N(t) (u_0, u_1) \right)\|_{\cH_x^{1/2}} \lesssim \left( \log\frac{N}{N'} \right)^{-\theta},
\]
with implicit constants depending on $R, T$, and $K$.
\end{prop}

\begin{proof}
Let $u$ and $u_N$ to the Cauchy problem \eqref{equ:cubic_nlkg} and the truncated equation \eqref{equ:cubic_nlkg_trun} respectively on $[0,T)$. By Proposition \ref{prop:error_eqn}, there exists some $M \in [N', N]$
\begin{align}
\square \,u_{lo} + u_{lo} = P_{ M} F(u_{lo}, u_{lo}, u_{lo}) + \cO_{K, R, T}( (\log (N / N'))^{-\theta})
\end{align}
for $u_{lo} = P_M u$. By the same reasoning,
\begin{align}
\square \,u_{N,\,lo} + u_{N,\,lo} = P_{ M} F(u_{N,\,lo}, u_{N,\,lo}, u_{N,\,lo}) + \cO_{K, R, T}( (\log (N / N'))^{-\theta})
\end{align}
with the same frequency cut-off, as in the proof of Theorem \ref{thm:local_compare}, and again with a slightly different error term. Since $u$ and $u_N$ have the same initial data, the arguments used to prove Theorem \ref{thm:local_compare} yield the result.
\end{proof}

\begin{rmk}
Note that although any initial data in $P_{N'} \textbf{B}_R$ gives rise to global solutions of the relevant Cauchy problems, Proposition \ref{prop:convergence for truncated flow_later} is insufficient to prove Theorem \ref{thm:trunc_approx} since the implicit constants would depend on the truncation parameter and thus we could not guarantee convergence uniformly.
\end{rmk}

To conclude, we present the following lemma which yields our small data result.

\begin{lem}
\label{lem:trunc_bds_small}
 Fix $T > 0$ and let $\Phi_N$ denote the flow of the cubic nonlinear Klein-Gordon equation with truncated nonlinearity \eqref{equ:cubic_nlkg_trun}. There exists some sufficiently small absolute constant $\rho_0 = \rho_0(T)$ such that for any $0 < \rho < \rho_0$ and for all $(u_0, u_1) \in \textbf{B}_\rho \subset \cH^{1/2}(\bT^3)$ there exists a unique solution $u_N := \Phi_N(u_0, u_1)$ on $[0,T]$ which satisfies
\[
\|u_N\|_{L_{t,x}^4([0,T) \times \bT^3) } \lesssim \rho.
\]
\end{lem}

\begin{proof}
Fix $T > 0$ By the small data theory, we know that there exists some $\rho_1(T) > 0$ sufficiently small so that for any $0 < \rho < \rho_1$, and for all $(u_0, u_1) \in \textbf{B}_\rho$, a unique solution $u := \Phi(u_0, u_1)$ exists to the cubic nonlinear Klein-Gordon equation \eqref{equ:cubic_nlkg}, which satisfies
\[
\|u \|_{L_{t,x}^4([0,T) \times \bT^3) } \lesssim \rho.
\]
Let $F(u) = u^3$, then we can expand
\begin{align}
 &F(u) - P_N F(P_N u_N) \\
 &= F(u) -  P_N F(u)  + P_N F(u) - P_N F(P_N u) + P_N F(P_Nu) - P_N F(P_Nu_N),
\end{align}
hence by the boundedness of the smooth projections and Strichartz estimates we obtain
\begin{align}
\|u - u_N \|_{L_{t,x}^4} 
& \lesssim \|u\|^3_{L_{t,x}^{4}} +  \| F( P_Nu) - F( P_Nu_N)\|_{L_{t,x}^{4/3}} \\
& \lesssim \|u\|^3_{L_{t,x}^{4}} +  \| u - u_N\|^3_{L_{t,x}^{4}} + \|u\|^2_{L_{t,x}^{4}} \| u - u_N\|_{L_{t,x}^{4}}  ,
\end{align}
where the implicit constants may depend on time. By taking $\rho_0 = \rho_0(T) > 0$ smaller if necessary, we obtain the desired result by a standard continuity argument. 
\end{proof}

\appendix

\section{Adapted function spaces}
\label{sec:harmonic}

\subsection{Adapted Function spaces}
\subsubsection{$X^{s,b}$ spaces}

For a good overview of these spaces, see Chapter 2.6 in \cite{Taononlin}. For completeness, we recall the definition of $X^{s,b}(\bR \times \bT^3)$ spaces, with norm
\begin{align}
\label{equ:xsb_def}
\|u\|_{X^{s,b}(\bR \times \bT^3)} = \|\langle n \rangle^s \langle |\tau| - \langle n \rangle \rangle^b \widehat u(n,\tau) \|_{L^2_\tau \ell_n^2}.
\end{align}
We will also work with the local-in-time restriction spaces $X^{s,b,\delta}$, which are defined by the norm
\begin{align}
\|u\|_{X^{s,b,\delta}} = \inf \bigl\{\|\widetilde{u}\|_{X^{s,b}(\bR \times \bT^3)} : \widetilde{u} |_{[-\delta, \delta]} = u \bigr\}.
\end{align}
We have the obvious inclusions
\[
X^{s',b'} \subseteq X^{s, b}
\]
for $s \leq s'$ and $b \leq b'$. We remark that these spaces are not invariant under conjugation or modulation but they are invariant under translation. 

\medskip
Heuristically, these spaces measure how far a given function is from being a free solution. Additionally, free solutions lie in $X^{s,b}$ provided we localize in time. 
\begin{lem}
\label{lem:free_solns}
Let $f \in H^s$ for $s \in \bR$ and let $S(t)$ denote the free evolution for the Klein-Gordon equation. Then for any Schwartz time cutoff $\eta \in S_x(\bR)$, 
\[
\|\eta(t) S(t) f \|_{X^{s,b}(\bR \times \bT^3)} + \|\eta(t) \partial_t S(t) f \|_{X^{s-1,b}(\bR \times \bT^3)} \leq c(\eta, b) \|f\|_{\cH^s(\bT^3)}.
\]
\end{lem}
We note that it does not hold that free solutions lie in $X^{s,b}$ globally hence these spaces are really only suitable for local theory. An important property of these spaces is the so-called transfer principle which allows one to convert bounds for free solutions into bounds for $X^{s,b}$ functions
\begin{lem}[Lemma 2.9, \cite{Taononlin}]
Let $\Lambda = \bR$ or $\bT$ and let $L = i P(\nabla /i)$ for some polynomial $P: \Lambda^d \to \bR$, and let $s \in \bR$ and let $Y$ be a Banach space of functions on $\bR \times \Lambda^d$ such that
\[
\|e^{i \tau_0} e^{tL} f \|_Y \lesssim \|f\|_{H_\cX^s(\Lambda^d)}
\]
for all $f \in H_\cX^s(\Lambda^d)$ and $\tau_0 \in \bR$. Then for $b > \frac{1}{2}$
\[
\|u\|_Y \lesssim_b \|u\|_{X^{s, b}(\bR \times \Lambda^d)}.
\]
\end{lem}

Letting $Y$ be the Strichartz spaces from Proposition \ref{prop:strichartz2}, we obtain the following corollary immediately.

\begin{cor}
Let $(q,r)$ be Strichartz admissible pairs and let $b > \frac{1}{2}$. Then
\[
\|u\|_{L_t^qL_x^r} \lesssim \|u\|_{X^{0,b}}.
\]
\end{cor}
For $b > \frac{1}{2}$, $X^{s,b}$ embeds into $C_t H^s_x$ (for both the full and restricted spaces on the appropriate domains). This embedding fails at the endpoint $b= \frac{1}{2}$ and should be thought of analogously to the failure of the endpoint Sobolev embedding $L^\infty \not \subseteq H^{\frac{n}{2}}$. It is precisely at the endpoint $b= \frac{1}{2}$ that these spaces respect the scaling of $C_t H_x^s$ and for critical problems where scale invariance is an issue, one no longer has the appropriate control in order to close the contraction mapping argument. It is possible to remedy this problem by including a Besov space type refinement, however we will focus instead on $U^p$ and $V^p$ spaces.

\subsubsection{$U^p$ and $V^p$ Function spaces}
\label{sec:up_vp}

In this appendix, we introduce the basic facts we will need about the $U^p$ and $V^p$ spaces. We follow the exposition in \cite{HHK}. Consider partitions given by a strictly increasing finite sequence $-\infty < t_0 < t_2 < \ldots t_{K} \leq \infty$. If $t_K = \infty$ we use the convention $v(t_{K}) := 0$ for all functions $v: \bR \to H$. We will usually be working on bounded intervals $I \subset \bR$. A step functions associated to a partition is a function which is constant on each open sub-interval of the partition. In the sequel, we let $\cB$ denote an arbitrary Banach space.

\begin{defn}[$U^p$ spaces]
\label{defn:up}
Let $1 \leq p < \infty$. Consider a partition $\{t_0, \ldots, t_K\}$ and let $(\varphi_k)_{k=0}^{K-1} \subseteq \cB$ with $\sum_{ k =0}^{K-1} \|\varphi_k\|_{L^2}^p = 1$. We define a $U^p$ atom to be a function
\[
a = \sum_{k=1}^{K} \mathbbm{1}_{[t_{k-1}, t_{k})} \varphi_{k-1}
\]
and we define the atomic space $U^p(\bR, \cB)$ to be the set of all functions $u: \bR \to \cB$ such that
\[
 u = \sum_{j=1}^\infty \lambda_j a_j,
\]
for $a_j$ $U^p$ atoms, and $\{\lambda_j\} \in \ell^1(\bC)$, endowed with the norm
\[
\|u\|_{U^p} := \inf \left\{ \sum_{j=1}^\infty |\lambda_j|, \,u = \sum_{j=1}^\infty \lambda_j a_j \,: \, a_j \textup{ is a } U^p\textup{ atom} \right\} .
\]
\end{defn}
\begin{rmk}
This yields a Banach space which satisfies the embeddings
\[
U^p(\bR, \cB) \hookrightarrow U^q(\bR , \cB) \hookrightarrow L^\infty(\bR, \cB)
\]
for $1 \leq p < q < \infty$. Furthermore, every $u \in U^p$ is right-continuous and $\lim_{t \to -\infty} u(t) = 0$.

\end{rmk}

\begin{defn}[$V^p$ spaces]
\label{defn:vp}
Let $1 \leq p < \infty$. We define $V^p(\bR, \cB)$ as the space of all $\cB$ valued functions, $v$, such that the norm 
\[
\|v\|_{V^p(\bR, \cB)} = \sup_{\textup{partitions }} \left( \sum_{i=1}^K \|v(t_{i}) - v(t_{i-1})\|^p_{\cB} \right)^{1/p}< \infty
\]
with the convention $v(\infty) = 0$. We let $V_- (\bR , \cB)$ denote the subspace of all functions satisfying $\lim_{t \to - \infty} v(t) = 0$ and we let $V_{rc}^p(\bR , \cB) $ denote the subspace of all right continuous functions in $V_- (\bR , \cB)$, endowing both these subspaces with the above norm.
\end{defn}

\begin{rmk}
\label{rmk:embedding}
Note that for $1 \leq p < \infty$ we have the embeddings
\[
U^p(\bR , \cB) \hookrightarrow V_{rc}^p(\bR , \cB) \hookrightarrow  L^\infty(\bR, \cB) \qquad \textup{and} \qquad V^p(\bR , \cB) \hookrightarrow V^q(\bR , \cB).
\]
If further $1 < p < q < \infty$, then
\begin{align}
\label{equ:v_in_u}
V_{rc}^p(\bR , \cB) \hookrightarrow  U^q(\bR , \cB).
\end{align}
\end{rmk}

A crucial property of the $U^p$ and $V^p$ spaces is the following duality relation.

\begin{thm}
\label{thm:duality}
Let $1 < p < \infty$ and $\frac{1}{p} + \frac{1}{p'} = 1$. Then
\[
(U^p(\bR, \cB))^* = V^{p'}(\bR, \cB^*),
\]
that is, there is a bounded bilinear form 
\[
T: V^{p'}(\bR, \cB^*) \to (U^p(\bR,\cB))^*, \qquad T(v) := B( \cdot, v)
\]
which is an isometric isomorphism.
\end{thm}

\begin{prop}[Proposition 2.9, \cite{HHK}]
\label{prop:duality}
For $1 < p < \infty$, let $u \in U^p$ be continuous and $v, v* \in V^{p'}$ for $\frac{1}{p} + \frac{1}{p'} = 1$ such that $v(s) = v^*(s)$ except for at most countably many points. Then
\[
B(u,v) = B(u, v^*).
\]
\end{prop}

\begin{prop}
\label{prop:duality2}
Let $1 < p < \infty$, $u \in V_-^1(\cB)$ be absolutely continuous on compact intervals and $v \in V^{p'}(\cB^*)$ for $\frac{1}{p} + \frac{1}{p'} = 1$. Then,
\[
B(u,v) = - \int_{-\infty}^\infty \langle v(t), u'(t) \rangle_{\cB^*,\, \cB} \,dt.
\]
In particular $B(u,v) = B(u, \widetilde{v})$ if $v(t) = \widetilde{v}(t)$ almost everywhere. Consequently, $v$ can be replaced by its right-continuous version.
\end{prop}

\begin{proof}
See \cite{HHK}
\end{proof}

\begin{rmk}[Remark 2.11 in \cite{HHK}]
\label{rmk:duality_improv}
Let $1 < p < \infty$ and $u \in U^p$. Then for $\frac{1}{p} + \frac{1}{p'} = 1$ one clearly has
\[
\|u\|_{U^p} = \sup_{v \in V^{p'} \,:\, \|v\|_{V^{p'} = 1}} \bigl|B(u,v)\bigr|.
\]
However, in light of Proposition \ref{prop:duality}, one can restrict to taking a supremum over right-continuous functions, which we will do in the estimates in Section \ref{sec:multi}. The bilinear form above should be thought of as the corresponding the Stieltjes integral
\[
\int f dg = \sum_{i=1}^n f(t_i) \bigl( g(t_{i+1}) - g(t_i) \bigr).
\]
\end{rmk}

We define variants of the $U^p$ and $V^p$ spaces adapted to the linear propagator for the Klein-Gordon equation $e^{\pm it \langle \nabla \rangle}$. In the sequel we will suppress the notational dependence on $I \subset \bR $ and the Banach space $\cB$.
\begin{defn}
Defnote by $U^p_\pm$ the space $U^p$ adapted to the linear propagator equipped with the norm
\[
\|u\|_{U_\pm^p} = \|e^{\pm it \langle \nabla \rangle} u \|_{U^p}
\]
and similarly for the $V^p$ spaces.
\end{defn}

\begin{rmk}
The space $U_\pm^p$ is again an atomic space with atoms $\widetilde{a} = e^{it \langle \nabla \rangle} a$ for a $U^p$ atom, a. Further, in the case that $\cB = H^s$, we compare the above definition to that of the $X^{s,b}$ spaces associated to the Klein-Gordon equation given by $\|u\|_{X^{s,b}} = \|e^{-it \langle \nabla \rangle} u\|_{H_t^b H_\cX^s}$.
\end{rmk}

\medskip
A useful feature of the $U^p$ and $V^p$ spaces is that they satisfy a transfer principle, namely one can transfer multilinear estimates for free solutions to estimates for $U_\pm^2$ functions.

\begin{prop}[Transfer principle, Proposition 2.19 \cite{HHK}]
\label{prop:transfer}
Let $T_0 : L_x^2 \times \ldots \times L_x^2 \to L^1_{loc}$ be an $m$-linear operator. Suppose that for some $1 \leq q, r \leq \infty$
\[
\|T_0(S(t) \phi_1, \ldots, S(t) \phi_m) \|_{L_{t}^q L_x^r} \lesssim \prod_{i=1}^m \|\phi_i \|_{L_x^2}
\]
then there exists an extension $T_0: U_\pm^q \times \ldots \times U_\pm^q \to L_t^q L_x^r$ with
\[
\|T(u_1, \ldots u_m) \|_{L_t^q L_x^r} \lesssim \prod_{i=1}^m \|u_i \|_{U_\pm^q}
\]
which agrees with $T_0$ for almost every $t \in \mathbb{R}$.
\end{prop}

\begin{rmk}
\label{rmk:log_loss}
The transfer principle will be the key tool which allows us to use the Strichartz estimates available for free solutions of the Klein-Gordon equation to derive bounds in $U^p$ and $V^p$ spaces. Because of the duality relation, typically one needs to put one function in $V^p$ when performing multilinear estimates. In these cases, the following proposition demonstrates that this is possible if one allows for a logarithmic loss.
\end{rmk}

\begin{prop}
\label{prop:log_loss}
Let $q_1, \ldots q_m > 2$ where $m=1,2$ or $3$ and let $E$ be a Banach space and $T: U^{q_1} \times \ldots U^{q_m} \to E$ be a bounded $m$-linear operator with
\[
\|T(u_1, \ldots u_m) \|_{E} \leq C \prod_{i=1}^m \|u_i\|_{U_\pm^{q_i}}.
\]
Suppose further theses exists some $0 < C_2 \leq C$ such that the estimate
\[
\|T(u_1, \ldots u_m) \|_{E} \leq C_2 \prod_{i=1}^m \|u_i\|_{U_\pm^{2}}
\]
holds. Then $T$ satisfies
\[
\|T(u_1, \ldots u_m) \|_{E} \leq C_2 \left( \log \frac{C}{C_2} + 1 \right)^m \prod_{i=1}^m \|u_i\|_{V^{2}}, \qquad u_i \in V_{rc, \,\pm}^p.
\]
\end{prop}

\section{Standard stability arguments}
\label{ap:pert}

This appendix is devoted the proofs of the critical stability lemmata for the cubic nonlinear Klein-Gordon equation. As usual, these statements are proved first for solutions which have sufficiently small Strichartz bounds, we will call these short-time stability arguments. In order to conclude the statement for arbitrarily large bounds, one needs to divide a given time interval into subintervals such that the norm of the solution is sufficiently small on each subinterval, then the statement follows from an iteration argument. We include these proofs as we would like to make explicit the dependence of the constants on the various parameters involved. We also point out that the dependence on the time interval in the following estimates arises solely due to localizing Strichartz estimates and thus can be taken uniformly for all $I \subset [0,1]$. In the sequel we let $F(u) := u^3$ and $F_N(u_N) := P_N (P_N u)^3$.

\subsection{Stability theory for NLKG}
\begin{lem}[Short-time stability]
\label{lem:pert_short}
Let $I \subset \bR$ a compact time interval and $t_0 \in I$. Let $v$ be a solution defined on $I \times \bT^3$ of the Cauchy problem
\begin{equation*}
\left\{ 
\begin{split}
&v_{tt} - \Delta v + v + F(v) = e \\
&\bigl(v, \partial_tv \bigr)\big|_{t=t_0} = (v_0, v_1) \in  \cH^{1/2}(\bT^3).
\end{split}
\right.
\end{equation*}
Let $(u, \partial_t u) \big|_{t=t_0} = (u_0, u_1) \in \cH^{1/2}(\bT^3)$ be such that
\[
\|(v_0 - u_0, v_1 - u_1)\|_{ \cH^{1/2}(\bT^3)} \leq R_1
\]
for some $R_1 > 0$. Suppose also that we have the smallness conditions
\begin{align}
\|v\|_{L^4(I \times \bT^3)} &\leq \rho_0 \\
\|S(t-t_0)(v_0 - u_0, v_1 - u_1) \|_{L^4(I \times \bT^3)} &\leq \rho \label{equ:l4}\\
\|e\|_{L_t^{\tilde{q}'} L_x^{\tilde{r}'}(I \times \bT^3)} &\leq \rho,
\end{align}
for some $0 < \rho < \rho_0(R_1)$ a small constant and $(\tilde{q}', \tilde{r}')$ a conjugate admissible pair. Then there exists a unique solution $(u(t), \partial_t u(t))$ to the cubic nonlinear Klein-Gordon equation on $I \times \bT^3$ with initial data $(u_0, u_1)$ at time $t_0$ and $C \equiv C(I) \geq 1$ which satisfies
\begin{align}
 \|v - u\|_{L^4(I \times \bT^3)} &\leq C \,\rho \label{equ:diff_bds0}\\
 \|F(v) - F(u)\|_{L^{4/3}(I \times \bT^3)} &\leq C \,\rho \label{equ:diff_bds11}  \\
\| (v- u, \partial_t u - \partial_tv) \|_{L_t^\infty  \cH_x^{1/2} (I \times \bT^3)} +  \|v - u\|_{L_{t,x}^{q,r}(I \times \bT^3)} &\leq C R_1 \label{equ:diff_bds1} 
\end{align}
for all admissible pairs $(q,r)$. Furthermore, the dependence of the constant $C$ on time arises only from the constant in the localized Strichartz estimates.
\end{lem}

\begin{rmk}
By Strichartz estimates, assumption \eqref{equ:l4} is redundant if $R_1 = O(\rho)$.
\end{rmk}

\begin{proof}
Without loss of generality, let $t_0 = \inf I$ and let $\phi := v- u$, then $\phi$ satisfies the Cauchy problem
\begin{equation*}
\label{equ:nlw}
\left\{ 
\begin{split}
&\phi_{tt} - \Delta \phi + \phi + F(v) - F(\phi + v) = e \\
&\bigl(\phi, \partial_t\phi \bigr)\big|_{t=t_0} = (v_0 - u_0, v_1 - u_1) .
\end{split}
\right.
\end{equation*}
By Strichartz estimates, H\"older's inequality and the assumptions above,
\begin{align*}
 \|\phi\|_{L^4(I\times \bT^3)} & \leq C \Bigl(\|S(t)(v_0 - u_0, v_1 - u_1)\|_{L^4(I \times \bT^3)} + \| F(\phi+v) - F(v)\|_{L^{4/3}_{t,x}(I \times \bT^3)} + \|e\|_{L_t^{\tilde{q}'} L_x^{\tilde{r}'}(I \times \bT^3)} \Bigr)\\
 & \leq  C \Bigl( 2 \rho +  (\rho_0)^2 \, \|\phi\|_{L^4(I\times \bT^3)} +  \|\phi\|^3_{L^4(I\times \bT^3)} \Bigr)
\end{align*}
hence a continuity argument yields \eqref{equ:diff_bds0} provided $\rho_0$ is sufficiently small. 
Consequently
\begin{align*}
 \| F(\phi+v) - F(v)\|_{L^{4/3}_{t,x}(I \times \bT^3)} \leq C \rho
\end{align*}
for such $\rho_0$. From \eqref{equ:diff_bds0} we have
\begin{align*}
\|(\phi,& \partial_t \phi)\|_{L_t^\infty  \cH_x^{1/2} (I \times \bT^3)} + \|\phi\|_{L_{t,x}^{q,r}(I \times \bT^3)} +\|\phi\|_{L^4(I\times \bT^3)} \\
 & \leq  C \Bigl(R_1 + \| F(\phi+v) - F(v)\|_{L^{4/3}_{t,x}(I \times \bT^3)} + \| e\|_{L_t^{\tilde{q}'} L_x^{\tilde{r}'}(I \times \bT^3)}  \Bigr)  \\
& \leq  C \Bigl(R_1 +  \rho + C \rho  \Bigr),
\end{align*}
hence we obtain \eqref{equ:diff_bds1} provided $\rho_0 \equiv \rho_0( R_1)$ is chosen sufficiently small.
\end{proof}

\begin{lem}[Long-time stability]
\label{lem:pert_long}
Let $I \subset \bR$ a compact time interval and $t_0 \in I$. Let $v$ be a solution defined on $I \times \bT^3$ of the Cauchy problem
\begin{equation*}
\left\{ 
\begin{split}
&v_{tt} - \Delta v + v + F(v) = e \\
&\bigl(v, \partial_tv \bigr)\big|_{t=t_0} = (v_0, v_1) \in  \cH^{1/2}(\bT^3).
\end{split}
\right.
\end{equation*}
Suppose that
\[
\|v\|_{L^4(I \times \bT^3)} \leq L
\]
for some constant $L > 0$. Let $t_0 \in I$ and let $(u, \partial_t u) \big|_{t=t_0} = (u_0, u_1) \in \cH^{1/2}(\bT^3)$ be such that
\[
\|(v_0 - u_0, v_1 - u_1)\|_{ \cH^{1/2}(\bT^3)} \leq R_1
\]
for some $R_1 > 0$. Suppose also that we have the smallness conditions
\begin{equation}
\label{equ:cond1} 
\begin{split}
\|S(t-t_0)(v_0 - u_0, v_1 - u_1) \|_{L^4(I \times \bT^3)} &\leq \rho  \\
\|e\|_{L_t^{\tilde{q}'} L_x^{\tilde{r}'}(I \times \bT^3)} &\leq \rho,
\end{split}
\end{equation}
for some $0 < \rho < \rho_1(R_1 , I, L)$ a small constant and $(\tilde{q}', \tilde{r}')$ a conjugate admissible pair. Then there exists a unique solution $(u, \partial_t u)$ to the cubic nonlinear Klein-Gordon equation on $I \times \bT^3$ with initial data $(u_0, u_1)$ at time $t_0$ and $C \equiv C( I, L) \geq 1$ which satisfies
\begin{equation}
\label{equ:diff_bds22} 
\begin{split}
 \|v - u\|_{L^4(I \times \bT^3)} & \leq C \,\rho \\ 
  \|F(v) - F(u)\|_{L^{4/3}(I \times \bT^3)} &\leq C \,\rho \\
\| (v- u, \partial_t u - \partial_tv) \|_{L_t^\infty  \cH_x^{1/2} (I \times \bT^3)} + \|v - u\|_{L_{t,x}^{q,r}(I \times \bT^3)}  &\leq  C R_1 
\end{split}
\end{equation}
admissible pairs $(q,r)$. Moreover, the dependence of the constants $C$ and $\rho_1$ on time arise solely from the constant in the time localized Strichartz inequality.
\end{lem}

\begin{proof}
Fix $I \subset \bR$ and let $\rho_0 \equiv \rho_0(2 R_1) > 0$ be as in Lemma \ref{lem:pert_short}. This will allow for some growth in the argument. We divide the time interval $I$ into $J \sim \left( 1 + \frac{L}{\rho_0} \right)^4$ subintervals $I_j = [t_j, t_{j+1}]$ such that $\|v\|_{L^4(I_j)} \leq \rho_0$ letting $\phi := u - v$, we can apply the previous lemma on the first interval, yielding
 
\begin{align*}
\|v - u\|_{L^4(I_0 \times \bT^3)} &\leq C \rho \\
\|F(v) - F(u)\|_{L^{4/3}(I_0 \times \bT^3)} &\leq C \rho \\
\|(v - u, \partial_t v - \partial_t u) \|_{ L^\infty_t \cH_x^{1/2}(I_0 \times \bT^3)} + \|v - u\|_{L_{t,x}^{q,r}(I_0 \times \bT^3)} &\leq CR_1.
\end{align*}
We would like to apply this argument iteratively to claim that
 
\begin{align*}
\|v - u\|_{L^4(I_j \times \bT^3)} &\leq  C(j) \rho \\
\|F(v) - F(u)\|_{L^{4/3}(I_j \times \bT^3)} &\leq C(j) \rho \\
\|(v - u, \partial_t v - \partial_t u) \|_{ L^\infty_t \cH_x^{1/2}(I_j \times \bT^3)} + \|v - u\|_{L_{t,x}^{q,r}(I_j \times \bT^3)} &\leq C(j) R_1.
\end{align*}
In order to do this, we need to ensure that for each $t_j$ we have 
\begin{align}
\label{equ:guarantee1}
\|(v(t_j) - u(t_j), \partial_t v(t_j) - \partial_t u(t_j))\|_{ \cH^{1/2}(\bT^3)} &\leq 2 R_1  \\
\label{equ:guarantee2}
\|S(t - t_j)(u(t_j) - v(t_j) ) \|_{L^4(I_j \times \bT^3)} &< C(I, L) \, \rho
\end{align}
We prove these statements by induction. For \eqref{equ:guarantee1} we use Strichartz estimates and we bound
 
\begin{align*}
&\|(\phi(t_j) , \partial_t \phi(t_j))\|_{\cH_x^{1/2} (\bT^3)}  \\
&\leq \|(\phi(t_0), \partial_t \phi(t_0))\|_{ \cH_x^{1/2} (\bT^3)} +\| F(\phi+v) - F(v)\|_{L^{4/3}_{t,x}([0,t_{j-1}] \times \bT^3)}  +  \|e\|_{L_t^{\tilde{q}'} L_x^{\tilde{r}'}([0,t_{j-1}] \times \bT^3)}\\ 
& \leq R_1 + \sum_{k = 0}^{j-1} C(k) \rho  +  \rho 
\end{align*}
and similarly for \eqref{equ:guarantee2}, we have
\begin{align}
&\|S(t - t_j)(u(t_j) - v(t_j)) \|_{L^4(I_j \times \bT^3)} \\
 & \lesssim \|S(t - t_j)(u(t_0) - v(t_0)) \|_{L^4([t_0, t_j) \times \bT^3)}  + \|F(u) - F(v)\|_{L_t^{\tilde{q}'} L_x^{\tilde{r}'}([t_0, t_j) \times \bT^3)} +\|e\|_{L_t^{\tilde{q}'} L_x^{\tilde{r}'}([t_0, t_j) \times \bT^3)} \\
 & \lesssim  2 \rho + \sum_{k = 0}^{j-1} C(k) \rho ,
\end{align}
so the conclusion follows by choosing $\rho_1(R_1, I,L)$ sufficiently small.
\end{proof}

\subsection{Stability theory for the truncated NLKG}

\begin{lem}[Short-time stability for the truncated equation]
\label{lem:pert_short_trunc}
Let $I \subset \bR$ a compact time interval and $t_0 \in I$. Let $v_N$ be a solution defined on $I \times \bT^3$ of the Cauchy problem
\begin{equation*}
\left\{ 
\begin{split}
&(v_N)_{tt} - \Delta v_N + v_N + P_N(P_Nv_N)^3 = e \\
&\bigl(v_N, \partial_tv_N \bigr)\big|_{t=t_0}  = (v_0, v_1) \in  \cH^{1/2}(\bT^3).
\end{split}
\right.
\end{equation*}
Let $t_0 \in I$ and let $(u_N, \partial_t u_N) \big|_{t=t_0} = (u_0, u_1) \in \cH^{1/2}(\bT^3)$ be such that
\[
\|(v_0 - u_0, v_1 - u_1)\|_{ \cH^{1/2}(\bT^3)} \leq R_1
\]
for some $R_1 > 0$. Suppose also that we have the smallness conditions
\begin{align*}
\|P_N \, v_N\|_{L^4(I \times \bT^3)} &\leq \rho_0 \\
\|S(t-t_0)(v_0 - u_0, v_1 - u_1) \|_{L^4(I \times \bT^3)} &\leq \rho \\
\|e\|_{L_t^{\tilde{q}'} L_x^{\tilde{r}'}(I \times \bT^3)} &\leq \rho,
\end{align*}
for some $0 < \rho < \rho_0(R_1)$ a small constant and $(\tilde{q}', \tilde{r}')$ a conjugate admissible pair. Then there exists a unique solution $(u_N, \partial_t u_N)$ to the truncated cubic nonlinear Klein-Gordon equation on $I \times \bT^3$ with initial data $(u_0, u_1)$ at time $t_0$ and $C \equiv C(I) \geq 1$ which satisfies
\begin{align}
 \|v_N - u_N\|_{L^4(I \times \bT^3)} &\leq C \,\rho \label{equ:diff_bds4}\\
  \|F_N(v_N) - F_N(u_N)\|_{L^{4/3}(I \times \bT^3)} &\leq C \,\rho \label{equ:diff_bds44}\\
\| (v_N- u_N, \partial_t u_N - \partial_tv_N) \|_{L_t^\infty  \cH_x^{1/2} (I \times \bT^3)} + \|v_N - u_N\|_{L_{t,x}^{q,r}(I \times \bT^3)} &\leq C\, R_1  \label{equ:diff_bds5}
\end{align}
for any admissible pair $(q,r)$. Furthermore, the dependence of $C$ on time arises only from the constant in the localized Strichartz estimates.
\end{lem}
\begin{proof}
Without loss of generality, let $t_0 = \inf I$. Let $\phi_N = v_N - u_N$, then $\phi_N$ satisfies the Cauchy problem
\begin{equation*}
\label{equ:nlw_trunc_diff}
\left\{ 
\begin{split}
&(\phi_N)_{tt} - \Delta \phi_N + \phi_N + F_N(v_N) - F_N(\phi_N + v_N) = e \\
&\bigl(\phi_N, \partial_t \phi_N \bigr) \big|_{t= 0 }= (v_0 - u_0, v_1 - u_1) .
\end{split}
\right.
\end{equation*}
By Strichartz estimates, H\"older's inequality, the boundedness of $P_N$ and the assumptions above,
\begin{align*}
 &\|\phi_N\|_{L^4(I\times \bT^3)} \\
 & \leq C \Bigl(\|S(t)(v_0 - u_0, v_1 - u_1)\|_{L^4(I \times \bT^3)} + \|F_N(v_N) - F_N(\phi_N + v_N)\|_{L^{4/3}_{t,x}(I \times \bT^3)} + \| e\|_{L_t^{\tilde{q}'} L_x^{\tilde{r}'}(I \times \bT^3)} \Bigr) \\
 & \leq  C \Bigl( 2 \rho +  (\rho_0)^2 \, \|\phi_N\|_{L^4(I\times \bT^3)} +  \|\phi_N\|^3_{L^4(I\times \bT^3)} \Bigr),
\end{align*}
hence a continuity argument yields \eqref{equ:diff_bds4} provided $\rho_0$ is sufficiently small. Similarly
\begin{align*}
 \|(&\phi_N, \partial_t \phi_N)\|_{L_t^\infty  \cH_x^{1/2} (I \times \bT^3)}  + \|\phi_N\|_{L_t^qL_x^r(I \times \bT^3))} + \|\phi_N\|_{L^4(I \times \bT^3)} \\
& \leq C \Bigl( \|(v_0 - u_0, v_1 - u_1)\|_{ \cH^{1/2}(\bT^3)}  +  \|F_N(v_N) - F_N(\phi_N + v_N)\|_{L^{4/3}_{t,x}(I \times \bT^3)} + \| e\|_{L_t^{\tilde{q}'} L_x^{\tilde{r}'}(I \times \bT^3)} \Bigr)\\
& \leq C \Bigl( R_1  +  (\rho_0)^2 \, \|\phi_N\|_{L^4(I\times \bT^3)} + \|\phi_N\|^3_{L^4(I\times \bT^3)} + \rho\, \Bigr).
\end{align*}
We conclude \eqref{equ:diff_bds5} by a continuity argument for $\rho_0 \equiv \rho_0(R_1) > 0$ sufficiently small.
\end{proof}

\begin{rmk}
\label{rmk:low_freq_bds}
If one only requires bounds on the low frequency component $P_N(u_N - v_N)$, then from the proof of the previous Lemma, it is clear that it suffices to assume that
\begin{align}
\|S(t-t_0)P_N (v_0 - u_0, v_1 - u_1) \|_{L^4(I \times \bT^3)} &\leq \rho \\
\|P_N e\|_{L_t^{\tilde{q}'} L_x^{\tilde{r}'}(I \times \bT^3)} &\leq \rho.
\end{align}
\end{rmk}

Lemma \ref{lem:pert_short_trunc} and the same proof as in Lemma \ref{lem:pert_long} yields the long-time stability argument for the truncated equation with bounds uniform in the truncation parameter.

\begin{lem}[Long-time stability]
\label{lem:pert_long_trunc}
Let $I \subset \bR$ a compact time interval and $t_0 \in I$. Let $v_N$ be a solution defined on $I \times \bT^3$ of the Cauchy problem
\begin{equation*}
\left\{ 
\begin{split}
&(v_N)_{tt} - \Delta v_N + v_N + F_N(v_N) = e \\
&\bigl(v_N, \partial_tv_N \bigr)\big|_{t=t_0}  = (v_0, v_1) \in  \cH^{1/2}(\bT^3).
\end{split}
\right.
\end{equation*}
Suppose that
\[
\|P_N \,v_N \|_{L^4(I \times \bT^3)} \leq L
\]
for some constant $L > 0$. Let $(u_N, \partial_t u_N) \big|_{t=t_0} = (u_0, u_1) \in \cH^{1/2}(\bT^3)$ be such that
\[
\|(v_0 - u_0, v_1 - u_1)\|_{ \cH^{1/2}(\bT^3)} \leq R_1
\]
for some $R_1 > 0$. Suppose also that we have the smallness conditions
\begin{align*}
\|S(t-t_0)(v_0 - u_0, v_1 - u_1) \|_{L^4(I \times \bT^3)} &\leq \rho \\
\|e\|_{L_t^{\tilde{q}'} L_x^{\tilde{r}'}(I \times \bT^3)} &\leq \rho,
\end{align*}
for some $0 < \rho < \rho_1(R_1 , I, L)$ a small constant and any $(\tilde{q}', \tilde{r}')$ a conjugate admissible pair. Then there exists a unique solution $(u_N, \partial_t u_N)$ to the truncated cubic nonlinear Klein-Gordon equation on $I \times \bT^3$ with initial data $(u_0, u_1)$ at time $t_0$ and $C \equiv C(I) \geq 1$ which satisfies
\begin{align*}
 \|v_N - u_N\|_{L^4(I \times \bT^3)} &\leq C \, \rho \\
 \|(v_N)^3 - (u_N)^3\|_{L^{4/3}(I \times \bT^3)} &\leq C \,\rho \\
\| (v_N- u_N, \partial_t u_N - \partial_tv_N) \|_{L_t^\infty  \cH_x^{1/2} (I \times \bT^3)} + \|v_N - u_N\|_{L_{t,x}^{q,r}(I \times \bT^3)} &\leq C\, R_1  .
\end{align*}
\end{lem}

\bibliographystyle{myamsplain}
\bibliography{refs}

\end{document}